\documentclass[12pt, a4paper, reqno]{amsart}

\pdfoutput=1

\usepackage{microtype}

\usepackage{geometry}
\makeatletter
\def\@tocline#1#2#3#4#5#6#7{\relax
  \ifnum #1>\c@tocdepth %
  \else
    \par \addpenalty\@secpenalty\addvspace{#2}%
    \begingroup \hyphenpenalty\@M
    \@ifempty{#4}{%
      \@tempdima\csname r@tocindent\number#1\endcsname\relax
    }{%
      \@tempdima#4\relax
    }%
    \parindent\z@ \leftskip#3\relax \advance\leftskip\@tempdima\relax
    \rightskip\@pnumwidth plus4em \parfillskip-\@pnumwidth
    #5\leavevmode\hskip-\@tempdima
      \ifcase #1
       \or\or \hskip 1em \or \hskip 2em \else \hskip 3em \fi%
      #6\nobreak\relax
    \dotfill\hbox to\@pnumwidth{\@tocpagenum{#7}}\par
    \nobreak
    \endgroup
  \fi}
\makeatother

\usepackage{suffix}
\usepackage{xifthen}

\DeclareRobustCommand{\gobblefive}[5]{}
\newcommand*{\SkipTocEntry}{\addtocontents{toc}{\gobblefive}}

\makeatletter %
\let\old@subsection\subsection
\WithSuffix\def\subsection*#1{
    \SkipTocEntry\old@subsection*{#1}
}
\makeatother

\makeatletter
\def\bigsubsection{\@startsection{subsection}{2}%
  \z@{.5\linespacing\@plus.7\linespacing}
{.5\baselineskip}%
  {\normalfont\centering\scshape}%
}
\makeatother

\newcommand{\unicodemu}{\unichar{"03BC}}
\newcommand{\unicodenu}{\unichar{"03BD}}

\usepackage[misc]{ifsym}

\usepackage{amssymb}

\usepackage{bbm}

\usepackage{mathtools}

\usepackage{enumitem}

\usepackage{stackengine}
\usepackage{stackrel}

\usepackage{bigints}

\usepackage{nicefrac}

\usepackage{xfrac}

\allowdisplaybreaks[4]

\usepackage[numbers,sort]{natbib}

\usepackage{multirow}

\usepackage{relsize}

\usepackage{float}

\usepackage{caption}

\usepackage{subfigure}

\usepackage{centernot}

\makeatletter
\newcommand{\leqnomode}{\tagsleft@true\let\veqno\@@leqno}
\newcommand{\reqnomode}{\tagsleft@false\let\veqno\@@eqno}
\makeatother

\usepackage{xcolor}

\usepackage[unicode]{hyperref}
\hypersetup{
    colorlinks=true,
    linktoc=page,
    linkcolor=black,
    urlcolor=black,%
    citecolor=blue!55!black,
}

\usepackage{cleveref}

\usepackage{amsthm}

\makeatletter
\def\th@plain{%
  \thm@notefont{}%
  \itshape %
}
\def\th@definition{%
  \thm@notefont{}%
  \normalfont %
}
\makeatother

\theoremstyle{plain}
\newtheorem{theorem}{Theorem}[section]
\newtheorem{corollary}[theorem]{Corollary}
\newtheorem{proposition}[theorem]{Proposition}
\newtheorem{lemma}[theorem]{Lemma}
\newtheorem*{theorem*}{Theorem}
\newtheorem*{corollary*}{Corollary}
\newtheorem*{proposition*}{Proposition}
\newtheorem*{lemma*}{Lemma}
\newtheorem{introtheorem}{Theorem}

\newtheorem{introcorollary}{Corollary}

\theoremstyle{definition}
\newtheorem{definition}[theorem]{Definition}
\newtheorem{remark}[theorem]{Remark}
\newtheorem{example}[theorem]{Example}
\newtheorem*{definition*}{Definition}
\newtheorem*{remark*}{Remark}
\newtheorem*{example*}{Example}

\AtBeginEnvironment{remark}{%
  \pushQED{\qed}%
}
\AtEndEnvironment{remark}{\popQED\endremark}
\AtBeginEnvironment{example}{%
  \pushQED{\qed}%
}
\AtEndEnvironment{example}{\popQED\endexample}

\theoremstyle{plain}

\newtheorem{manualtheoreminner}{Theorem}
\newenvironment{manualtheorem}[1]{%
  \renewcommand\themanualtheoreminner{#1}%
  \manualtheoreminner
}{\endmanualtheoreminner}

\newtheorem{manualcorollaryinner}{Corollary}
\newenvironment{manualcorollary}[1]{%
  \renewcommand\themanualcorollaryinner{#1}%
  \manualcorollaryinner
}{\endmanualcorollaryinner}

\usepackage{xpatch}
\makeatletter
\AtBeginDocument{
    \xpatchcmd{\cref@thmoptarg}{\thm@headpunct{.}}{\thm@headpunct{}}{}{}
}
\AtBeginDocument{
    \xpatchcmd{\cref@thmnoarg}{\thm@headpunct{.}}{\thm@headpunct{}}{}{}
}
\makeatother

\makeatletter
\def\namedlabel#1#2{\begingroup
    #2%
    \def\@currentlabel{#2}%
    \phantomsection\label{#1}\endgroup
}
\makeatother

\hypersetup{bookmarksdepth=3}

\newcommand{\defeq}{\vcentcolon=}

\newcommand{\supp}[1]{\operatorname{supp}({#1})}
\newcommand{\diam}[1]{\operatorname{diam}({#1})}
\newcommand{\conv}[1]{\operatorname{conv}({#1})}

\newcommand{\altll}{< \mspace{-6mu} <}
\newcommand{\altgg}{> \mspace{-6mu} >}

\usepackage{xifthen}
\usepackage{etoolbox}
\usepackage{xparse}

\newcommand{\Koperator}[1][]{%
    \ifthenelse{\equal{#1}{}}
        {\mathrm{I}}
        {    \ifthenelse{\equal{#1}{*}}
            {\mathrm{I}_*}
            {\mathrm{I}{\left[{#1}\right]}}
        }
}

\newcommand{\Moperator}[1][]{%
    \ifthenelse{\equal{#1}{}}
        {\mathrm{C}}
        {\mathrm{C}[{#1}]}
}

\newcommand{\semidiscrete}[1]{
    {#1^{\mspace{2mu}s}}
}

\newcommand{\continuous}[1]{
    {#1^{\mspace{2mu}c}}
}

\DeclareMathOperator*{\esssup}{ess\,sup}

\newcommand{\newunderset}[2]{\: \underset{
    \text{\raisebox{1.2ex}{\smash{\scalebox{0.9}{$#1$}}}}%
}{
    \text{\raisebox{0.2ex}{\smash{$#2$}}}
} \:}

\ifdefined\longrightharpoonup
    \renewcommand{\longrightharpoonup}{%
        \mathrel{-}\joinrel\rightharpoonup%
    }
\else
    \newcommand{\longrightharpoonup}{%
        \mathrel{-}\joinrel\rightharpoonup%
    }
\fi

\newcommand{\converges}[1][]{%
    \ifthenelse{\isempty{#1}}%
        {\rightarrow}%
        {\smash{\newunderset{{#1}}{\rightarrow}}}%
}

\newcommand{\longconverges}[1][]{%
    \ifthenelse{\isempty{#1}}%
        {\longrightarrow}%
        {\smash{\newunderset{{#1}}{\longrightarrow}}}%
}

\newcommand{\weaklyconverges}[1][]{%
    \ifthenelse{\isempty{#1}}%
        {\rightharpoonup}%
        {\smash{\newunderset{{#1}}{\rightharpoonup}}}%
}

\newcommand{\longweaklyconverges}[1][]{%
    \ifthenelse{\isempty{#1}}%
        {\longrightharpoonup}%
        {\smash{\newunderset{{#1}}{\longrightharpoonup}}}%
}

\newcommand{\customQ}{\mathchoice
  {\mbox{\larger[-0.8]$Q$}}
  {\mbox{\larger[-0.8]$Q$}}
  {\mbox{\larger[-2]$Q$}}
  {\mbox{\larger[-2.8]$Q$}}
}
\newcommand{\Qfact}{{\customQ}}

\newcommand{\minsetkantorovich}{\mathcal{O}_{\operatorname{KP}}}
\newcommand{\minsetmonge}{\mathcal{O}_{\operatorname{MP}}}

\begin{document}

\title[Discrete Approximation of OT on Compact Spaces]{%
    Discrete Approximation of Optimal Transport \break on Compact Spaces
}
\newcommand{\myname}{Maximiliano Frungillo}
\author{\myname}
\address{\begin{minipage}{0.85\textwidth}\vspace{3mm}
\myname \hfill%
{\Letter \hspace{1mm} \tt mfrungillo@dm.uba.ar}%
\break\indent%
Departamento de Matem{\'a}tica, FCEyN, Universidad de Buenos Aires,%
\break\indent%
(1428) Buenos Aires, Argentina.%
\end{minipage}}

\begin{abstract}

We investigate the approximation of Monge--Kantorovich problems
on general compact metric spaces,
showing that optimal values, plans and maps can be effectively
approximated via a fully discrete method.
First we approximate optimal values and plans by solving finite dimensional
discretizations of the corresponding Kantorovich problem.
Then we approximate optimal maps by
means of the usual barycentric projection or by
an analogous procedure available in general spaces without a linear structure.
We prove the convergence of all these approximants in full generality and show
that our convergence results are sharp.

\end{abstract}

\maketitle

\vspace{-8mm}

{\normalfont\Small
\begin{enumerate}[%
    wide, labelsep=6pt, leftmargin=36pt, rightmargin=36pt, itemindent=0pt%
]
\item[\hskip\labelsep\scshape Keywords:]
optimal transport, Monge--Kantorovich problem, compact spaces,
discrete approximation, convergence in $W^p$, convergence in $L^p$.
\end{enumerate}
}

\vspace{4mm}

\centerline{ {MSC 2020\,:}
65J99, %
65K99, %
49Q22, %
90C08  %
}

\tableofcontents

\parskip5pt

\section{Introduction}
\label{section_introduction}

In this paper we study the discrete approximation of
Optimal Transport (OT)
\cite{
villani_topics,
villani_old_new,
ambrosio_gigli_savare,
santambrogio_2015_OT_for_app,
ambrosio_brue_semola_lectures_OT,
figalli_glaudo_invitation_OT}
on general compact metric spaces.
Simply put, the OT problem $(X,Y,\mu,\nu,c)$ consists in
transporting the probability $\mu \in \mathcal{P}(X)$ into $\nu\in \mathcal{P}(Y)$
while minimizing the total cost induced by the cost function $c(x,y)$.
This transportation of measure is done by a map $T:X\to Y$ in Monge's version
of OT and by a measure $\pi\in\mathcal{P}(X\times Y)$ in Kantorovich's version,
so the total cost operators are
\begin{equation}
\label{eq_intro_def_total_cost_operators_monge_kantorovich}
\Moperator[T] = \int_{X}{c(x,T(x)) \; d\mu(x)}
\;\;\text{ and }\;\;
\Koperator[\pi] = \int_{X\times Y}{c(x,y) \; d\pi(x,y)}
\end{equation}
respectively.
Our main objective is to show that, when they exist, minimizers $T_*$ and $\pi_*$
can be successfully approximated via discrete methods under minimal (sharp) assumptions
on $(X,Y,\mu,\nu,c)$.
In order to state our main results and discuss how they fit within the current
literature we start by summarizing the basic theory of OT and introducing
our discretization scheme.

\subsection*{Briefing on OT}

An \emph{OT problem} is a quintuple $(X,Y,\mu,\nu,c)$ where
$(X,d_X)$ and $(Y,d_Y)$ are Polish metric spaces,
$\mu\in\mathcal{P}(X)$ and $\nu\in\mathcal{P}(Y)$ are Borel probability measures
and $c:X\times Y\to \mathbb{R}$ is a Borel cost function.
In this context, a \emph{transport map} is a Borel map $T\in\mathcal{B}(X,Y)$
with $T_\#\mu = \nu$, i.e.\
\[
(T_\#\mu)[F] \defeq \mu[T^{-1}(F)] = \nu[F]
\;\;\text{for all Borel sets $F\in\mathcal{B}(Y)$,}
\]
while a \emph{transport plan}
is a measure $\pi\in\mathcal{P}(X\times Y)$ with marginals $\mu$ and $\nu$, i.e.\
\[
\pi[E\times Y] = \mu[E]
\;\text{ and }\;
\pi[X\times F] = \nu[F]
\quad
\text{for all $E\in\mathcal{B}(X)$, $F\in\mathcal{B}(Y)$.}
\]
The set $\mathcal{T}(\mu, \nu)$ of transport maps injects into the set
$\mathcal{A}(\mu, \nu)$ of transport plans
since every map $T\in\mathcal{T}(\mu,\nu)$ induces a transport plan
$\pi_T \defeq (\operatorname{Id}_X \times \mspace{5mu}T )_\#\mu$,
where $(\operatorname{Id}_X \times \mspace{5mu}T)(x)=(x,T(x))$ for $x\in X$.
In general $\mu\otimes\nu \in \mathcal{A}(\mu, \nu)\neq \emptyset$ but
$\mathcal{T}(\mu, \nu)$ may be empty, which is prevented e.g.\
by requiring non-atomic $\mu$
\cite{gangbo_monge_mass_transfer_problem}.

To $(X,Y,\mu,\nu,c)$ we associate its \emph{Monge problem}
\begin{equation*}
\label{monge_problem}\tag{MP}
\min{
    \left\{
        \int_{X}{c(x,T(x)) \; d\mu(x)} \; : \; T \in \mathcal{T}(\mu, \nu)
    \right\}
}
=\min_{T\in\mathcal{T}(\mu,\nu)}{\Moperator[T]}
\end{equation*}
and its \emph{Kantorovich problem}
\begin{equation*}
\label{kantorovich_problem}\tag{KP}
\min{
    \left\{
        \int_{X\times Y}{c(x,y) \; d\pi(x,y)} \; : \; \pi \in \mathcal{A}(\mu, \nu)
    \right\}
}
=\min_{\pi \in \mathcal{A}(\mu,\nu)}{\Koperator[\pi]},
\end{equation*}
where $\Moperator$ and $\Koperator$ are the total cost operators
defined in \eqref{eq_intro_def_total_cost_operators_monge_kantorovich}.
Notice that \eqref{kantorovich_problem} is a relaxation of
\eqref{monge_problem} since
$\pi_T \in \mathcal{A}(\mu,\nu)$ and $\Koperator[\pi_T] = \Moperator[T]$
for every $T\in\mathcal{T}(\mu,\nu)$.

The problem \eqref{kantorovich_problem} achieves a minimum
whenever $c$ is lower semicontinuous and bounded below.
This fact follows from a standard application of the direct method
in the calculus of variations to the cost operator $\Koperator$
after endowing $\mathcal{P}(X\times Y)$ with the topology of the
\emph{weak (or narrow) convergence} of measures.
This metrizable topology is characterized on a general Polish metric space $(Z,d_Z)$
by stating that a sequence $(\sigma_k)_k \subset \mathcal{P}(Z)$
converges to $\sigma\in\mathcal{P}(Z)$, notated $\sigma_k \longweaklyconverges[k] \sigma$, iff
\begin{equation}
\label{eq_intro_weak_convergence_definition}
\int_{Z}{\phi(z) \, d\sigma_k(z)}
\longconverges[k]
\int_{Z}{\phi(z) \, d\sigma(z)}
\;\text{ for all $\phi \in C_b(Z)$.}
\end{equation}

The analysis of the well-definedness for \eqref{monge_problem}
requires additional assumptions and is much subtler than
the straightforward argument applied to \eqref{kantorovich_problem}.
The following facts illustrate what we can expect
under adequate hypotheses:
\begin{enumerate}[%
    label=\emph{\roman*)}, ref=\emph{\roman*)},%
    topsep=-2pt, itemsep=2pt, leftmargin=20pt%
]

\item The optimal value
of \eqref{kantorovich_problem} agrees with
$\inf_{T\in\mathcal{T}(\mu,\nu)}{\Moperator[T]}$
for continuous $c$ and non-atomic $\mu$
\cite{ambrosio_lecture_notes,pratelli2007equality},
while in general we only have the inequality
\begin{equation}
\label{eq_equality_kantorovich_min_monge_inf}
\min\nolimits_{\pi\in\mathcal{A}(\mu,\nu)}{\Koperator[\pi]}
\leq \inf\nolimits_{T\in\mathcal{T}(\mu,\nu)}{\Moperator[T]}.
\end{equation}

\item For nice cost functions and measures \eqref{monge_problem}
admits a minimizer $T_*$, unique up to redefinition on
a $\mu$-null set.
This holds e.g.\ on $\mathbb{R}^d$ for
$c(x,y) = |x-y|^p$ with $p>1$ when $\mu \altll \mathcal{L}^d$
\cite{caffarelli_allocation_maps_general_cost_functions,
gangbo_mccann_the_geometry_of_OT}
and also holds on a compact connected Riemannian manifold $(M,g)$
for the cost $c(x,y)=d^{\,2}(x,y)$ if $\mu \altll vol_{M}$,
where $d(x,y)$ is the geodesic distance on $M$
\cite{mccann_polar_factorization}.
In both cases the transport plan $\pi_{T_*}$ is the unique
minimizer for \eqref{kantorovich_problem} and
then equality holds in
\eqref{eq_equality_kantorovich_min_monge_inf}.
\end{enumerate}

\subsection*{Uniqueness Hypotheses}

We will reserve the notation $\Koperator[*]$ for the optimal
value of \eqref{kantorovich_problem}
and reserve $\pi_*$ and $T_*$ for the minimizers of
\eqref{kantorovich_problem} and \eqref{monge_problem} respectively.
Their existence and uniqueness will be enforced by imposing
appropriate hypotheses on $(X,Y,\mu,\nu,c)$; precisely, a uniqueness hypothesis
for \eqref{kantorovich_problem}
and a stronger uniqueness hypothesis
linking \eqref{kantorovich_problem} and \eqref{monge_problem}.
They read as follows:
\begin{enumerate}[leftmargin=40pt,labelsep=8pt,topsep=4pt,itemsep=10pt]
\item[\namedlabel{kantorovich_uniqueness_hypothesis}{{\normalfont(KU)}}]
There is a unique optimal transport plan
    $\pi_*\in\mathcal{A}(\mu,\nu)$ for \eqref{kantorovich_problem}.
\item[\namedlabel{strong_uniqueness_hypothesis}{{\normalfont(SU)}}]
There is a unique optimal transport plan
    $\pi_*\in\mathcal{A}(\mu,\nu)$ for \eqref{kantorovich_problem} and it is
    induced by a transport map $T_*\in \mathcal{T}(\mu,\nu)$,
    in the sense that $\pi_{T_*}=\pi_*$.
\end{enumerate}

It follows from \ref{strong_uniqueness_hypothesis} that the optimal values of
\eqref{kantorovich_problem} and \eqref{monge_problem} agree
and that the map $T_*$ is a minimizer for \eqref{monge_problem},
unique up to redefinition on a $\mu$-null set.
Of course, \ref{kantorovich_uniqueness_hypothesis}
follows from \ref{strong_uniqueness_hypothesis},
which in turn holds true for a variety of problems.
Among them we have problems on $\mathbb{R}^d$
with a convex or concave cost
\cite{
caffarelli_allocation_maps_general_cost_functions,
gangbo_mccann_the_geometry_of_OT,
gangbo_mccann_optimal_maps,
santambrogio_concave}
or with a general cost satisfying a twist condition
\cite{
levin_abstract_cyclical_monotonicity,
carlier_duality_existence,
ma_trudinger_wang,
de_philippis_figalli_survey},
problems on Riemannian manifolds
\cite{
mccann_polar_factorization,
bernard_buffoni_OT_mather,
villani_old_new,
figalli_fathi_non_compact,
figalli_regularity_of_OT_maps_non_compact}
and problems on geodesic metric spaces satisfying
certain curvature conditions
\cite{
bertrand_alexandrov_spaces,
gigli_optimal_maps_non_branching,
rajala_schultz_alexandrov,
cavalletti_metric_measure_spaces}.
Generally speaking,
\ref{strong_uniqueness_hypothesis}
follows when $c(x,y)$ satisfies a twist condition
and $\mu$ gives zero mass to ``small'' sets.

It is worth noting that even though
\ref{kantorovich_uniqueness_hypothesis}
is strictly weaker than
\ref{strong_uniqueness_hypothesis},
it is hard to decide nontrivial classes
of problems for which \ref{kantorovich_uniqueness_hypothesis} holds true
without forcing \ref{strong_uniqueness_hypothesis} to also hold true.
On the one hand, \ref{kantorovich_uniqueness_hypothesis} may fail
even when $(X,Y,\mu,\nu,c)$ is required to satisfy hypotheses
strong enough to imply the regularity of the
supports of optimal transport plans \cite[Remark 3.3]{mccann_warren_pass}.
On the other hand, if those hypotheses are strengthened to prevent
non-uniqueness for \eqref{kantorovich_problem} in the usual way
(i.e.\ via a twist condition)
then \ref{strong_uniqueness_hypothesis} ends up holding true.
For a positive result in this direction see \cite{mccann_et_al},
in particular Theorem 5.1 and the subsequent discussion.

\subsection*{Discretization Scheme}

We consider finite partitions for $X,Y$ by pointed Borel sets of diameters
at most $h$, which we call \emph{$h$-partitions} and generically notate
\[
\mathcal{C}_X = \big\{ (E_i, x_i)\big\}_{i\in I} \;\text{ and }\;
\mathcal{C}_Y = \big\{ (F_j, y_j) \big\}_{j\in J}
\mspace{1mu}.
\]
Given $h$-partitions $\mathcal{C}_X$, $\mathcal{C}_Y$
we define the discrete probability measures
\begin{equation}
\label{eq_intro_discrete_measures_u_v}
\mu_h = \sum\nolimits_{i\in I}{\,\mu_i \mspace{1mu} \delta_{x_i}}
\;\;\text{ and }\;\;
\nu_h = \sum\nolimits_{j\in J}{\,\nu_j \mspace{1mu} \delta_{y_j}}
\mspace{1mu},
\end{equation}
where $\mu_i = \mu[E_i]$ and $\nu_j = \nu[F_j]$.
We say that $\mu_h$ and $\nu_h$ are \emph{$h$-approximations} of $\mu$ and $\nu$
and that $(X,Y,\mu_h,\nu_h,c)$ is an \emph{$h$-approximation} of $(X,Y,\mu,\nu,c)$.

The problem \eqref{kantorovich_problem} for
$(X,Y,\mu_h,\nu_h,c)$ turns into the finite linear program
\begin{equation*}
\label{kantorovich_discrete_problem}\tag{FP}
\begin{alignedat}{2}
&\min_{(\pi_{ij})_{ij}\,\geq 0} \quad && \sum_{i \in I}{\sum_{j \in J}{\pi_{ij}\,c_{ij}}} \\[1ex]
&\mspace{17mu}\operatorname{s.t.}
    &&\sum\nolimits_{j\in J}{\pi_{ij}} = \mu_i \quad \forall i \in I,
    \;\;
    \sum\nolimits_{i\in I}{\pi_{ij}} = \nu_j \quad \forall j \in J,
\end{alignedat}
\end{equation*}
where $c_{ij} = c(x_i,y_j)$ for all $i,j$.
Here we have identified the measure
$\sum_{ij}{\pi_{ij}\mspace{2mu}\delta_{(x_i,y_j)}}$ on $X\times Y$
with the doubly indexed sequence $(\pi_{ij})_{ij}\in\mathbb{R}^{I\times J}$
and both will be generically notated $\pi_h$.
Thus, the constraints in \eqref{kantorovich_discrete_problem}
enforce the condition $\pi_h\in\mathcal{A}(\mu_h,\nu_h)$
and the operator $\Koperator$ turns into the cost functional of
\eqref{kantorovich_discrete_problem} since
\[
\sum_{i \in I}{\sum_{j \in J}{\pi_{ij}\,c_{ij}}}
= \sum_{(i,j) \in I \times J}{\pi_{ij}\,c(x_i,y_j)}
= \int_{X\times Y}{c(x,y)\,d\pi_h(x,y)}
= \Koperator[\pi_h].
\]

A feasible solution $\pi_h = (\pi_{ij})_{ij}$ for
\eqref{kantorovich_discrete_problem} is \emph{$\varepsilon$-optimal} ($\varepsilon\geq 0$)
if it is at most $\varepsilon$ apart in cost from an optimal solution, that
is, if
$\Koperator[\pi_h] \leq \varepsilon + \Koperator[\widetilde{\pi}_h]$
for any other feasible solution $\widetilde{\pi}_h$ for
\eqref{kantorovich_discrete_problem}.
Of course, $\pi_h$ is \emph{optimal} when it is $0$-optimal.
An \emph{$h$-plan} for $(X,Y,\mu,\nu,c)$ is a feasible solution
for \eqref{kantorovich_discrete_problem}, so
if $\pi_h$ is an $\varepsilon$-optimal $h$-plan we expect
$\Koperator[\pi_h]$ to be a good approximation of $\Koperator[*]$
for sufficiently small $h$ and $\varepsilon$.
In addition, if \ref{kantorovich_uniqueness_hypothesis} holds true,
we also expect $\pi_h$ to be near $\pi_*$ in some sense.

From a computational viewpoint, one advantage of discretizing
\eqref{kantorovich_problem}
as
\eqref{kantorovich_discrete_problem}
is the availability of efficient procedures for finding good feasible solutions,
such as the (network) simplex method
\cite{bertsimas1997introduction, luenberger2008linear},
a variety of combinatorial algorithms
\cite{network_flows_1993, assignment_problems_2012}
or the entropic regularization method
\cite{cuturi2013sinkhorn, peyre_cuturi_computational_OT}.
It is worth noting, however, that some care is needed to keep
\eqref{kantorovich_discrete_problem}
tractable for small $h$.
The key insight here is that, despite the number of variables in
\eqref{kantorovich_discrete_problem}
being $|I||J|$,
it is always possible to get good or even optimal solutions $\pi_h$
with $|\supp{\pi_h}| \leq |I| + |J|$.
As a consequence, a good estimation of $\supp{\pi_h}$ reduces the
(effective) number of variables in
\eqref{kantorovich_discrete_problem}
from $|I||J|$ to $O(|I|+|J|)$ and improves the
cost of computing $\pi_h$.
In essence, this is why multiscale methods
\cite{schmitzer_schnorr_2013,
schmitzer_2016,
gerber_maggioni_2017,
oberman_ruan_2020,
bartels_hertzog_2022}
for \eqref{kantorovich_discrete_problem}
perform well in practice.

\subsection*{Projection Maps}

Constructing an approximation $T_h$ for $T_*$ requires some extra effort
since in general the support of $\pi_h$ is not function-like, in the sense
that more than one $\pi_{ij}$ might be non-zero for any given $i\in I$.
When $Y$ is a subset of a linear space this obstruction is usually
overcome by considering
the \emph{barycentric projection map} $T_h:X\to Y$ of $\pi_h$ given by
\begin{equation}
\label{eq_intro_barycentric_proj_definition}
T_h(x) = y_i \defeq \sum\nolimits_{j\in J}{\mathlarger{\tfrac{\pi_{ij}}{\mu_i}}\, y_j}
\;\; \text{ for $i\in I$ and $x\in E_i\,$.}
\end{equation}
Observe that this construction is meaningless for general $(Y,d_Y)$ and it
fails to be well-defined even if $Y$ is a general
(i.e.\ non-convex) subset of a linear space.

To adapt the idea of barycentric projections to general spaces note
that each $y_i$ in
\eqref{eq_intro_barycentric_proj_definition}
is the expected value $\mathbb{E}[Y_i]$ of an $Y$-valued
random variable $Y_i$ taking each value $y_j$ with probability $\pi_{ij}/\mu_i$.
Then, it is enough to substitute $\mathbb{E}$
by some other notion of central tendency depending only on the metric structure
of $(Y,d_Y)$. We will do this by considering weighted geometric medians and
maps $T_h$ obtained this way, as well as barycentric projection maps, will be
collectively called \emph{projection maps}.

\subsection*{Main Results}

In what follows $(X,Y,\mu,\nu,c)$ will be \emph{compact}, meaning
that $X$ and $Y$ are compact and $c$ is continuous.
Then, the product metric
\[
d_{X\times Y} ((x_1,y_1),(x_2,y_2)) = \max{\{ d_X(x_1,x_2),\, d_Y(y_1,y_2)\}}
\]
makes $(X\times Y, \, d_{X \times Y})$ compact and $c$ uniformly continuous
with modulus
\[
\omega_c(h)
    = \max{\Big\{\,|c(z_1)-c(z_2)|\,:\, z_1,z_2\in X\times Y \text{ with }
        d_{X\times Y}(z_1,z_2) \leq h
    \,\Big\}}.
\]
The compactness of $X\times Y$ also implies that the topology of the
weak convergence on $\mathcal{P}(X\times Y)$,
defined by \eqref{eq_intro_weak_convergence_definition},
is metrized by the
\emph{Wasserstein distances}
\cite{villani_topics, santambrogio_2015_OT_for_app}
\begin{alignat*}{4}
&W_p(\pi_1,\pi_2) &&= && \; \left( \min_{\sigma \in \mathcal{A}(\pi_1,\pi_2)}{
    \int_{Z\times Z}{d_Z^{\,p}(z_1,z_2) \;d\sigma(z_1,z_2)}
    } \right)^{1/p}
&&\text{ for } {1\leq p < \infty},
\end{alignat*}
where $Z = X\times Y$ and $d_Z = d_{X \times Y}$.
Also, the boundedness of $Y$ together with the finiteness of $\mu$ implies that the
set of Borel maps $\mathcal{B}(X,Y)$ agrees with
\[ \textstyle
L^p(\mu,Y) =
\Big\{
T\in \mathcal{B}(X,Y) \,:\,
\bigintssss_{X}{d_Y^{\,p}(T(x),\mspace{1mu}y_0) \;d\mu(x)} < \infty
\,\text{ for all $y_0\in Y$}
\Big\}
\]
for $1\leq p < \infty$.
Then, maps in $\mathcal{B}(X,Y)$ can be compared with the pseudometrics
\begin{alignat*}{4}
&d_p(T_1,T_2) &&= && \; \left( \int_{X}{d_Y^{\,p}(T_1(x),\mspace{1mu}T_2(x)) \;d\mu(x)} \right)^{1/p}
&&\text{ for } {1\leq p < \infty},
\end{alignat*}
which turn into true metrics after identifying maps agreeing $\mu$-a.e.

Since we want to approximate $(X,Y,\mu,\nu,c)$ to arbitrary precision we
consider an iterative procedure: in each step $k$ we compute
an $\varepsilon_k$-optimal $h_k$-plan $\pi_k$ by solving an instance
of \eqref{kantorovich_discrete_problem} and then we extract a projection map
$T_k$ from $\pi_k$, taking care to choose $h_k$ and $\varepsilon_k$
so that $h_k, \varepsilon_k \longconverges[k] 0$.
Then
$(\pi_k)_k$ is an \emph{approximating sequence of discrete plans}
and $(T_k)_k$ is an \emph{approximating sequence of projection maps}.
Our first result,
\Cref*{introtheorem_conv_values_plans} below,
asserts that $(\pi_k)_k$ converges in $W_p$ to the set
\[
\minsetkantorovich = \big\{\mspace{2mu}
    \pi \in \mathcal{A}(\mu,\nu)
        \mspace{2mu}:\mspace{2mu} \Koperator[\pi] = \Koperator[*]
\mspace{2mu}\big\}
\]
containing all the minimizers for \eqref{kantorovich_problem}.
In a general setting this would immediately imply the convergence
of the values $\Koperator[\pi_k]$ to $\Koperator[*]$, but
due to our assumptions we also get a bound on the convergence rate.
Our second result, \Cref*{introtheorem_conv_maps},
ensures the convergence of $T_k$ to $T_*$ in $d_p$ provided that
\ref{strong_uniqueness_hypothesis} holds.

\begin{introtheorem}
\label{introtheorem_conv_values_plans}
Let $(X,Y,\mu,\nu,c)$ be a compact OT problem and let $\Koperator[*]$
be the optimal value of its Kantorovich problem.
Then, every $\varepsilon$-optimal $h$-plan $\pi_h$ satisfies
\begin{equation}
\label{eq_intro_cuantitative_convergence_values}
\big\rvert\mspace{2mu} \Koperator[\pi_h] - \Koperator[*] \big\rvert
\leq \omega_c(h) + \varepsilon
\end{equation}
and for every approximating sequence of discrete plans $(\pi_k)_k$ we have
\[
\Koperator[\pi_k] \longconverges[k] \Koperator[*]
\mspace{10mu}\text{ and }\mspace{15mu}
W_p(\pi_k, \minsetkantorovich) \longconverges[k]0
\mspace{10mu}\text{ for all $1\leq p < \infty$.}
\]
In particular, when $(X,Y,\mu,\nu,c)$ also satisfies
\ref{kantorovich_uniqueness_hypothesis}
with optimal transport plan $\pi_*$ we get
$W_p(\pi_k,\pi_*)\longconverges[k]0$ for all $1\leq p < \infty$.
\end{introtheorem}

\begin{introtheorem}
\label{introtheorem_conv_maps}
Let $(X,Y,\mu,\nu,c)$ be a compact OT problem satisfying
\ref{strong_uniqueness_hypothesis} and let $T_*$ be an optimal transport map.
Then, if $(T_k)_k$ is an approximating sequence of projection maps,
we have $d_p(T_k,T_*)\longconverges[k]0$
for all $1\leq p < \infty$.
\end{introtheorem}

Since we are working with fully discrete approximations of $(X,Y,\mu,\nu,c)$
it could be argued that formulating
\Cref*{introtheorem_conv_maps}
in terms of $d_p(T_k,T_*)$
is somewhat unsatisfactory as in general $\mu$ is a non-discrete measure.
In fact, a common approach
\cite{oberman_ruan_2020,li_nochetto_2021}
in a setting like ours is to approximate
$d_p(T_k,T_*)$ by something like
\begin{equation}
\label{eq_intro_definition_disc_p}
\operatorname{disc}_p(T_k,T_*)
= \left(\sum\nolimits_{i\in I}{
    \mu_i \mspace{2mu} d^{\,p}_Y(T_k(x_i),T_*(x_i))
}\right)^{1/p}
\approx \mspace{6mu} d_p(T_k,T_*) \mspace{1mu},
\end{equation}
where the points $x_i$ and the weights $\mu_i=\mu[E_i]$ come from the
$h_k$-partition of $X$ used for constructing $T_k$.
Naturally, this leads to ask if the conclusion in
\Cref*{introtheorem_conv_maps}
holds when replacing $d_p(T_k,T_*)$ by $\operatorname{disc}_p(T_k,T_*)$.
It turns out that the answer is yes provided the map $T_*$ is
mildly discontinuous from the point
of view of $\mu$, in the sense that its discontinuity set
$D(T_*)$ is $\mu$-null.

\setcounter{introcorollary}{1}
\begin{introcorollary}
\label{introcorollary_fully_discrete_version_convergence_maps}
The condition $\mu[D(T_*)]=0$ is both necessary and sufficient for the
conclusion in \Cref*{introtheorem_conv_maps} to hold
when $d_p(T_k,T_*)$ is replaced by $\operatorname{disc}_p(T_k,T_*)$.
\end{introcorollary}

A few comments are in order.
First, notice that the set $D(T_*)$ is $\mu$-measurable since it
is always an $F_\sigma$ set and then Borel.
Second, the points $x_i$ and the weights $\mu_i$ in
\eqref{eq_intro_definition_disc_p}
vary with $k$ since each map $T_k$ comes from a different $h_k$-partition.
Third, it is known that the condition $\mu[D(T_*)]=0$ holds for a variety of
problems on $\mathbb{R}^d$ and Riemannian manifolds.
Indeed,
after the theory of partial regularity for OT maps
developed by Figalli, Kim and De Philippis
\cite{figalli_regularity_plane,
figalli_kim_partial_regularity,
figalli_de_philippis_partial_regularity},
we know that under mild regularity assumptions on $c$, $\mu$ and $\nu$
there are closed sets $\Sigma_X\subset X$, $\Sigma_Y\subset Y$ with
$\mu[\Sigma_X] = 0 =\nu[\Sigma_Y]$
such that
$T_*:X\backslash\Sigma_X\to Y\backslash\Sigma_Y$ is a homeomorphism.

\subsection*{Related Work and Our Contribution}

Experimental and theoretical work has been done on the approximation
of optimal values
\cite{delanoue_et_al_2016, gerber_maggioni_2017, bartels_hertzog_2022}
and optimal transport maps
\cite{oberman_ruan_2020,
li_nochetto_2021,
deb_ghosal_sen_2021,
manole_et_al_2022}
with discretization schemes similar to ours.
However, all these works are limited to Euclidean spaces
and in most cases they are also limited to the quadratic cost $c(x,y)=|x-y|^2$.
The exceptions are
the experimental works \cite{gerber_maggioni_2017, oberman_ruan_2020}
where some costs of the form $c(x,y)=|x-y|^p$ with $p \geq 1$
are considered and the
work \cite{bartels_hertzog_2022} where the convergence of optimal values
is proved for $c\in C^{1,\alpha}(X \times Y)$.
To our knowledge, neither theoretical nor empirical work has been done on the
discrete approximation of optimal transport at the level of generality
considered here.
In particular, we are not aware of previous work proposing the use of a
\emph{geometric median projection}
as a tool akin to the well known barycentric projection
\cite{ambrosio_gigli_savare}
for the study of OT on spaces without a linear structure.

When comparing
\Cref*{introtheorem_conv_values_plans}
with previously existing results on the convergence of optimal values
the most similar seem to be those obtained in
\cite{delanoue_et_al_2016, bartels_hertzog_2022}.
In the case of \cite{delanoue_et_al_2016}
the authors apply interval analysis
\cite{moore_interval_analysis_methods_applications,
hansen_walster_interval_analysis}
to study a discretization scheme similar to ours, obtaining
a convergence result for problems on $\mathbb{R}^1$
with absolutely continuous measures and $c\in C(\mathbb{R}^1\times \mathbb{R}^1)$.
A different, but still similar to ours, is the approach adopted
in \cite{bartels_hertzog_2022},
where $X$ and $Y$ are assumed to be triangulable domains in $\mathbb{R}^d$.
In this case the authors apply nodal interpolation estimates from
the theory of finite elements
\cite{brenner_scott_fem}
to deduce quantitative convergence in the sense of
\eqref{eq_intro_cuantitative_convergence_values},
obtaining rates of order $h^{1+\alpha}$
for $c\in C^{1,\alpha}(X \times Y)$.
We also point out that even though the main focus in
\cite{gerber_maggioni_2017} is computational,
the authors managed to prove something
\cite[eq.\ (13)]{gerber_maggioni_2017}
similar to a one-sided version of \eqref{eq_intro_cuantitative_convergence_values}.

Concerning the convergence of maps in the sense of
\Cref*{introtheorem_conv_maps}
or
\Cref*{introcorollary_fully_discrete_version_convergence_maps}
we highlight the works
\cite{li_nochetto_2021, oberman_ruan_2020}.
In the work \cite{li_nochetto_2021}
the authors %
assume quadratic cost and Lipschitz $T_*$
to derive quantitative estimates which imply that
$\operatorname{disc}_2(T_h,T_*)$ is $O(\sqrt{h})$
under those assumptions.
The work \cite{oberman_ruan_2020}, in turn, is more computationally oriented
but still relevant for our purposes for two main reasons:
first because their numerical experiments in $\mathbb{R}^d$ are consistent
with our convergence results
and second because they suggest a possible strategy
\cite[Remark 2.2]{oberman_ruan_2020}
for proving the convergence of (what we call) projection maps
from general results on barycentric projections
\cite{ambrosio_gigli_savare}.
Such a strategy, however, would work only for linear
$Y$ and convex $c(x,y)$.

We conclude this overview of related results by mentioning the works
\cite{deb_ghosal_sen_2021,
manole_et_al_2022},
where
the authors study a discretization scheme similar to ours
that also consists in solving an instance of
\eqref{kantorovich_discrete_problem}
and using the barycentric projection to approximate $T_*$.
Their approach differs from ours in that they work with empirical measures
instead of discretizing $\mu$ and $\nu$ as we did in
\eqref{eq_intro_discrete_measures_u_v},
so all their results are stochastic in nature.
Precisely, they start by considering i.i.d.\ random variables
$(X_i)_i \sim \mu \in \mathcal{P}(\mathbb{R}^d)$
and $(Y_j)_j \sim \nu \in \mathcal{P}(\mathbb{R}^d)$
in order to approximate $\mu$ and $\nu$ by
\[
\mu_n = \sum\nolimits_{i=1}^{n}{\tfrac{1}{n} \mspace{2mu} \delta_{X_i}}
\;\;\text{ and }\;\;
\nu_m = \sum\nolimits_{j=1}^{m}{\tfrac{1}{m} \mspace{2mu} \delta_{Y_j}}.
\]
Then, they obtain upper bounds for
$\mathbb{E}\big[\|T_{n,m}-T_*\|_{L^2(\mu_n)}\big]$
of order $n^{-\alpha}+m^{-\alpha}$,
where the exponent $\alpha>0$ depends on the dimension $d$
and $T_{n,m}$ is the
approximation of $T_*$ obtained by working with $\mu_n$ and $\nu_m$.
As in \cite{li_nochetto_2021},
the authors assume quadratic cost and Lipschitz regularity for $T_*$.

\subsection*{Outline}

We start with a discussion on our discretization scheme in
\Cref{section_discrete_approximation}, where we also introduce all the
notation and lemmas needed to work with $h$-plans and projection maps.
To emphasize the practical relevance of our contribution
we dedicate the last part of \Cref{section_discrete_approximation}
to show that projection maps can be effectively calculated
on general compact spaces and give an upper bound on the complexity of their computation.
Then we prove Theorems
\ref{introtheorem_conv_values_plans}
and
\ref{introtheorem_conv_maps}
in Sections
\ref{section_convergence_plans_values}
and
\ref{section_convergence_projection_maps_I}
respectively,
where we also discuss their sharpness and illustrate it by examples.
Finally,
we discuss and prove
\Cref{introcorollary_fully_discrete_version_convergence_maps}
in \Cref{section_convergence_projection_maps_II}.
The stability properties of OT on which our main results depend are
established in the \hyperref[section_stability_analysis_OT]{Appendix}.

\section{Discrete Approximation of %
    \texorpdfstring{$(X,Y,\mu,\nu,c)$}{(X, Y, \unicodemu, \unicodenu, c)}%
}
\label{section_discrete_approximation}

We stick to the definitions of $h$-partitions, $h$-approximations and $h$-plans
presented in
\nameref{section_introduction}
and insist on using the generic notation
\begin{equation}
\label{eq_generic_notation_partitions}
\begin{gathered}
    \mathcal{C}_X = \big\{ (E_i, x_i) \big\}_{i\in I} \;,\;\;
    \mathcal{C}_Y = \big\{ (F_j, y_j) \big\}_{j\in J} \;,\;\;
    \mu_i = \mu[E_i] \;,\;\;
    \nu_j = \nu[F_j] \;,\;\; \\
    \mu_h = \sum_{i\in I}{\mu_i \mspace{1mu} \delta_{x_i}}  \;,\;\;
    \nu_h = \sum_{j\in J}{\nu_j \mspace{1mu} \delta_{y_j}} \;,\;\;
    c_{ij} = c(x_i,y_j) \;,\;\;
    \pi_h = (\pi_{ij})_{ij}
\end{gathered}
\end{equation}
when working with them.
The parameter $h\geq 0$ should be interpreted as a quality parameter for the
$h$-partitions, which by definition satisfy
\[
\max{ \left\{
    \max_{i\in I}{\diam{E_i}},\, \max_{j\in J}{\diam{F_j}}
\right\}}
\leq h.
\]
Note that the case $h=0$ can only happen
for finite $X,Y$ since we require $I,J$ to be finite,
and in such case we get $\mu=\mu_h=\mu_0$ and $\nu=\nu_h=\nu_0$.
In general, we have $\mu_h,\nu_h \longweaklyconverges[h] \mu,\nu$
for $h\to 0$.
See
\Cref{remark_weak_convergence_marginals}
for a precise statement of this fact.

\begin{remark}
For our purposes we could have defined $h$-partitions by requiring
the sets to be only ``essentially disjoint'', in the sense that
\[
\mu[E_{i_1} \cap E_{i_2}] = 0 = \nu[F_{j_1} \cap F_{j_2}]
\;\;\text{ for $i_1\neq i_2$ and $j_1\neq j_2$}.
\]
We avoided this to prevent nonessential technicalities,
but for that same reason we will consider
$h$-partitions with only essentially disjoint sets in some examples.
\end{remark}

\begin{remark}
\label{remark_outer_approximation_of_X_Y}
Notice that we have not required either $\supp{\mu}=X$ or $\supp{\nu}=Y$.
Therefore, the usual practice of approximating a domain in $\mathbb{R}^d$
by a union of cubes is fully contained in our approach.
If we cover $X,Y \subset \mathbb{R}^d$ with finite families of
``simple'' sets like (disjoint) cubes, boxes or simplices, say
\begin{equation*}
\{S_i\}_{i\in I}
\;,\;
\{S_j\}_{j\in J}
\;\text{ with }\;
X \subset X_0 \defeq \bigsqcup\nolimits_{i}{S_i}
\;\text{ and }\;
Y \subset Y_0 \defeq \bigsqcup\nolimits_{j}{S_j},
\end{equation*}
we may work with
$(X_0,Y_0,\mu,\nu,c)$ instead of $(X,Y,\mu,\nu,c)$ and all our results
would remain valid.
In addition, note that choosing $x_i\in S_i$ and $y_j\in S_j$
would immediately yield a pair of $h$-partitions
for $X_0,Y_0$ with $h=\max_{k\in I\cup J}{\diam{S_k}}$.
\end{remark}

For an $h$-plan $\pi_h = {(\pi_{ij})}_{ij}$
we define its \emph{support} by
\begin{align*}
\supp{\pi_h} &= \big\{ \; (i,j) \in I\times J \,: \;\; \pi_{ij}>0  \; \big\}, \\
\intertext{which shares notation with the support of $\pi_h$
as a measure on $X\times Y$, i.e.}
\supp{\pi_h}
    &= \big\{ \; (x_i,y_j) \in I\times J \,: \;\; \pi_{ij}>0  \; \big\}.
\end{align*}
It will always be clear from context
which of this two notions is being referred.
Since some of our constructions require to operate with $\supp{\pi_h}$
we define the sets
\begin{equation*}
\begin{aligned}
I_+(j) &= \big\{ \; i \in I \,: \;\; \pi_{ij}>0  \; \big\}
\;\;\text{ for $j \in J$, } \\[1.0ex]
I_+ &= \big\{ \; i \in I \,:\,
    \exists \, j \in J \;\; \text{with} \;\; \pi_{ij}>0  \; \big\}
= \,\bigcup\nolimits_{j\in J}{I_+(j)}\,.
\end{aligned}
\end{equation*}
The sets $J_+(i)$ for $i \in I$ and $J_+$  are defined analogously, so
\begin{equation}
\label{eq_decomposition_of_support}
\operatorname{supp}(\pi_h)
= \bigsqcup_{j\in J_+}{I_+(j) \times \{ j \}}
= \bigsqcup_{i\in I_+}{\{ i \} \times J_+(i)}\,.
\end{equation}

\begin{remark}
\label{remark_measure_partitioning_by_I_plus_and_J_plus}
By definition, every $h$-plan $\pi_h$ is a feasible solution for
some instance of \eqref{kantorovich_discrete_problem} induced by a pair of
$h$-partitions $\mathcal{C}_X$, $\mathcal{C}_Y$.
Therefore, sticking to generic notation
\eqref{eq_generic_notation_partitions}
and taking into account the constraints in
\eqref{kantorovich_discrete_problem}, we obtain
\[
i\in I_+ \iff \mu[E_i]=\mu_i>0
\;\;\text{ and }\;\;
j\in J_+ \iff \nu[F_j]=\nu_j>0 \,.
\]
Then $\mu$ and $\nu$ are concentrated on $\bigsqcup_{i\in I_+}{E_i}$
and $\bigsqcup_{j\in J_+}{F_j}$ respectively, so we get
\begin{equation}
\label{eq_measure_partitioning_by_I_plus_and_J_plus}
\mu[E] = \sum\nolimits_{i\in I_+}{\mu[E\cap E_i]}
\;\;\text{ and }\;\;
\nu[F] = \sum\nolimits_{j\in J_+}{\nu[F\cap F_j]}
\end{equation}
for all Borel sets $E\in \mathcal{B}(X)$, $F\in \mathcal{B}(Y)$.
\end{remark}

\begin{definition}
\label{def_approximating_sequence_discrete_plans}
An \emph{approximating sequence of discrete plans} for $(X,Y,\mu,\nu,c)$
is any sequence $(\pi_k)_k \subset \mathcal{P}(X\times Y)$ such that
$\pi_k$ is an $\varepsilon_k$-optimal $h_k$-plan for every $k$
and the sequences $(\varepsilon_k)_k$ and $(h_k)_k$ satisfy
$\varepsilon_k \longconverges[k] 0$ and $h_k \longconverges[k] 0$.
\end{definition}

When $(Z,d_Z)$ is assumed Polish the existence of $h$-partitions for
arbitrary $h>0$ is equivalent to the compactness of $(Z,d_Z)$.
In particular, this implies the existence of approximating sequences of
discrete plans for compact OT problems.

\begin{lemma}
\label{lemma_existence_of_h_partitions}
Let $(X,Y,\mu,\nu,c)$ be an OT problem such that
$(X,d_X)$ and $(Y,d_Y)$ are Polish metric spaces.
Then the following statements are equivalent:
\begin{enumerate}[label=\roman*), ref=\emph{\roman*)}, topsep=-2pt, itemsep=2pt]
\item \label{lemma_existence_h_partitions_item_compact}
The spaces $(X,d_X)$ and $(Y,d_Y)$ are compact.
\item \label{lemma_existence_h_partitions_item_partition}
For every $h>0$ the problem $(X,Y,\mu,\nu,c)$ admits an optimal $h$-plan.
\item \label{lemma_existence_h_partitions_item_sequence}
There are approximating sequences of discrete plans for $(X,Y,\mu,\nu,c)$.
\end{enumerate}
\end{lemma}

\begin{proof}

Implication
\ref{lemma_existence_h_partitions_item_partition}
$\mspace{-6mu}\implies\mspace{-6mu}$
\ref{lemma_existence_h_partitions_item_sequence}
is immediate from
\Cref{def_approximating_sequence_discrete_plans}.
The existence of an $h$-plan for $(X,Y,\mu,\nu,c)$ implies
that $X$ and $Y$ can be covered by a finite number of sets with diameters
bounded by $h$, so
\ref{lemma_existence_h_partitions_item_sequence}
implies that $X$ and $Y$ are totally bounded.
Then, implication
\ref{lemma_existence_h_partitions_item_sequence}
$\mspace{-6mu}\implies\mspace{-6mu}$
\ref{lemma_existence_h_partitions_item_compact}
follows since $X$ and $Y$ are complete by hypothesis.

Now assume
\ref{lemma_existence_h_partitions_item_compact},
fix $h>0$
and consider a finite, minimal covering of $X$ by balls
of diameters at most $h$, say $\{B_i\}_{1\leq i \leq n}$.
By minimality, we can take $x_i\in B_i$ laying outside
every other ball in the covering.
Then, taking $I=\{1,2,\dots,n\}$,
\[
E_{1} \defeq B_{1}
\;\;\text{ and }\;\;
E_{k} \defeq B_{k} \backslash \,
\big( B_{1} \cup \dots \cup B_{{k-1}} \big)
\; \text{ for $2\leq k \leq n$}
\]
we obtain an $h$-partition
$\mathcal{C}_X = \big\{(E_i,x_i)\big\}_{i\in I}$ for $X$.
Constructing also an $h$-partition $\mathcal{C}_Y$ for $Y$,
the resulting instance of
\eqref{kantorovich_discrete_problem}
consists in the minimization
of a linear functional on the nonempty compact set
$\mathcal{A}(\mu_h,\nu_h) \subset \mathbb{R}^{I\times J}$.
Since the existence of a minimizer $\pi_h$ for such a problem is immediate
then \ref{lemma_existence_h_partitions_item_partition} follows.
\end{proof}

\begin{remark}
\label{remark_weak_convergence_marginals}
For compact $(X,Y,\mu,\nu,c)$ and $(\pi_k)_k$ as in
\Cref{def_approximating_sequence_discrete_plans}, the marginals
$\mu_k$, $\nu_k$ of the measures $\pi_k$ define
sequences of $h_k$-approximations with $h_k \converges 0$ and
this implies $\mu_k, \nu_k \longweaklyconverges[k] \mu, \nu$.
By symmetry, it suffices to check $\mu_k \longweaklyconverges[k] \mu$.
To this end, fix $\phi \in C_b(X)$ and note that
$\omega_\phi(h)\to 0$ for $h\to 0$.
Fixing $k$, generic notation
\eqref{eq_generic_notation_partitions}
applied to the fixed $h_k$-partition yields $\diam{E_i}\leq h_k$ and then
\begin{equation}
\label{eq_proof_uk_vk_weakly_convergent_to_u_v_bound_E_i}
\left|
\int_{E_i}{\phi \,d\mu} - \int_{E_i}{\phi \,d\mu_k}
\right|
=
\left|
\int_{E_i}{\phi \,d\mu} - \int_{E_i}{\phi(x_i) \,d\mu}
\right|
\leq
\mu[E_i] \, \omega_\phi(h_k).
\end{equation}
As $h_k\longconverges[k]0$,
summing
\eqref{eq_proof_uk_vk_weakly_convergent_to_u_v_bound_E_i}
over $i\in I$ proves $\mu_k \longweaklyconverges[k]\mu$ since
\[
\left|
\int_{X}{\phi \,d\mu} - \int_{X}{\phi \,d\mu_k}
\right|
\leq
\sum\nolimits_{i}{\,\mu[E_i] \, \omega_\phi(h_k)}
= \omega_\phi(h_k) \longconverges[k] 0.
\qedhere
\]
\end{remark}

\bigsubsection{Geometric Medians and Barycenters}

We start by making a distinction between general metric spaces and a class
of metric spaces for which barycenters are well-defined.

\begin{definition}
A metric space $(Y, d_Y)$ is \emph{normed} if $Y$ is a convex subset of a
vector space $\mathbb{V}$ and $d_Y$ is induced by a norm
$\|\cdot\|$ on $\mathbb{V}$, i.e.\ $d_Y(y_1,y_2)=\|y_1-y_2\|$.
\end{definition}

Given points $y_1 , \dots , y_m$ in a normed metric space $(Y, d_Y)$
and nonnegative weights $w_1 , \dots , w_m$ with $\sum_{j}{w_j}=1$,
the \emph{barycenter of $(y_j)_j$ with weights $(w_j)_j$}
is the point $\sum_{j}{w_j y_j} \in Y \subset \mathbb{V}$.
For a general metric space $(Y, d_Y)$ weighted geometric medians will be
considered in lieu of barycenters.

\begin{definition}
Given points $(y_j)_j$ and nonnegative weights $(w_j)_j$ as above, let
\[
V(y) = \sum\nolimits_{j}{w_j \mspace{1mu} d_Y(y,y_j)}
\;\text{ for every $y\in Y$.}
\]
A point $\hat{y}\in Y$ is a
\emph{geometric median for $(y_j)_j$ with weights $(w_j)_j$}
if $\hat{y}$ is a minimizer for $V$,
while $\hat{y}$ is
an \emph{$\varepsilon$-approximate geometric median
for $(y_j)_j$ with weights $(w_j)_j$}
if $V(\hat{y}) \leq V(y) + \varepsilon$ for every $y\in Y$.
Clearly, \emph{$0$-approximate geometric median} is a synonym for
\emph{geometric median}.
\end{definition}

Geometric medians exist for compact $(Y, d_Y)$ but they might not be unique.
For example, if $(Y, d_Y)$ is a Riemannian manifold uniqueness follows when
the points $(y_j)_j$ do not lie on the same geodesic
\cite[Theorem 3.1]{yang_2010_riemannian_median}.
In contrast to barycenters, geometric medians cannot be computed by evaluating
an effective formula, so in practice they have to be approximated.
It is because of this practical difficulty that we also consider approximate
geometric medians.

Even though geometric medians can be efficiently approximated by the well known
Weiszfeld algorithm on Euclidean spaces
\cite{weiszfeld_1937,
plastria_2011_weiszfeld}
and Riemannian manifolds
\cite{fletcher_2008_robust_estimation_median},
for arbitrary spaces it might be hard or even impossible to devise
a general and efficient algorithm.
The next lemma shows that, in our context, good enough approximations can be
obtained by a finite computation for general $(Y, d_Y)$.
Recall that a set $A\subset Y$ is an \emph{$\varepsilon$-net}
if for every $y\in Y$ there is $a\in A$ with $d_Y(a,y)\leq \varepsilon$.
For example, the points $(y_j)_j$ in any $h$-partition $\{(F_j,y_j)\}_j$
for $Y$ define an $h$-net in $Y$.

\begin{lemma}
\label{lemma_varepsilon_GM}
For a compact metric space $(Y,d_Y)$, fix an $\varepsilon$-net
$A = \{y_j\}_{1\leq j\leq m} \subset Y$ and nonnegative weights
$w_1,\dots, w_m$ with $\sum_{j}{w_j}=1$.
Then
\begin{equation}
\label{minimization_varepsilon_GM}
\hat{y}_\varepsilon
\,\defeq \,\underset{y \in A}{\operatorname{argmin}}{\;V(y)}
\,= \,\underset{y \in A}{\operatorname{argmin}}{
    \;\sum\nolimits_{j=1}^{m}{w_j \mspace{1mu} d_Y(y,y_j)}
}
\end{equation}
is an $\varepsilon$-approximate geometric median for $(y_j)_j$
with weights $(w_j)_j$.
\end{lemma}

\begin{proof}
Since $Y$ is compact, let $\hat{y} \in Y$ be a
weighted geometric median for $(y_j)_j$ with weights $(w_j)_j$.
Fixing $y_{j_0} \in A$ with $d_Y(\hat{y},y_{j_0}) \leq \varepsilon$,
for every $y\in Y$ we have
\begin{alignat*}{3}
\sum\nolimits_{j}{w_j d_Y(y_{j_0}, y_j)}
&\leq \sum\nolimits_{j}{w_j [d_Y(y_{j_0}, \hat{y}) + d_Y(\hat{y}, y_j)]}
\leq \sum\nolimits_{j}{w_j [\varepsilon + d_Y(\hat{y}, y_j)]} \\
&= \varepsilon + \sum\nolimits_{j}{w_j d_Y(\hat{y}, y_j)}
\leq \varepsilon + \sum\nolimits_{j}{w_j d_Y(y, y_j)}\,.
\end{alignat*}
Then, taking $\hat{y}_\varepsilon$ as in \eqref{minimization_varepsilon_GM},
for every $y\in Y$ we get
\begin{equation*}
V(\hat{y}_\varepsilon)
\leq \sum\nolimits_{j}{w_j d_Y(y_{j_0}, y_j)}
\leq \varepsilon + \sum\nolimits_{j}{w_j d_Y(y, y_j)}
= \varepsilon + V(y). \qedhere
\end{equation*}
\end{proof}

As a final note,
observe that geometric medians and barycenters are \emph{local},
in the sense that they remain near a point $y \in Y$
if the points $y_1 , \dots , y_m$ are clustered around that $y$.
We record this simple but useful idea in the following lemma,
which will be key for establishing the convergence of projection maps in
\Cref{section_convergence_projection_maps_I}.

\begin{lemma}[Nearness Lemma]
\label{nearness_lemma}
Consider points $y, y_1 , \dots , y_m$ in a metric space $(Y, d_Y)$ and
nonnegative weights $w_1 , \dots , w_m$ with $\sum_{j}{w_j}=1$.
\begin{enumerate}[%
label=\roman*), ref=\emph{\roman*)}, topsep=0pt, itemsep=2pt, leftmargin=25pt]
\item\label{nearness_lemma_item_general}
If $\hat{y}_\varepsilon$ is an $\varepsilon$-approximate geometric median for
$(y_j)_j$ with weights $(w_j)_j$, then
\[
d_Y(y,\hat{y}_\varepsilon)
\leq \varepsilon + 2\sum\nolimits_{j}{w_j \mspace{1mu} d_Y(y,y_j)}\,.
\]
\item \label{nearness_lemma_item_linear}
If $(Y,d_Y)$ is normed and $\hat{y}$ is the barycenter of $(y_j)_j$ with weights
$(w_j)_j$, then
\[
d_Y(y,\hat{y}) \leq \sum\nolimits_{j}{w_j \mspace{1mu} d_Y(y,y_j)}\,.
\]
\end{enumerate}
\end{lemma}

\begin{proof}
For \ref{nearness_lemma_item_general}, note that
\begin{equation*}
d_Y(y,\hat{y}_\varepsilon) \leq d_Y(y,y_j) + d_Y(y_j,\hat{y}_\varepsilon)
\;\;\text{ for $1\leq j\leq m$. }
\end{equation*}
Multiplying that inequality by $w_j$ and summing over $j$ yields
\begin{equation*}
d_Y(y,\hat{y}_\varepsilon)
\leq V(y) + V(\hat{y}_\varepsilon) \leq \varepsilon + 2 V(y) \,.
\end{equation*}
For \ref{nearness_lemma_item_linear},
since $d_Y$ is induced by the norm $\|\cdot\|$,
operating in $\mathbb{V} \supset Y$ yields
\begin{equation*}
d_Y(y,\hat{y})
    = \left\| \,y - \left(\sum\nolimits_{j}{w_j y_j}\right) \right\|
        = \left\| \sum\nolimits_{j}{w_j \left(y - y_j\right)} \right\|
    \leq \sum\nolimits_{j}{w_j \big\| y- y_j \big\|}
        = V(y)\,. \qedhere
\end{equation*}
\end{proof}

\bigsubsection{Projection Maps}

\begin{definition}
\label{definition_projection_maps}
Let $\pi_h = (\pi_{ij})_{ij}$ be an $h$-plan for $(X,Y,\mu,\nu,c)$
and let $(y_i)_{i\in I} \subset Y$ and $T\in\mathcal{B}(X,Y)$ be such that
$T|_{E_i}\equiv y_i$ for all $i\in I$.
\begin{enumerate}[%
label=\emph{\roman*}), topsep=0pt, itemsep=3pt, leftmargin=25pt]
\item \label{def_projection_maps_item_GM}
The map $T$ is a \emph{GM projection} if for some $\Qfact \geq 0$ and all
$i\in I_+$ the point $y_i$ is a
$\Qfact h$-approximate geometric median for
$(y_j)_{j\in J}$ with weights $\left( {\pi_{ij}} / {\mu_i} \right)_{j\in J}$.
\item \label{def_projection_maps_item_B}
When $(Y, d_Y)$ is normed, the map $T$ is a \emph{B projection} if
for all $i\in I_+$ the point $y_i$ is the barycenter of $(y_j)_{j\in J}$
with weights $\left({\pi_{ij}}/{\mu_i}\right)_{j\in J}$.
\end{enumerate}
The maps $T$ introduced by
\ref{def_projection_maps_item_GM}
and
\ref{def_projection_maps_item_B}
are \emph{projections extracted from $\pi_h$} and they define the sets
$\mathcal{T}_{GM}^{\Qfact}[\pi_h]$ and $\mathcal{T}_{B}[\pi_h]$ respectively.
If $T$ is a projection extracted from $\pi_h$ then $T$ is of
\emph{quality $\Qfact$} if
$T \in \mathcal{T}^{\Qfact}[\pi_h] \defeq
    \mathcal{T}_{GM}^{\Qfact}[\pi_h] \cup \mathcal{T}_{B}[\pi_h]$.
For $\Qfact=0$ we lighten the notation by defining
$\mathcal{T}_{GM}[\pi_h] \defeq \mathcal{T}_{GM}^{\,0}[\pi_h]$
and $\mathcal{T}[\pi_h] \defeq \mathcal{T}^{\,0}[\pi_h]$.
\end{definition}

\begin{remark}
\label{remark_projections_def}
The sets
$\mathcal{T}_{GM}^{\Qfact}[\pi_h]$,
$\mathcal{T}_{B}[\pi_h]$ and
$\mathcal{T}^{\Qfact}[\pi_h]$
depend on the parameter $h$,
on $h$-partitions $\mathcal{C}_X, \mathcal{C}_Y$
and on the $h$-plan $\pi_h$.
We avoid making this dependence explicit to keep notation light.
Note that $\mathcal{T}_{GM}^{\Qfact}[\pi_h]$ is nonempty
when $(Y, d_Y)$ is compact and that
$\mathcal{T}_{B}[\pi_h]$ is nonempty for normed $(Y, d_Y)$.
Lastly, notice that
\(
\mathcal{T}_{GM}^{\Qfact_1}[\pi_h]
\subset
\mathcal{T}_{GM}^{\Qfact_2}[\pi_h]
\)
for $0 \leq \Qfact_1 \leq \Qfact_2$, so projection maps of quality
$\Qfact_1$ are also of quality $\Qfact_2$.
\end{remark}

\begin{remark}
When $Y$ is a non-convex subset of a normed space $\mathbb{V}$ then
$(Y,d_Y)$ fails to be normed, but it may still make sense to extract
a B projection $T_h:X\to \mathbb{V}$ from $\pi_h$ if we are
willing to accept $T_h(X) \not\subset Y$.
This fits in with our approach: in the same spirit of
\Cref{remark_outer_approximation_of_X_Y},
we may enlarge $Y$ to its convex hull $\conv{Y}$ and work with
$(X,\conv{Y},\mu,\nu,c)$ instead of $(X,Y,\mu,\nu,c)$.
\end{remark}

\begin{definition}
\label{def_approximating_sequence_projection_maps}
An \emph{approximating sequence of projection maps} for $(X,Y,\mu,\nu,c)$
is a sequence $(T_k)_k \subset \mathcal{B}(X,Y)$ satisfying
$T_k \in \mathcal{T}^{\Qfact}[\pi_k]$ for all $k$, where
$(\pi_k)_k$ is some approximating sequence of discrete plans
and $\Qfact \geq 0$ is fixed.
\end{definition}

\begin{remark}
The existence of approximating sequences of projection maps
for a compact OT problem $(X,Y,\mu,\nu,c)$ is immediate from
\Cref{lemma_existence_of_h_partitions}
and
\Cref{remark_projections_def}.
In particular, such a sequence can be constructed by using B projections
for normed $(Y, d_Y)$, completely avoiding the computation of geometric medians.
Also, observe that $Q$ does not have to be ``small'' for
$(T_k)_k$ to satisfy
\Cref{def_approximating_sequence_projection_maps}.
Then, for any $\Qfact \geq 0$ we get
$T_k \longconverges[k] T_*$
according to
\Cref{introtheorem_conv_maps}.
At the end of this section we show that a sequence $(T_k)_k$ can be effectively
constructed for $\Qfact = 1$.
\end{remark}

\paragraph*{{\bfseries Computing Projection Maps}}
\pdfbookmark[3]{Computing Projection Maps}{Computing Projection Maps}

Computing a B projection for an $h$-plan $\pi_h$ requires the
calculation of linear combinations in a vector space $\mathbb{V} \supset Y$.
Precisely, $O(|J_+(i)|)$ vector space operations are needed
for each $i \in I_+$, yielding a total of
\[
\sum\nolimits_{i\in I_+}{O(|J_+(i)|)}
= O\left(\sum\nolimits_{i\in I_+}{|J_+(i)|}\right)
\stackrel[\eqref{eq_decomposition_of_support}]{}{=}
O\left(|\supp{\pi_h}|\right)
\]
vector space operations, which is $O(|I||J|)$ in the worst case.
In practice this bound improves to
$O(|I|+|J|)$
since, as explained in
\nameref{section_introduction},
any sensible solver for
\eqref{kantorovich_discrete_problem}
will produce a sparse solution $\pi_h$ with $|\supp{\pi_h}| \lesssim |I|+|J|$.
For example, when working in $\mathbb{V} = \mathbb{R}^d$ a total of
$O(d|I|+d|J|)$ sums and products of real numbers suffice to extract
a B projection from a sparse $h$-plan $\pi_h$.

Before estimating the computational cost of GM projections
note first that the minimization
in \eqref{minimization_varepsilon_GM} can be computed
whenever $d_Y$ can be effectively calculated.
As a consequence, a projection $T_h \in \mathcal{T}_{GM}^\Qfact[\pi_h]$
can be computed for $\Qfact = 1$ by repeatedly applying
\Cref{lemma_varepsilon_GM} to the $h$-net induced by
$\mathcal{C}_Y$.
This shows that GM projections can be implemented in practice when
the mathematical objects defining $(X,Y,\mu,\nu,c)$ can be represented
in a computer, as it happens with the usual barycentric projections,
which we call B projections after \Cref{definition_projection_maps}.

To derive an upper bound for the cost of computing a GM projection we start by
estimating the number of operations needed to evaluate
\eqref{minimization_varepsilon_GM}. To this end, write
\[
A= \{y_j :  1 \leq j \leq m\}
\;\text{ and }\;
A_+= \{y_j \in A : w_j>0\}
\]
and note that the brute-force evaluation of \eqref{minimization_varepsilon_GM}
requires $O(|A|^2)$ operations, including evaluations of $d_Y$ and
sums, products and comparisons of real numbers.
Note also that the sharper complexity bound $O(|A||A_+|)$ holds since
\eqref{minimization_varepsilon_GM}
consists in computing the minimum of $|A|$ sums with $|A_+|$ terms each.
When computing a GM projection by evaluating \eqref{minimization_varepsilon_GM}
this way, for each $i\in I_+$ we have
\[
A = \{ y_j : j \in J \}
\;\text{ and }\;
A_+= \{ y_j : j \in J_+(i) \},
\]
which yields $O(|J||J_+(i)|)$ operations.
Then, a GM projection requires
\[
\sum\nolimits_{i\in I_+}{O(|J||J_+(i)|)}
= O\left(|J|\sum\nolimits_{i\in I_+}{|J_+(i)|)}\right)
\stackrel[\eqref{eq_decomposition_of_support}]{}{=}
O\left(|J||\supp{\pi_h}|\right)
\]
operations. This is $O(|I||J|^2)$ in the worst case and improves to
$O(|I||J| + |J|^2)$ when $\pi_h$ is a sparse $h$-plan
with $|\supp{\pi_h}| \lesssim |I|+|J|$.

Despite the preceding complexity estimations,
it should be kept in mind that
algorithms with better complexity bounds may be devised for concrete metric
spaces and that, as a general rule,
the complexity bound $O(|J||J_+(i)|)$ may be improved by a clever evaluation of
\eqref{minimization_varepsilon_GM}.
To see why, notice first that $|I|$ and $|J|$ are large when $h$ is small,
and for small $h$ we have that:
\begin{enumerate}[%
label=\emph{\roman*}), topsep=0pt, itemsep=3pt, leftmargin=25pt]

\item the elements of $\{y_j: j \in J_+(i)\}$ tend to cluster around $T_*(x_i)$,
reducing the number of terms in the sums
\eqref{minimization_varepsilon_GM}
and making $|A_+| = |J_+(i)|$ smaller;

\item when $\{y_j: j \in J_+(i)\}$ is clustered, most of the points
in $\{y_j: j \in J\}$ can be discarded while searching for the approximate
minimizer of $V$ given by
\eqref{minimization_varepsilon_GM},
reducing the factor $|A| = |J|$.

\end{enumerate}

\section{Convergence of Discrete Plans and Optimal Values}
\label{section_convergence_plans_values}

Given an approximating sequence of discrete plans $(\pi_k)_k$,
the existence of a weakly convergent subsequence $(\pi_{k_l})_l$
satisfying $\Koperator[\pi_{k_l}]\longconverges[l]\Koperator[*]$
follows from general stability properties of optimal transport
\cite[Theorem 5.20]{villani_old_new}.
\Cref{introtheorem_conv_values_plans} asserts that, under appropriate hypotheses,
the convergence of both $(\pi_k)_k$ and its values
follows even without passing to a subsequence.
Our proof of
\Cref{introtheorem_conv_values_plans}
consists in estimating $\big| \Koperator[\pi_h] - \Koperator[*] \big|$
from above as a first step and then applying
\ref{proposition_stability_properties_OT_item_kantorovich}
in
\Cref{proposition_stability_properties_OT}.
That first step depends on the fact that every $h$-plan $\pi_h$ can be
slightly modified in order to obtain a \emph{continuous version}
$\continuous{\pi_h}$ of $\pi_h$ in $\mathcal{A}(\mu,\nu)$.

\bigsubsection{Continuous Version of an \texorpdfstring{$h$-plan}{h-plan}}

\begin{definition}
\label{definition_continuous_pi_h}
Given a feasible solution $\pi_h = (\pi_{ij})_{ij}$ for
\eqref{kantorovich_discrete_problem},
the \emph{continuous version of $\pi_h$} is the probability measure
$\continuous{\pi_h} \in \mathcal{P}(X\times Y)$ given by
\begin{equation}
\label{eq_definition_continuous_pi_h}
\continuous{\pi_h}
    = \sum_{(i,j)\,\in\,\supp{\pi_h}}{
            \frac{\pi_{ij}}{\mu_{i}\nu_{j}}} \; (\mu \otimes \nu) |_{E_i \times F_j
    }\,.
\end{equation}
\end{definition}

The probability measure $\continuous{\pi_h}$ is well-defined:
since $\mu_i,\nu_j>0$ for $(i,j)\in\supp{\pi_h}$, the sum
\eqref{eq_definition_continuous_pi_h}
is a well-defined measure in $\mathcal{M}^+(X\times Y)$ and
\[
\continuous{\pi_h} [X\times Y] =
\sum\nolimits_{ij}{\mathlarger{\tfrac{\pi_{ij}}{\mu_i\nu_j}} \; \mu|_{E_i}[X] \mspace{2mu} \nu|_{F_j}[Y]}
=\sum\nolimits_{ij}{\mathlarger{\tfrac{\pi_{ij}}{\mu_i\nu_j}} \; \mu[E_i] \, \nu[F_j]}
=\sum\nolimits_{ij}{\pi_{ij}} = 1 .
\]
As explained above, $\continuous{\pi_h}$ is a feasible plan for
\eqref{kantorovich_problem} constructed to be as similar as possible to $\pi_h$,
 in the sense described by the following proposition.

\begin{proposition}
\label{proposition_continuous_pi_h}
Let $\pi_h = (\pi_{ij})_{ij}$ be a feasible $h$-plan for a
general OT problem $(X,Y,\mu,\nu,c)$.
Then $\continuous{\pi_h}\in\mathcal{A}(\mu,\nu)$
is similar to $\pi_h$, in the sense that
\begin{equation}
\label{eq_continuous_pi_h_support_distance}
d_\mathcal{H}(\supp{\pi_h},\,\supp{\continuous{\pi_h}})\leq h
\;\;\text{ and }\;\;
\continuous{\pi_h}[E_i\times F_j] = \pi_{ij} = \pi_h[E_i\times F_j]
\;\text{ $\forall \,i,j$}.
\end{equation}
In addition,
if $c$ is bounded and uniformly continuous with modulus $\omega_c$, then
\begin{equation}
\label{eq_continuous_pi_h_cost}
\big\rvert\,
    \Koperator[\continuous{\pi_h}] - \Koperator[\pi_h]
\,\big\rvert \leq \omega_c(h).
\end{equation}
\end{proposition}

\begin{remark}
The notation $d_\mathcal{H}$
in \eqref{eq_continuous_pi_h_support_distance}
stands for the usual Hausdorff distance for subsets of
a metric space. Precisely, for subsets $A_1,A_2 \subset X \times Y$ we have
\[
d_\mathcal{H}(A_1,A_2) = \inf{ \big \{ \rho \geq 0 :
    A_2 \subset B_\rho[A_1]  \text{ and } A_1 \subset B_\rho[A_2]
\, \big \} },
\]
where $B_\rho[A]$ is the (closed) set of points at distance $\rho$ or less from $A\subset X\times Y$.
\end{remark}

\begin{proof}[Proof of \Cref{proposition_continuous_pi_h}]
For every Borel set $E \in \mathcal{B}(X)$ we get
\begin{alignat*}{2}
\continuous{\pi_h}[E\times Y] &=
\sum\nolimits_{ij}{\mathlarger{\tfrac{\pi_{ij}}{\mu_{i}\nu_{j}}} \, (\mu \otimes \nu) |_{E_i \times F_j}}[E\times Y]\\
&= \sum\nolimits_{ij}{\mathlarger{\tfrac{\pi_{ij}}{\mu_{i}\nu_{j}}}\, (\mu \otimes \nu) [(E\cap E_i) \times F_j]}\\
&= \sum\nolimits_{ij}{\mathlarger{\tfrac{\pi_{ij}}{\mu_{i}\nu_{j}}} \, \mu [E\cap E_i] \, \nu [F_j]}
\stackrel[\eqref{eq_generic_notation_partitions}]{}{=}
    \sum\nolimits_{ij}{\mathlarger{\tfrac{\pi_{ij}}{\mu_{i}}} \, \mu [ E\cap E_i ]} \\
&\stackrel[\eqref{eq_decomposition_of_support}]{}{=}
    \sum\nolimits_{i\in I_+}{\Big(\sum\nolimits_{j\in J_+(i)}{\pi_{ij}}\Big) \mathlarger{\tfrac{1}{\mu_{i}}} \, \mu[E\cap E_i]}\\
&= \sum\nolimits_{i\in I_+}{\mu_{i} \, \mathlarger{\tfrac{1}{\mu_{i}}} \, \mu[E\cap E_i]}
= \sum\nolimits_{i\in I_+}{\mu[E \cap E_i]}
\stackrel[]{\eqref{eq_measure_partitioning_by_I_plus_and_J_plus}}{=} \mu[E]\,,
\end{alignat*}
so the first marginal of $\continuous{\pi_h}$ is $\mu$.
A similar computation shows that $\nu$ is the second marginal of $\continuous{\pi_h}$,
proving that $\continuous{\pi_h}$ is a measure in $\mathcal{A}(\mu,\nu)$.

To prove \eqref{eq_continuous_pi_h_support_distance},
note first that
$\continuous{\pi_h}[E_i\times F_j] = \pi_{ij}$
follows immediately from
\eqref{eq_definition_continuous_pi_h}.
On the other hand, fixing $(i,j)\in \supp{\pi_h}$ we get
\begin{equation}\label{eq_continuous_pi_h_support_distance_proof_1}
\supp{\mu|_{E_i}\otimes\nu|_{F_j}} \subset \overline{E_i}\times\overline{F_j}
\subset B_h[(x_i,y_j)] \subset B_h[\supp{\pi_h}]  \,.
\end{equation}
Notice that $\supp{\continuous{\pi_h}} \cap (E_i\times F_j) \neq \emptyset$
since $\continuous{\pi_h}[E_i\times F_j]=\pi_{ij}>0$.
Then, taking $(\widehat{x}_i,\widehat{y}_j) \in \supp{\continuous{\pi_h}} \cap (E_i\times F_j)$,
we get
\begin{equation}
\label{eq_continuous_pi_h_support_distance_proof_2}
(x_i,y_j) \in E_i\times F_j \subset B_h[(\widehat{x}_i,\widehat{y}_j)]
\subset B_h[\supp{\continuous{\pi_h}}].
\end{equation}
Inclusions \eqref{eq_continuous_pi_h_support_distance_proof_1} and
\eqref{eq_continuous_pi_h_support_distance_proof_2}
for all $(i,j)$ in $\supp{\pi_h}$ imply inclusions
\[
\supp{\continuous{\pi_h}} \subset B_h[ \supp{\pi_h} ]
\quad\text{and}\quad
\supp{\pi_h} \subset B_h[ \supp{\continuous{\pi_h}} ]\,,
\]
which together imply
$d_\mathcal{H}(\supp{\pi_h},\,\supp{\continuous{\pi_h}})\leq h$.

Finally, to prove \eqref{eq_continuous_pi_h_cost} note first that for all $i,j$ we have
\[
\left\vert\int_{E_i \times F_j}{c\; d\continuous{\pi_h}} - \pi_{ij} c_{ij} \right\vert
=
\left\vert\int_{E_i \times F_j}{ [ c - c_{ij} ] \; d\continuous{\pi_h}}\right\vert
\leq
\pi_{ij}\mspace{2mu}\omega_c(h)
\]
since $\continuous{\pi_h}[E_i\times F_j]=\pi_{ij}$
and $\diam{E_i\times F_j}\leq h$.
Therefore, we get
\[
\big\rvert\, \Koperator[\continuous{\pi_h}] - \Koperator[\pi_h] \,\big\lvert
=
\left\rvert\, \sum_{ij}{\int_{E_i \times F_j}{c\; d\continuous{\pi_h}}} - \sum_{ij}{\pi_{ij} c_{ij}} \,\right\lvert
\leq
\sum_{ij}{\pi_{ij}\mspace{1mu}\omega_c(h)}
= \omega_c(h). \qedhere
\]
\end{proof}

\bigsubsection{Proof of Theorem \ref*{introtheorem_conv_values_plans}}

\begin{manualtheorem}{\ref*{introtheorem_conv_values_plans}}
\label{introtheorem_conv_values_plans_restatement}
Let $(X,Y,\mu,\nu,c)$ be a compact OT problem and let $\Koperator[*]$
be the optimal value of its Kantorovich problem.
Then every $\varepsilon$-optimal $h$-plan $\pi_h$ satisfies
\begin{equation}
\label{eq_introtheorem_conv_values_plans_restatement_bound_order}
\big\rvert\mspace{2mu} \Koperator[\pi_h] - \Koperator[*] \big\rvert
\leq \omega_c(h) + \varepsilon.
\end{equation}
Further, if $(\pi_k)_k$ is any approximating sequence of discrete plans then
\[
\Koperator[\pi_k] \longconverges[k] \Koperator[*]
\mspace{10mu}\text{ and }\mspace{15mu}
W_p(\pi_k, \minsetkantorovich) \longconverges[k]0
\mspace{10mu}\text{ for all $1\leq p < \infty$.}
\]
In particular, if the OT problem $(X,Y,\mu,\nu,c)$ also satisfies
\ref{kantorovich_uniqueness_hypothesis}
hypothesis with optimal transport plan $\pi_*$, then
$W_p(\pi_k,\pi_*)\longconverges[k]0$ for all $1\leq p < \infty$.
\end{manualtheorem}

\begin{proof}
Let $\pi_h$ be an $\varepsilon$-optimal $h$-plan for $(X,Y,\mu,\nu,c)$.
Then $\continuous{\pi_h}$ is a feasible solution for
\eqref{kantorovich_problem} by \Cref{proposition_continuous_pi_h}
and \eqref{eq_continuous_pi_h_cost} implies
\begin{equation*}
\Koperator[*]
\leq \Koperator[\continuous{\pi_h}]
\leq \Koperator[\pi_h] + \omega_c(h)
\leq \Koperator[\pi_h] + \omega_c(h) + \varepsilon.
\end{equation*}

Now take an optimal plan $\pi_*$ for
\eqref{kantorovich_problem}
and define
\begin{equation*}
\widetilde{\pi}_{ij} = \pi_* [E_i \times F_j]
\; \text{ for } \; (i,j) \in I \times J
\end{equation*}
to produce a discretization $\widetilde{\pi}_h = (\widetilde{\pi}_{ij})_{ij}$
of $\pi_*$.
It is easy to check that $\widetilde{\pi}_h$ is a feasible solution for
\eqref{kantorovich_discrete_problem}
and, since $\pi_h$ is $\varepsilon$-optimal, we get
\begin{align*}
\Koperator[\pi_h] - \varepsilon &\leq I[\widetilde{\pi}_h] = \sum_{ij}{\widetilde{\pi}_{ij}c_{ij}} = \sum_{ij}{c_{ij}\pi_*[E_i \times F_j]} \\
&= \sum_{ij}{\int_{E_i \times F_j}{c_{ij}\; d\pi_*}}
\stackrel[(*)]{}{\leq}
\sum_{ij}{\int_{E_i \times F_j}{[c + \omega_c(h)]\; d\pi_*}} \\
&= \sum_{ij}{\int_{E_i \times F_j}{c \; d\pi_*}} + \sum_{ij}{\int_{E_i \times F_j}{\omega_c(h)\;d\pi_*}} \\
&= \int_{X \times Y}{c \; d\pi_*} + \int_{X \times Y}{\omega_c(h) \; d\pi_*}
= \Koperator[\pi_*] + \omega_c(h) = \Koperator[*] + \omega_c(h) ,
\end{align*}
where inequality
$(*)$
follows since $\diam{E_i\times F_j}\leq h$ for all $(i,j)$.
This completes the proof of
\eqref{eq_introtheorem_conv_values_plans_restatement_bound_order}.
Now, taking $(\pi_k)_k$
as in
\Cref{def_approximating_sequence_discrete_plans}
and keeping in mind that $\omega_c(h)\converges 0$ for $\converges 0$,
it follows from
\eqref{eq_introtheorem_conv_values_plans_restatement_bound_order}
that
\[
\big| \Koperator[\pi_k] - \Koperator[*] \big|
\leq \omega_c(h_k) + \varepsilon_k \longconverges[k] 0.
\]
Then $\Koperator[\pi_k] \longconverges[k] \Koperator[*]$
and Theorem
\ref{introtheorem_conv_values_plans_restatement}
follows from
\ref{proposition_stability_properties_OT_item_kantorovich}
in \Cref{proposition_stability_properties_OT}.
\end{proof}

\begin{remark}
Note that $\pi_h$ is a purely atomic measure even if $\mu$ and $\nu$ are
non-atomic, which might be undesirable in some cases.
Fortunately, approximation properties similar to those in
\Cref{introtheorem_conv_values_plans_restatement}
hold for $\continuous{\pi_h}$, which is easily seen to be non-atomic
when $\mu$ and $\nu$ are.
To see why, notice first that
\eqref{eq_continuous_pi_h_cost}
and
\eqref{eq_introtheorem_conv_values_plans_restatement_bound_order}
imply
\[
\big\rvert\, \Koperator[\continuous{\pi_h}] - \Koperator[*] \,\big\rvert
\leq
\big\rvert\, \Koperator[\continuous{\pi_h}] - \Koperator[\pi_h] \,\big\rvert
+
\big\rvert\, \Koperator[\pi_h] - \Koperator[*] \,\big\rvert
\leq
\omega_c(h) +  \omega_c(h) + \varepsilon,
\]
so $\continuous{\pi_h}$ satisfies a cost inequality similar to
\eqref{eq_introtheorem_conv_values_plans_restatement_bound_order}.
Now note that the proof of
\Cref{introtheorem_conv_values_plans_restatement}
depends only on
\Cref{proposition_stability_properties_OT}
and bound
\eqref{eq_introtheorem_conv_values_plans_restatement_bound_order},
so that same proof works if we want a result like
\Cref{introtheorem_conv_values_plans_restatement}
for the continuous versions of the $h$-plans.
Thus, we immediately obtain the corollary below.
\end{remark}

\begin{corollary}
Let $(X,Y,\mu,\nu,c)$ be a compact OT problem and let $\Koperator[*]$
be the optimal value of its Kantorovich problem.
Then every $\varepsilon$-optimal $h$-plan $\pi_h$ satisfies
\begin{equation*}
\big\rvert\, \Koperator[\continuous{\pi_h}] - \Koperator[*] \,\big\rvert
\leq 2\mspace{2mu}\omega_c(h) + \varepsilon.
\end{equation*}
Further, if $(\pi_k)_k$ is any approximating sequence of discrete plans then
\[
\Koperator[\continuous{\pi_k}] \longconverges[k] \Koperator[*]
\mspace{10mu}\text{ and }\mspace{15mu}
W_p(\continuous{\pi_k}, \minsetkantorovich) \longconverges[k]0
\mspace{10mu}\text{ for all $1\leq p < \infty$.}
\]
In particular, if the OT problem $(X,Y,\mu,\nu,c)$ also satisfies
\ref{kantorovich_uniqueness_hypothesis}
hypothesis with optimal transport plan $\pi_*$, then
$W_p(\continuous{\pi_k},\pi_*)\longconverges[k]0$ for all $1\leq p < \infty$.
\end{corollary}

\bigsubsection{Discussion}
\label{discussion_plans_values}

The general theory of OT is usually formulated on Polish spaces
and its first result is the well-definedness of \eqref{kantorovich_problem},
which requires $c$ to be lower semicontinuous and bounded below
by a function in $L^1(\mu\otimes\nu)$.
Taking this as a starting point it is easy to see why hypotheses in
\Cref{introtheorem_conv_values_plans} are optimal.
First, note that dropping the compactness hypothesis on $X,Y$ is pointless since
it is necessary for approximating sequences of discrete plans to exist,
as shown by
\Cref{lemma_existence_of_h_partitions}.
Second, neither the convergence of values $\Koperator[\pi_k]$ nor the
convergence of plans $\pi_k$ hold if the continuity hypothesis on
$c$ is relaxed to lower semicontinuity,
as shown by
\Cref{example_no_convergence_for_semicontinuous_cost}
below.

On the other hand, the conclusions in
\Cref{introtheorem_conv_values_plans}
are also optimal: if no additional regularity is assumed for
$(X,Y,\mu,\nu,c)$ the convergence of $\pi_k$ is no better than weak
and the error bound
\eqref{eq_introtheorem_conv_values_plans_restatement_bound_order}
is sharp.
The first of these claims might by challenged by asking if
$(\pi_k)_k$ converges to $\pi_*$ in some stronger sense, e.g.\
in total variation or in the Wasserstein $W_\infty$ distance.
The sharpness of
\eqref{eq_introtheorem_conv_values_plans_restatement_bound_order},
in turn,
might by challenged by asking if the convergence rate
for the values $\Koperator[\pi_h]$ is better than
$O(\omega_c(h) + \varepsilon)$.

Recall that for $\sigma_1,\sigma_2 \in \mathcal{P}(Z)$
their \emph{total variation distance} is
\begin{equation}
\label{eq_total_variation_distance}
\delta(\sigma_1, \sigma_2) =
\tfrac{1}{2} \| \sigma_1 - \sigma_2 \|_{TV}
=
\sup_{B \in \mathcal{B}(Z)}{
    \big| \sigma_1[B] - \sigma_2[B] \mspace{2mu} \big|,
}
\end{equation}
while their \emph{Wasserstein $W_\infty$ distance}
\cite{santambrogio_2015_OT_for_app}
is
\begin{equation}
\label{eq_wasserstein_distance_infinity}
W_\infty(\sigma_1,\sigma_2)
= \mspace{-8mu}
\inf_{\sigma \in \mathcal{A}(\sigma_1,\sigma_2)}{
    \mspace{-2mu}
    \left[ \max_{(z_1,z_2)\in\supp{\sigma}}{d_{Z\times Z}(z_1,z_2)} \right]
    \mspace{-4mu}
\geq
d_{\mathcal{H}}(\supp{\sigma_1},\,\supp{\sigma_2}).
}
\end{equation}
In \Cref{example_only_weak_convergence_for_pi_k} we show that
convergence in $W_\infty$ fails in
\Cref{introtheorem_conv_values_plans}.
To see why convergence in total variation distance fails too notice
that every measure in $\mathcal{A}(\mu,\nu)$ is non-atomic if either
$\mu$ or $\nu$ are non-atomic,
so in that case we get
\[
\delta(\pi_k,\pi_*)
\stackrel[\eqref{eq_total_variation_distance}]{}{\geq}
    \big| \pi_k[\supp{\pi_k}] - \pi_*[\supp{\pi_k}] \big|
= |1-0\mspace{1mu}| = 1
\]
since $\supp{\pi_k}$ is always a finite set.

For brevity reasons we omit the construction of concrete examples
illustrating the sharpness of
\eqref{eq_introtheorem_conv_values_plans_restatement_bound_order},
but the reader can easily construct their own
by drawing some inspiration from those below
and the ones presented in Sections
\ref{section_convergence_projection_maps_I}
and
\ref{section_convergence_projection_maps_II}.
In exchange, we prove in
\Cref{proposition_optimality_order_conv_opt_values}
that
\eqref{eq_introtheorem_conv_values_plans_restatement_bound_order}
is indeed sharp if $\pi_h$ is an optimal $h$-plan (i.e. $\varepsilon=0$).
More precisely, we show that
\eqref{eq_introtheorem_conv_values_plans_restatement_bound_order}
cannot be improved without imposing additional hypotheses
on $\mu,\nu$ or the $h$-partitions $\mathcal{C}_X,\mathcal{C}_Y$.

\begin{example}
\label{example_no_convergence_for_semicontinuous_cost}
Take $X=Y=[0,1]$ and $\mu=\nu=\mathcal{L}^1$ the $1$-dimensional Lebesgue measure
on $[0,1]$.
Consider the lower semicontinuous cost $c$
taking the value $0$ on the diagonal $\Delta = \{ (x,x) : x \in X\}$
and the value $1$ elsewhere.
It is easy to check that
$T_* = \operatorname{Id}_{[0,1]}$
is the unique optimal transport map and
that $\pi_* = \pi_{T_*}$ is the unique optimal plan, so
\ref{strong_uniqueness_hypothesis}
and
\ref{kantorovich_uniqueness_hypothesis}
hold with $\Koperator[*]=\Koperator[\pi_*]=\Moperator[T_*]=0$.

Fix $k\in\mathbb{N}$, take $h_k=k^{-1}$ and consider
$h_k$-partitions $\mathcal{C}_X$, $\mathcal{C}_Y$ given by
\[
E_i = \big[\tfrac{i-1}{k}, \tfrac{i}{k}\big]  \,,\;
x_i = \tfrac{2i-1}{2k} + \delta_i  \,,\;
F_j = \big[\tfrac{j-1}{k}, \tfrac{j}{k}\big]  \,,\;
y_j = \tfrac{2j-1}{2k}
\;\;\text{ for $1\leq i,j \leq k$.}
\]
Then $\mathcal{C}_Y$ consists in $k$ intervals of length $h_k$ pointed
on their centers and we choose $0<|\delta_i|<(2k)^{-1}$ so that the intervals
$E_i$ in $\mathcal{C}_X$ are pointed off their centers.
Note that no point $(x_i,y_j)$ lies on $\Delta$ since $\delta_i\neq 0$ for all $i$.
Then $c_{ij}=1$ for all $i,j$
and every feasible solution of
\eqref{kantorovich_discrete_problem}
is optimal with cost $1$. Among them, choose
\[
\pi_k = \tfrac{1}{k} \, \delta_{(x_1,y_k)}
+
\tfrac{1}{k} \, \delta_{(x_2,y_{k-1})}
+ \dots +
\tfrac{1}{k} \, \delta_{(x_k,y_1)},
\]
as depicted in
\Cref{figure_example_no_convergence_semicontinuous_cost}.
Repeating this process for every $k\in\mathbb{N}$,
we get $(\pi_k)_k$ with $\Koperator[\pi_k]=1$ for all $k$.
Then $\Koperator[\pi_k]\centernot{\longconverges} \Koperator[*]$
and $\pi_k \centernot{\longweaklyconverges} \pi_*$.
Besides, if $T_k \in \mathcal{T}[\pi_k]$,
every map $T_k$ is an approximation of $T(x)=1-x$
so $d_p(T_k,T_*) \centernot{\longconverges} 0$.
This shows that Theorems
\ref{introtheorem_conv_values_plans}
and
\ref{introtheorem_conv_maps}
fail if $c$ is allowed to be lower semicontinuous.
\end{example}

\begin{example}
\label{example_only_weak_convergence_for_pi_k}
Take $X,Y,\mu,\nu$ as in
\Cref{example_no_convergence_for_semicontinuous_cost},
but this time consider the cost function
$c(x,y)=(x-y)^2(x-\tfrac{1}{2})^2(y-\tfrac{1}{2})^2$.
The cost $c$ is a smooth perturbation of the quadratic cost $(x-y)^2$
by a multiplicative factor $(x-\tfrac{1}{2})^2(y-\tfrac{1}{2})^2$, which
adds two lines to the zero set $Z_c$ of $c$ as shown in
\Cref{figure_example_only_weak_convergence_for_pi_k}.
Again, it is easy to check that
$T_* = \operatorname{Id}_{[0,1]}$
is the unique optimal transport map and
that $\pi_* = \pi_{T_*}$ is the unique optimal plan, so
\ref{strong_uniqueness_hypothesis}
and
\ref{kantorovich_uniqueness_hypothesis}
hold with $\Koperator[*]=\Koperator[\pi_*]=\Moperator[T_*]=0$.
Note that the lines added to $Z_c$ are not enough to break down the uniqueness
of optimal solutions since both $\mu$ and $\nu$ have no atoms.

Fix $k\in\mathbb{N}$, take $h_k=(2k+1)^{-1}$ and consider identical
$h_k$-partitions $\mathcal{C}_X = \mathcal{C}_Y$ given by
$2k+1$ intervals pointed on their centers, i.e.
\[
E_i = F_i = \big[\tfrac{i-1}{2k+1}, \tfrac{i}{2k+1}\big]
\;\;\text{ and }\;\;
x_i = y_i = \tfrac{2i-1}{2(2k+1)}
\;\;\text{ for $1\leq i \leq 2k+1$.}
\]
Since $c\geq 0$, every zero cost solution of
\eqref{kantorovich_discrete_problem} is optimal.
Among them, choose $\pi_k$ to be almost concentrated on the diagonal as in
\Cref{figure_example_only_weak_convergence_for_pi_k},
that is
\begin{equation*}
\pi_k =
\tfrac{1}{2k+1}\mspace{2mu}\delta_{(x_1,y_{k+1})}
+
\tfrac{1}{2k+1}\mspace{2mu}\delta_{(x_{k+1},y_1)}
+
\sum_{i\notin \{1,k+1\}}{\tfrac{1}{2k+1}\mspace{2mu}\delta_{(x_i,y_i)}}.
\end{equation*}

Construct $(\pi_k)_k$ by repeating this process for every $k\in\mathbb{N}$.
Observe that the support of each $\pi_k$ contains two points off the
diagonal, precisely
\[
\left(\tfrac{1}{2(2k+1)}\,,\tfrac{1}{2}\right), \,
\left(\tfrac{1}{2}\,,\tfrac{1}{2(2k+1)}\right)
\in \supp{\pi_k}
\;\text{ for all $k$. }
\]
Since they converge to $(0,\nicefrac{1}{2})$ and $(\nicefrac{1}{2},0)$
respectively, for large $k$ we have
\[
\tfrac{\sqrt{2}}{4}
\approx
d_{\mathcal{H}}(\supp{\pi_k},\,\supp{\pi_*})
\stackrel[\eqref{eq_wasserstein_distance_infinity}]{}{\leq}
W_\infty(\pi_k,\pi_*),
\]
so $W_\infty(\pi_k,\pi_*) \centernot\longconverges 0$.
\end{example}

\begin{remark}
\label{remark_conv_maps_fails_for_p_equals_infty}
The previous example also shows that
\Cref{introtheorem_conv_maps}
fails for $p=\infty$, that is, for the pseudometric $d_\infty$ defined by
\begin{equation*}
    d_\infty(T_1,T_2) = \; \esssup_{x\in X}{\;\,d_Y(T_1(x),T_2(x))}
    \;\;\text{ for $T_1,T_2 \in \mathcal{B}(X,Y)$.}
\end{equation*}
To see why, construct $(T_k)_k$ from the sequence $(\pi_k)_k$ in
\Cref{example_only_weak_convergence_for_pi_k}
and notice that each projection map $T_k \in \mathcal{T}[\pi_k]$
necessarily satisfies $T_k(x) = \nicefrac{1}{2}$ for every $x\in E_1$,
which is clear from
\Cref{figure_example_only_weak_convergence_for_pi_k}.
Then, we get
\[
d_\infty(T_k,T_*) \geq \esssup_{E_1}{\big|T_k-T_*\big|}
= \esssup_{0 \leq x \leq h_k}{\big|\tfrac{1}{2}-x\big|}
= \tfrac{1}{2}. \qedhere
\]
\end{remark}

\begin{figure}
\centering
\subfigure[\Cref{example_no_convergence_for_semicontinuous_cost}]{
    \label{figure_example_no_convergence_semicontinuous_cost}
    \includegraphics[width=0.45\textwidth]{%
        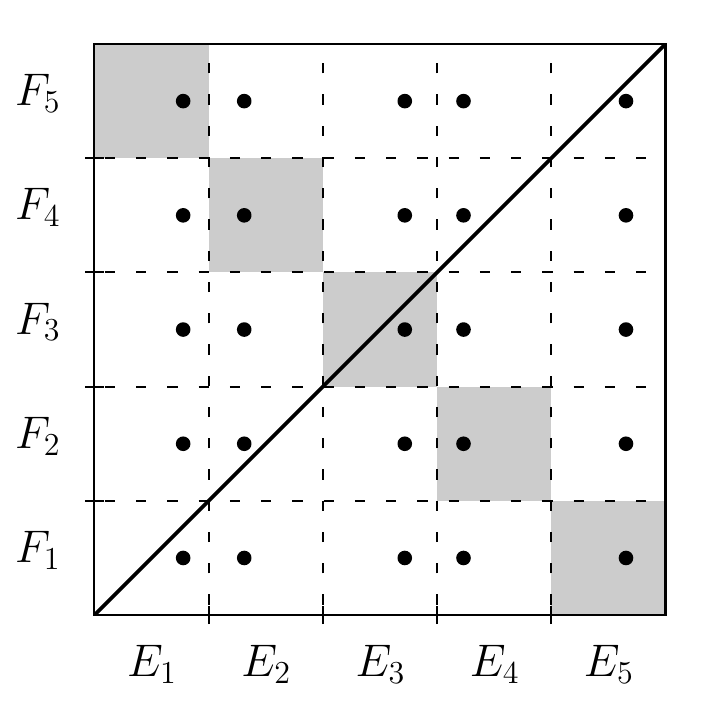%
        }
}\hspace{4mm}%
\subfigure[\Cref{example_only_weak_convergence_for_pi_k}]{
    \label{figure_example_only_weak_convergence_for_pi_k}
    \includegraphics[width=0.45\textwidth]{%
        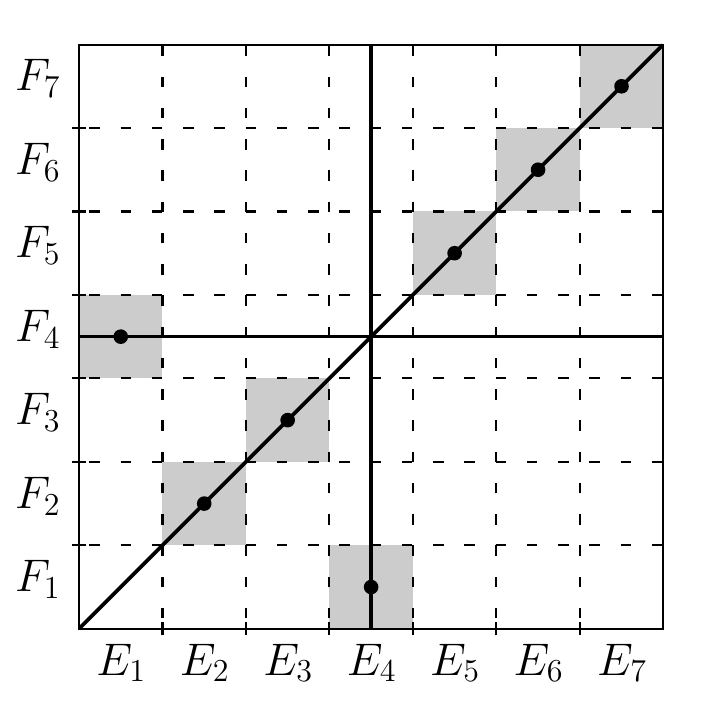%
    }
}
\captionsetup{width=0.9\linewidth}
\caption{The thick lines show the zero set of the cost $c$ and the dots inside
the gray squares show $\supp{\pi_k}$.
In \ref{figure_example_no_convergence_semicontinuous_cost}, note that all points
are off the diagonal.
In \ref{figure_example_only_weak_convergence_for_pi_k}, note that the
points off the diagonal ``do the work'' of the two missing points in
$E_1\times F_1$ and $E_4\times F_4$.
}
\end{figure}

\begin{proposition}
\label{proposition_optimality_order_conv_opt_values}
For compact $(X,d_X)$ and $(Y,d_Y)$ let  $c:X\times Y \to \mathbb{R}$ be a
continuous cost function. Assume there is a function $\phi:(0,\infty)\to(0,\infty)$
depending only on $X$, $Y$ and $c$ with the following property:
for any choice of measures
$\mu \in \mathcal{P}(X)$, $\nu \in \mathcal{P}(Y)$
and every optimal $h$-plan $\pi_h$ for $(X,Y,\mu,\nu,c)$, it holds that
\begin{equation*}
\big\rvert\, \Koperator[\pi_h] - \Koperator[*] \,\big\rvert \leq \phi(h).
\end{equation*}
Then $\phi \geq \omega_c$, which means that $\phi(h) \geq \omega_c(h)$ for all $h>0$.
\end{proposition}

\begin{proof}
Consider $\phi$ as above and fix $h>0$.
Taking $\rho>0$, by definition of $\omega_c$ it follows that
there are $(x_1,y_1)$, $(x_2,y_2)$ in $X\times Y$ such that
\begin{align*}
|c(x_1,y_1) - c(x_2,y_2)| &\geq (1-\rho)\,\omega_c(h) \, ,\\
h \geq d_{X\times Y}((x_1,y_1), (x_2,y_2))
&= \max{ \{ d_X (x_1, x_2) , d_Y (y_1, y_2) \} }.
\end{align*}
Define $\mu = \delta_{x_2}$, $\nu = \delta_{y_2}$ and consider $(X,Y,\mu,\nu,c)$.
Note that $\Koperator[*] = c(x_2,y_2)$ for said problem.
Now define $E_1 = \{x_1,x_2\}$, $F_1 = \{y_1,y_2\}$ and notice that
\[
E_1 \in \mathcal{B}(X)\,,\, x_1 \in E_1\,,\, \diam{E_1}\leq h
\;\text{ and }\;
F_1 \in \mathcal{B}(Y)\,,\, y_1 \in F_1\,,\, \diam{F_1}\leq h.
\]
Proceeding as in the proof of
\Cref{lemma_existence_of_h_partitions},
the ``partial'' $h$-partitions
$\{ (E_1, x_1) \}$ and $\{ (F_1, y_1) \}$
can be extended to $h$-partitions $\mathcal{C}_X$, $\mathcal{C}_Y$ for $X,Y$.
These partitions induce a unique optimal
$h$-plan $\pi_h = \delta_{(x_1,y_1)}$ with $\Koperator[\pi_h] = c(x_1,y_1)$.
Then
\[
\phi(h)\geq |\Koperator[\pi_h] - \Koperator[*]|
= |c(x_1,y_1) - c(x_2,y_2)| \geq (1-\rho)\,\omega_c(h)
\]
and $\phi(h)\geq \omega_c(h)$ follows by taking $\rho\to 0$.
\end{proof}

\section{Convergence of Projection Maps I}
\label{section_convergence_projection_maps_I}

Now we focus on the convergence of projections $T_h$ for
$h \to 0$ to give a proof of \Cref{introtheorem_conv_maps}
and discuss a few examples illustrating its sharpness.
From a conceptual point of view, the strategy for the proof is to control
$d_p(T_h,T_*)$ from above by the $\pi_h$-measure
of the ``bad'' sets
\begin{equation*}
B_{T_*}(\delta) = \big\{ \,
                (x,y)\in X\times Y \,:\, d_Y(y,T_*(x)) \geq \delta
            \,\big\},
\end{equation*}
and $T_h \converges T_*$ follows
since $\pi_h[B_{T_*}(\delta)] \approx \pi_*[B_{T_*}(\delta)] = 0$
for $h \approx 0^+$ and $\delta > 0$.
The intuition is clear: if the ``good'' sets
$(X\times Y) \setminus  B_{T_*}(\delta)$ are almost full measure then
$\pi_h$ is mostly concentrated around the graph of $T_*$,
pushing $T_h$ towards $T_*$.

Despite its conceptual simplicity, the strategy outlined above requires
some effort at the technical level to produce an actual proof.
As we did for \Cref{introtheorem_conv_values_plans},
we start by introducing the
\emph{semidiscrete version} $\semidiscrete{\pi_h}$ of $\pi_h$.
The idea here is to tweak $\pi_h$
by making its first marginal to be $\mu$ instead of $\mu_h$,
which enables the application of
\ref{proposition_stability_properties_OT_item_monge}
in
\Cref{proposition_stability_properties_OT}.
The second ingredient of the proof,
an upper bound for $d_p(T_h,T_*)$ in terms of
$\semidiscrete{\pi_h}[B_{T_*}(\delta)]$,
is derived in two steps:
first we obtain point estimates for $d_p(T_h(x),T_*(x))$
by applying Nearness Lemma \ref{nearness_lemma} and then
we glue them together by relying on the
disintegration of $\semidiscrete{\pi_h}$ with respect to $\mu$.

\bigsubsection{Semidiscrete Version of an \texorpdfstring{$h$-plan}{h-plan}}

\begin{definition}
\label{definition_semidiscrete_pi_h}
Given a feasible solution $\pi_h = (\pi_{ij})_{ij}$ for
\eqref{kantorovich_discrete_problem},
the \emph{semidiscrete version of $\pi_h$} is the probability measure
$\semidiscrete{\pi_h} \in \mathcal{P}(X\times Y)$ given by
\begin{equation}
\label{eq_definition_semidiscrete_pi_h}
\semidiscrete{\pi_h}
    = \sum_{(i,j)\,\in\,\supp{\pi_h}}{
        \mathlarger{\tfrac{\pi_{ij}}{\mu_i}} \; \mu|_{E_i}\otimes \delta_{y_j}
    }\,.
\end{equation}
\end{definition}

The probability measure $\semidiscrete{\pi_h}$ is well-defined:
since $\mu_i>0$ for $(i,j)\in\supp{\pi_h}$, the sum
\eqref{eq_definition_semidiscrete_pi_h}
is a well-defined measure in $\mathcal{M}^+(X\times Y)$ and
\[
\semidiscrete{\pi_h} [X\times Y] =
\sum\nolimits_{ij}{
    \mathlarger{\tfrac{\pi_{ij}}{\mu_i}} \;
        \mu|_{E_i}[X] \mspace{2mu} \delta_{y_j}[Y]}
=\sum\nolimits_{ij}{\mathlarger{\tfrac{\pi_{ij}}{\mu_i}} \; \mu[E_i]}
=\sum\nolimits_{ij}{\pi_{ij}} = 1 .
\]
As mentioned earlier, the idea behind \Cref{definition_semidiscrete_pi_h}
is to modify $\pi_h$ as little as possible to make it have first marginal
$\mu$.
To this end, for $i\in I_+$, the restriction $\mu|_{E_i}$ is rewritten as a
convex combination of copies of $\mu|_{E_i}$ ``lifted at heights $y_j$''
(i.e. $\mu|_{E_i}\otimes \delta_{y_j}$) with weights ${\pi_{ij}}/{\mu_i}$,
corresponding to the proportion of mass that $x_i$ sends to $y_j$ in the
discrete plan $\pi_h$ (see \Cref{figure_semidiscrete_pi_h_supports}).
The following proposition records all
the relevant properties of $\semidiscrete{\pi_h}$.

\begin{figure}%
\centering
\captionsetup{width=0.9\linewidth}
\includegraphics[width=1.0\textwidth]{%
    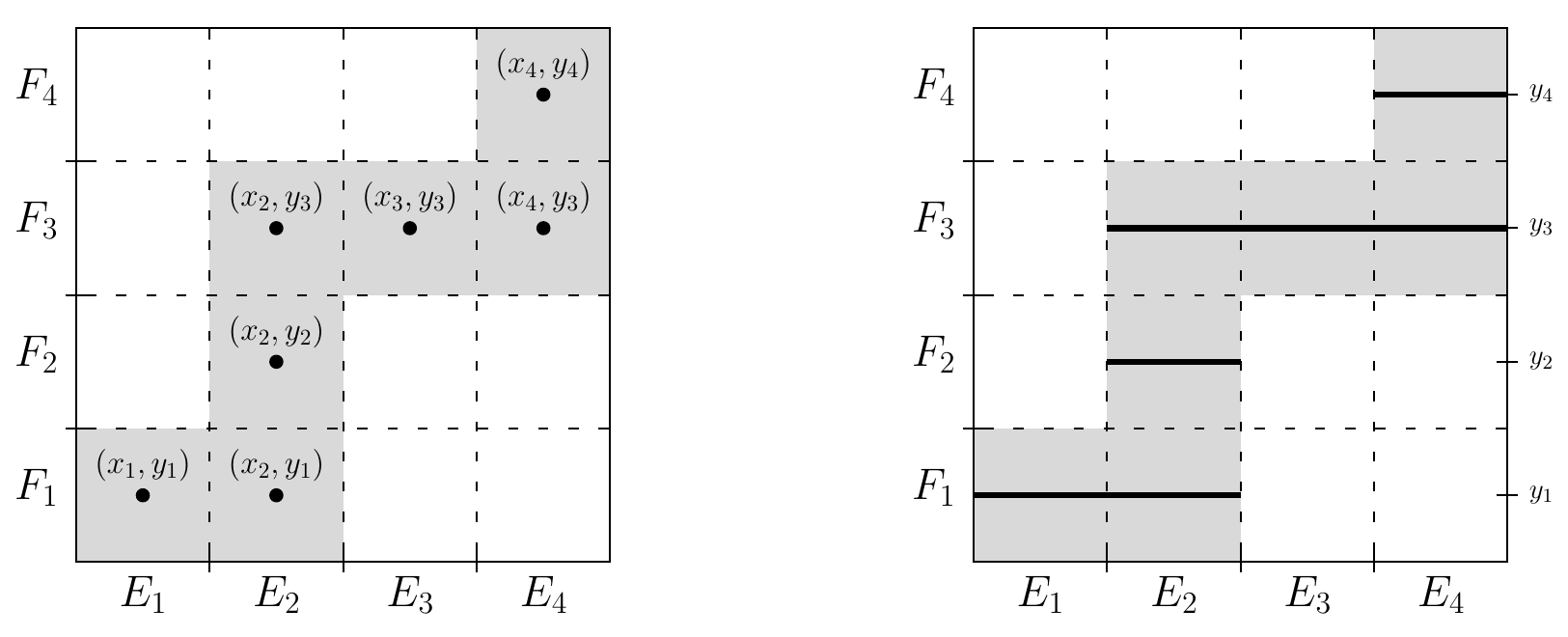%
}
\caption{
    The supports of $\pi_h$ (left), $\semidiscrete{\pi_h}$ (right) and
    $\continuous{\pi_h}$ (shaded region).
}
\label{figure_semidiscrete_pi_h_supports}
\end{figure}

\begin{figure}%
\centering
\captionsetup{width=0.8\linewidth}
\includegraphics[width=1.0\textwidth]{%
    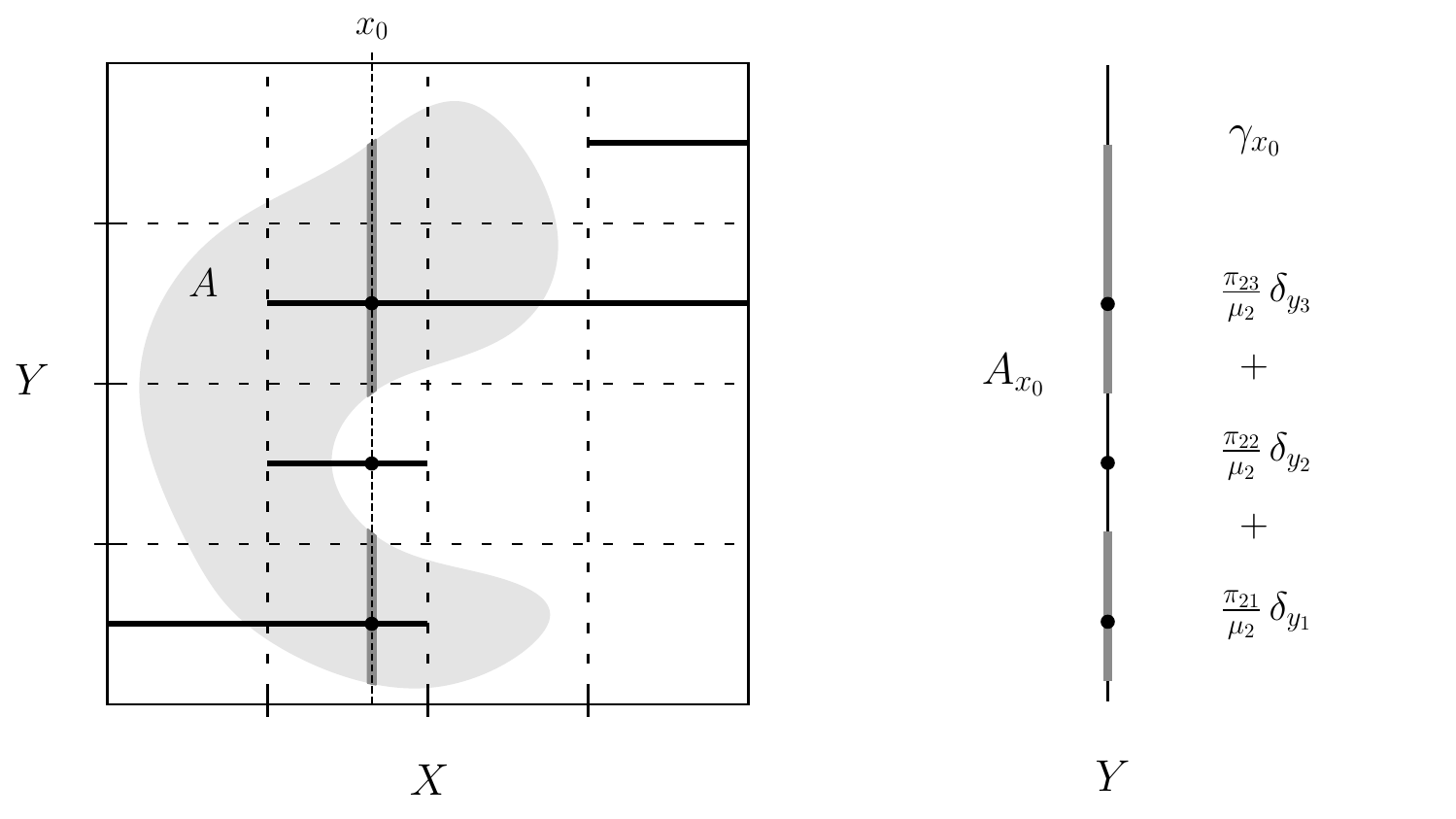%
}
\caption{
    The measure $\semidiscrete{\pi_h}$ in
    \Cref{figure_semidiscrete_pi_h_supports}
    disintegrates with respect to $\mu$ as a family of finitely supported
    measures $(\gamma_x)_{x\in X}$.
    For $A$ and $x_0$ as above we get
    $ (\delta_{x_0} \otimes \gamma_{x_0} )[A] =
    \gamma_{x_0}[A_{x_0}] = \tfrac{\pi_{21}}{\mu_2} + \tfrac{\pi_{23}}{\mu_2}$.
}
\label{figure_semidiscrete_pi_h_disintegration}
\end{figure}

\begin{proposition}
\label{proposition_semidiscrete_pi_h}
Let $\pi_h = (\pi_{ij})_{ij}$ be a feasible $h$-plan for a general
OT problem $(X,Y,\mu,\nu,c)$.
Then $\semidiscrete{\pi_h}\in\mathcal{A}(\mu,\nu_h)$
is similar to $\pi_h$, in the sense that
\begin{equation}
\label{eq_semidiscrete_pi_h_support_distance}
d_\mathcal{H}(\supp{\pi_h},\,\supp{\semidiscrete{\pi_h}})\leq h
\;\;\text{ and }\;\;
\semidiscrete{\pi_h}[E_i\times F_j] = \pi_{ij} = \pi_h[E_i\times F_j]
\;\text{ $\forall \,i,j$}.
\end{equation}
Additionally, $\semidiscrete{\pi_h}$ admits the disintegration
\begin{equation}
\label{eq_formula_disintegration}
\semidiscrete{\pi_h} =
\sum_{i}{
    \int_{E_i}{ \delta_x \otimes \Big(
        \sum\nolimits_{j}{\tfrac{\pi_{ij}}{\mu_i}\,\delta_{y_j}}
    \Big) \,d\mu(x) }
}\mspace{1mu}.
\end{equation}
That is, for each Borel set $A \in \mathcal{B}(X\times Y)$ the identity
\begin{equation}
\label{eq_formula_evaluation_disintegration}
\semidiscrete{\pi_h} [A] =
\sum_{i}{
    \int_{E_i}{ \Big(
        \sum\nolimits_{j}{\tfrac{\pi_{ij}}{\mu_i}\,\delta_{y_j}[A_x]}
    \Big) \,d\mu(x) }
}
\end{equation}
holds, where $A_x = \{y\in Y : (x,y) \in A\}$.
\end{proposition}

\begin{remark}
For $i\in I \setminus I_+$ we have $\mu_i=0$ so division by
$\mu_i$ is ill-defined, but the corresponding integrals in
\eqref{eq_formula_disintegration}
and
\eqref{eq_formula_evaluation_disintegration}
are well-defined and evaluate to zero since $\mu[E_i]=\mu_i=0$ in that case.
As a consequence, it makes no difference to consider the sums
over $I_+$ or over $I$.
On another note, disintegration
\eqref{eq_formula_disintegration}
is just the standard disintegration $(\gamma_x)_{x\in X}$ for
$\semidiscrete{\pi_h}$ with respect to $\mu$, where $\gamma_x$ is $\mu$-a.e.
defined by prescribing
\(
\gamma_x = \sum_{j}{\tfrac{\pi_{ij}}{\mu_i}\delta_{y_j}}
\)
for $x\in E_i$ and $i\in I_+$.
Then $\semidiscrete{\pi_h} = \int_{X}{\delta_x \otimes \gamma_x \, d\mu(x)}$
and, for $A \in \mathcal{B}(X\times Y)$, the rightmost integral in
\[
\semidiscrete{\pi_h}[A]
= \int_{X}{ (\delta_x \otimes \gamma_x) [A]\, d\mu(x)}
= \int_{X}{\gamma_x [A_x] \, d\mu(x)}
\]
expands to
\eqref{eq_formula_evaluation_disintegration}
by plugging in the piecewise definition for $\gamma_x$.
See \Cref{figure_semidiscrete_pi_h_disintegration} for an illustration of how
disintegration $(\gamma_x)_{x\in X}$ looks for $\semidiscrete{\pi_h}$.
\end{remark}

\begin{proof}[Proof of \Cref{proposition_semidiscrete_pi_h}]
The fact that $\semidiscrete{\pi_h}\in\mathcal{A}(\mu,\nu_h)$
and \eqref{eq_semidiscrete_pi_h_support_distance}
follow as in the proof of
\Cref{proposition_continuous_pi_h}.
For \eqref{eq_formula_evaluation_disintegration},
first we show that
$\eta:\mathcal{B}(X\times Y) \to \mathbb{R}$ given by
\begin{equation}
\label{eq_definition_of_eta}
\eta [A]
= \sum_{i}{\int_{E_i}{
    \Big(\sum\nolimits_{j}{
        \tfrac{\pi_{ij}}{\mu_i}\,\delta_{y_j}[A_x]
    }\Big) \,d\mu(x)}}
\end{equation}
is a well-defined Borel measure and then we get
$\eta=\semidiscrete{\pi_h}$ by a standard application of
Dynkin's $\pi-\lambda$ Theorem
\cite[Theorem 1.4]{evans_gariepy}.

Fix $A \in \mathcal{B}(X\times Y)$ and define $f_{ij}:E_{i} \to \mathbb{R}$ by
$f_{ij}(x) = \delta_{y_j}[A_x]$.
Since $f_{ij}$ takes values in $\{0,1\}$, then $f_{ij}$ is Borel iff
$f^{-1}_{ij}(1)$ is a Borel subset of $E_i$.
Now define the Borel map $\phi_{ij}:E_{i} \to X\times Y$
by $\phi_{ij}(x) = (x,y_j)$ and observe that
\begin{align*}
f^{-1}_{ij}(1)
&= \{x\in E_i: \delta_{y_j}[A_x]=1\}
 = \{x\in E_i: (x,y_j) \in A\} \\
&= \{x\in E_i: (x,y_j) \in A \cap (E_i \times \{y_j\}) \} \\
&= \phi^{-1}_{ij} (A \cap (E_i\times \{y_j\})) \mspace{1mu}.
\end{align*}
Then $f^{-1}_{ij}(1)$ is Borel and so is $f_{ij}$.
It follows that the integrals
\[
\int_{E_i}{
    \sum\nolimits_{j}{\tfrac{\pi_{ij}}{\mu_i}\,\delta_{y_j}[A_x]} \, d\mu(x)
}
= \int_{E_i}{\sum\nolimits_{j}{\tfrac{\pi_{ij}}{\mu_i}f_{ij}(x)} \,d\mu(x)}
\]
are well-defined, finite and nonnegative and
so is the map $\eta$ defined by \eqref{eq_definition_of_eta}.

To show that $\eta$ is $\sigma$-additive consider disjoint sets
$(A_k)_{k\in\mathbb{N}} \subset \mathcal{B}(X\times Y)$ and notice that,
for any $(x,y) \in X\times Y$, we have
$\, \delta_y[(A_k)_x] = 1 \mspace{-2mu} \iff \mspace{-2mu} (x,y) \in A_k \,$.
Then, the disjointness of $(A_k)_{k}$ implies
$\sum_{k}{\delta_y[(A_k)_x]} \in \{0,1\}$
and in both cases the equality
\begin{equation}
\label{eq_sigma_additivity_key_identity}
\delta_y \Big[
    \mspace{1mu} \Big({\bigsqcup\nolimits_{k}{A_k}}\Big)_x  \mspace{1mu}
\Big]
= \sum\nolimits_{k}{ \delta_y [ (A_k)_x ] }
\end{equation}
becomes apparent.
The $\sigma$-additivity of $\eta$ follows from
\eqref{eq_sigma_additivity_key_identity}
and \eqref{eq_definition_of_eta}.
Note also that \eqref{eq_definition_of_eta}
implies $\eta(X\times Y)=1$,
so $\eta \in \mathcal{P}(X\times Y)$.

To prove $\semidiscrete{\pi_h} = \eta$ note first that
$\mathcal{B}(X\times Y)$ is the $\sigma$-algebra generated by
\[
\Pi = \{ E\times F : E\in \mathcal{B}(X) \, , \; F\in \mathcal{B}(Y) \},
\]
which is a $\pi$-system.
If we are able to show $\Pi \subset \Lambda$ for the $\lambda$-system
\[
\Lambda = \{
    A \in \mathcal{B}(X\times Y) : \semidiscrete{\pi_h} [A] = \eta [A]
\},
\]
then Dynkin's $\pi-\lambda$ Theorem implies
$\Lambda \supset \sigma(\Pi) = \mathcal{B}(X\times Y)$.
Therefore, it suffices to show $\Pi \subset \Lambda$
to obtain $\semidiscrete{\pi_h} = \eta$.
To this end, fix $E\times F \in \Pi$ and observe that
\begin{equation}\label{eq_identity_for_proof_eta_pi_tilde}
\mathbbm{1}_E(x) \cdot \Big(
    \sum\nolimits_{j}{\tfrac{\pi_{ij}}{\mu_i} \mspace{1mu} \delta_{y_j}[F]}
\Big)
=
\sum\nolimits_{j}{
    \tfrac{\pi_{ij}}{\mu_i} \mspace{1mu} \delta_{y_j} [(E\times F)_x]
}
\end{equation}
for $i\in I$ and $x\in X$, which implies
\begin{align*}
\eta[E\times F]
&\stackrel[\eqref{eq_definition_of_eta}]{}{=}
    \sum_{i}{\int_{E_i}{\Big(
        \sum\nolimits_{j}{\tfrac{\pi_{ij}}{\mu_i}\delta_{y_j} [(E\times F)_x] }
    \Big) \,d\mu(x)}} \\
&\stackrel[\eqref{eq_identity_for_proof_eta_pi_tilde}]{}{=}
    \sum_{i}{\int_{E_i}{\Big(
        \sum\nolimits_{j}{\tfrac{\pi_{ij}}{\mu_i} \delta_{y_j}[F]}
    \Big) \mathbbm{1}_E(x) \,d\mu(x)}} \\
&= \sum_{i}{\int_{E\cap E_i}{\Big(
        \sum\nolimits_{j}{\tfrac{\pi_{ij}}{\mu_i} \delta_{y_j}[F]}
    \Big) \,d\mu(x)}} \\
&= \sum_{i}{\Big(
        \sum\nolimits_{j}{\tfrac{\pi_{ij}}{\mu_i} \delta_{y_j}[F]
    \Big)} \mu[E\cap E_i] } \\
&= \sum_{ij}{ \tfrac{\pi_{ij}}{\mu_i} \delta_{y_j}[F] \mu[E\cap E_i] }
    = \sum_{ij}{ \tfrac{\pi_{ij}}{\mu_i} \mu|_{E_i}[E]  \delta_{y_j}[F] }
\stackrel[\eqref{eq_definition_semidiscrete_pi_h}]{}{=}
    \semidiscrete{\pi_h} [E\times F] \,. \qedhere
\end{align*}
\end{proof}

\begin{proposition}
\label{proposition_semidiscrete_pi_concentrate_around_T_opt}
Let $(X,Y,\mu,\nu,c)$ be a compact OT problem satisfying
\ref{strong_uniqueness_hypothesis}
with optimal transport map $T_*$.
If $(\pi_k)_k$ is an approximating sequence of discrete plans
for $(X,Y,\mu,\nu,c)$, then for every $\delta>0$ we have
\begin{equation}
\label{eq_pi_tilde_far_from_graph_goes_to_zero}
\semidiscrete{\pi_k} \big[ \{
    (x,y) \in X \times Y: \;d_Y(y, T_*(x)) \geq \delta
\} \big] \longconverges[k] 0\,.
\end{equation}
\end{proposition}

\begin{proof}
We have $\semidiscrete{\pi_k} \in \mathcal{A}(\mu,\nu_k)$ by
\Cref{proposition_semidiscrete_pi_h} and $\nu_k \longweaklyconverges[k] \nu$
by
\Cref{remark_weak_convergence_marginals}.
Then, it suffices to show
$\Koperator[\semidiscrete{\pi_k}] \longconverges[k] \Koperator[*]$
in order to deduce
\eqref{eq_pi_tilde_far_from_graph_goes_to_zero}
by applying item
\ref{proposition_stability_properties_OT_item_monge}
in
\Cref{proposition_stability_properties_OT}
to $(\semidiscrete{\pi_k})^{\vphantom{s}}_k$.
Further, since
\[
\big| \Koperator[\semidiscrete{\pi_k}] - \Koperator[*] \big|
\leq
\big| \Koperator[\semidiscrete{\pi_k}] - \Koperator[\pi_k] \big|
+
\big| \Koperator[\pi_k]  - \Koperator[*] \big|
\]
and
$\Koperator[\pi_k] \longconverges[k] \Koperator[*]$
by \Cref{introtheorem_conv_values_plans},
the previous reasoning tells us that
\eqref{eq_pi_tilde_far_from_graph_goes_to_zero}
follows if we prove
$\big|
    \Koperator[\semidiscrete{\pi_k}] - \Koperator[\pi_k]
\big| \longconverges[k] 0$.
To prove this last claim
fix an $h_k$-plan $\pi_k$ and,
working with generic notation
\eqref{eq_generic_notation_partitions},
observe that
\begin{align*}
\big| \Koperator[\semidiscrete{\pi_k}] - \Koperator[\pi_k] \big|
&= \left| \sum_{ij}{
    \int_{E_i\times F_j}{c(x,y) \,d\semidiscrete{\pi_k}(x,y)}  - c_{ij}\pi_{ij}
} \right|\\
&\stackrel[\eqref{eq_semidiscrete_pi_h_support_distance}]{}{\leq}
    \sum_{ij}{
        \int_{E_i\times F_j}{| c(x,y) - c_{ij} | \,d\semidiscrete{\pi_k}(x,y)}
}
\stackrel[\eqref{eq_semidiscrete_pi_h_support_distance}]{}{\leq}
    \sum_{ij}{\omega_c(h_k) \pi_{ij}} = \omega_c(h_k) .
\end{align*}
The proof is complete since we assume $h_k \longconverges[k] 0$ by
\Cref{def_approximating_sequence_discrete_plans}.
\end{proof}

\bigsubsection{Proof of Theorem \ref*{introtheorem_conv_maps}}

\begin{proposition}
\label{proposition_inequalities_dp_for_different_projections}
Let $(X,Y,\mu,\nu,c)$ be a compact OT problem and fix $T\in\mathcal{B}(X,Y)$
to define, for each $\delta>0$, the set
\begin{equation}
\label{eq_proposition_inequalities_dp_for_different_projections_definition_B_T}
B_{T}(\delta) = \big\{ %
    (x,y)\in X\times Y \,:\, d_Y(y,T(x)) \geq \delta
\mspace{1mu} \big\}. %
\end{equation}
Then, given $p \in [1,+\infty)$ and $\Qfact\geq 0$, there are nonnegative
constants $A$, $B$, $C$ depending only on $p$, $\Qfact$ and $\diam{Y}$
such that
\begin{equation}
\label{eq_inequality_dp_T_h_general_projection}
d_p(T_h,T) \leq A \cdot h + B \cdot \delta +
C \cdot \semidiscrete{\pi_h} \big[B_T(\delta)\big]^{\frac{1}{p}}
\end{equation}
for every feasible $h$-plan $\pi_h$ and
every projection $T_h \in \mathcal{T}^{\Qfact}[\pi_h]$.
\end{proposition}

\begin{proof}
Fix a map $T\in\mathcal{B}(X,Y)$,
a feasible $h$-plan $\pi_h$ and $\delta>0$.
For each $x\in X$, partition
$J$ into a bad set and a good one, respectively given by
\begin{equation}
B_{h,\delta}(x) = \big\{ j : d_Y(y_j, T(x)) \geq  \delta \big\}
\,\text{ and }\,
G_{h,\delta}(x) = \big\{ j : d_Y(y_j, T(x))   <   \delta \big\} .
\end{equation}
Observe that the bad sets $B_{h,\delta}(x)$ and $B_T(\delta)$ are linked
by the relation
\begin{equation}
\label{eq_key_relation_bad_sets}
j \in B_{h,\delta}(x)
\iff
(x,y_j) \in B_T(\delta)
\iff
y_j \in B_T(\delta)_x
\end{equation}
for every $j\in J$, where the subscript $x$ in $B_T(\delta)_x$
denotes a vertical slice of the set $B_T(\delta)$
as in \Cref{proposition_semidiscrete_pi_h}.

We will write the proof under the assumption
$T_h \in \mathcal{T}^{\Qfact}_{GM}[\pi_h]$, extending it to the general case
$T_h \in \mathcal{T}^{\Qfact}[\pi_h] =
\mathcal{T}^{\Qfact}_{GM}[\pi_h] \cup \mathcal{T}_{B}[\pi_h]$
by showing that the resulting constants $A,B,C$
also work for the case $T_h \in \mathcal{T}_{B}[\pi_h]$.
Thus, for $i\in I_+$ and $x \in E_i$, the aforesaid assumption implies that
$T_h(x)$ is a $\Qfact h$-approximate geometric
median for $(y_j)_j$ with weights $(\pi_{ij}/\mu_i)_j$.
Item
\ref{nearness_lemma_item_general}
in Nearness Lemma \ref{nearness_lemma} then yields
\begin{align}
\label{eq_proof_proposition_inequalities_dp_application_nearness_lemma}
d_Y(T_h(x),T(x)) \leq
\Qfact  h
    +
2 \sum\nolimits_{j}{
    \mathlarger{\tfrac{\pi_{ij}}{\mu_i}} \, d_Y(y_j,T(x))
}\,.
\end{align}

Fixing $i\in I_+$ and $x \in E_i$, the considerations above imply
\begin{equation*}
\begin{multlined}
\sum_{j\in J}{\mathlarger{\tfrac{\pi_{ij}}{\mu_i}} \, d_Y(y_j,T(x))}
=
\sum_{j\in G_{h,\delta}(x)}{\mathlarger{\tfrac{\pi_{ij}}{\mu_i}} \, d_Y(y_j,T(x))}
    + \sum_{j\in B_{h,\delta}(x)}{\mathlarger{\tfrac{\pi_{ij}}{\mu_i}} \, d_Y(y_j,T(x))} \\
\leq \,
    \delta \mspace{1mu} \sum\nolimits_{j\in G_{h,\delta}(x)}{\mathlarger{\tfrac{\pi_{ij}}{\mu_i}}}
    \;+\;
    D_Y  \mspace{1mu} \sum\nolimits_{j\in B_{h,\delta}(x)}{\mathlarger{\tfrac{\pi_{ij}}{\mu_i}}}
\,\leq\,
    \delta \;+\; D_Y \mspace{1mu} \sum\nolimits_{j\in B_{h,\delta}(x)}{\mathlarger{\tfrac{\pi_{ij}}{\mu_i}}}\,,
\end{multlined}
\end{equation*}
where $D_Y \defeq \diam{Y}$.
This inequality together with
\eqref{eq_proof_proposition_inequalities_dp_application_nearness_lemma}
yields
\begin{align}
\label{eq_proof_proposition_inequalities_dp_key_inequality_before_jensen}
d_Y(T_h(x),T(x)) \leq
\Qfact  h + 2\delta + 2 D_Y \mspace{-2mu} \sum\nolimits_{j\in B_{h,\delta}(x)}{\mathlarger{\tfrac{\pi_{ij}}{\mu_i}}}\,.
\end{align}
Then, by applying Jensen's inequality to the RHS of
\eqref{eq_proof_proposition_inequalities_dp_key_inequality_before_jensen}
we get
\begin{align}
\label{eq_proof_proposition_inequalities_dp_applying_jensen_inequality}
\begin{split}
d_Y^{\,p}(T_h(x),T(x))
&\leq
    \left( \Qfact  h + 2\delta + 2 D_Y \sum\nolimits_{j\in B_{h,\delta}(x)}{\mathlarger{\tfrac{\pi_{ij}}{\mu_i}}} \right)^p \\
&\leq
    3^{p-1}\left( (\Qfact  h)^p + (2\delta)^p + (2 D_Y)^p \left(\sum\nolimits_{j\in B_{h,\delta}(x)}{\mathlarger{\tfrac{\pi_{ij}}{\mu_i}}}\right) \right),
\end{split}
\end{align}
where the last inequality follows from the fact that
$\left(
    \sum_{j\in B_{h,\delta}(x)}{\mathlarger{\tfrac{\pi_{ij}}{\mu_i}}}
\right) \leq1$.

Now integrate
\eqref{eq_proof_proposition_inequalities_dp_applying_jensen_inequality}
over $E_i$ for fixed $i\in I_+$ to get
\begin{equation*}
\begin{multlined}
\int_{E_i}{d_Y^{\,p}(T_h(x),T(x)) \,d\mu(x)} \\
\qquad\leq 3^{p-1} \mu_i \left( (\Qfact  h)^p + (2\delta)^p \right)
+ 3^{p-1} (2 D_Y)^p \int_{E_i}{\left(\sum\nolimits_{j\in B_{h,\delta}(x)}{\mathlarger{\tfrac{\pi_{ij}}{\mu_i}}}\right) d\mu(x)}.
\end{multlined}
\end{equation*}
Recall that $\mu$ is concentrated on $\bigsqcup_{i\in I_+}{E_i}$,
as explained in
\Cref{remark_measure_partitioning_by_I_plus_and_J_plus}.
Therefore, summing the previous inequality over $i\in I_+$ we obtain
\begin{equation}
\begin{multlined}
\label{eq_proof_proposition_inequalities_dp_ineq_after_sum_over_I_plus}
\int_{X}{d_Y^{\,p}(T_h(x),T(x)) \,d\mu(x)} \\
\qquad\leq 3^{p-1} \left( (\Qfact  h)^p + (2\delta)^p \right)
+ 3^{p-1} (2 D_Y)^p \sum_{i\in I_+}{\int_{E_i}{\left( \sum\nolimits_{j\in B_{h,\delta}(x)}{\mathlarger{\tfrac{\pi_{ij}}{\mu_i}}}\right) d\mu(x)}}.
\end{multlined}
\end{equation}

The last step in the proof consists in verifying that the outer sum on the RHS of
\eqref{eq_proof_proposition_inequalities_dp_ineq_after_sum_over_I_plus}
agrees with $\semidiscrete{\pi_h} [B_T(\delta)]$.
Indeed, notice that for $i\in I_+$ and $x\in E_i$ we have
\begin{equation*}
\sum\nolimits_{j\in B_{h,\delta}(x)}{\mathlarger{\tfrac{\pi_{ij}}{\mu_i}}} =
\sum\nolimits_{j\in J}{\mathlarger{\tfrac{\pi_{ij}}{\mu_i}} \, \mathbbm{1}_{B_{h,\delta}(x)}(j) }
\stackrel[\eqref{eq_key_relation_bad_sets}]{}{=}
\sum\nolimits_{j\in J}{\mathlarger{\tfrac{\pi_{ij}}{\mu_i}} \, \delta_{y_j} \big[ B_T(\delta)_x \big] },
\end{equation*}
so disintegration identity
\eqref{eq_formula_evaluation_disintegration}
implies
\begin{equation*}
\sum_{i\in I_+}{\int_{E_i}{\mspace{-8mu}\Big(\sum\nolimits_{j\in B_{h,\delta}(x)}{\mathlarger{\tfrac{\pi_{ij}}{\mu_i}}}\Big) \mspace{1mu} d\mu(x)}}
=
\sum_{i\in I_+}{\int_{E_i}{\mspace{-8mu}\Big(\sum\nolimits_{j}{\mathlarger{\tfrac{\pi_{ij}}{\mu_i}} \, \delta_{y_j} [ B_T(\delta)_x ] }\Big) \mspace{1mu} d\mu(x)}}
\stackrel[\eqref{eq_formula_evaluation_disintegration}]{}{=}
\semidiscrete{\pi_h} [B_T(\delta)] \mspace{1mu} .
\end{equation*}
Plugging this into inequality
\eqref{eq_proof_proposition_inequalities_dp_ineq_after_sum_over_I_plus}
and taking the $p$-th root on both sides yields
\begin{equation}
\label{eq_last_inequality_d_p_semidiscrete_measure}
\begin{aligned}
d_p(T_h,T)
&\leq  \Big(  3^{p-1} \left( (\Qfact  h)^p + (2\delta)^p \right)
+ 3^{p-1} (2 D_Y)^p \; \semidiscrete{\pi_h} [B_T(\delta)] \Big)^{\frac{1}{p}}\\
&\leq  3^\frac{p-1}{p} \left( \Qfact  h + 2\delta \right)
+ 3^\frac{p-1}{p} (2 D_Y) \; \semidiscrete{\pi_h} [B_T(\delta)] ^{\frac{1}{p}}\, ,
\end{aligned}
\end{equation}
completing the proof of
\eqref{eq_inequality_dp_T_h_general_projection}
for the case $T_h \in \mathcal{T}^{\Qfact}_{GM}[\pi_h]$.
Note that after rewriting
\eqref{eq_last_inequality_d_p_semidiscrete_measure}
in the form \eqref{eq_inequality_dp_T_h_general_projection}
the resulting constants $A,B,C$ depend only on $p$, $\Qfact$ and $D_Y$.
In particular, they are independent of both $\pi_h$ and $T_h$.

If $T_h \in \mathcal{T}_{B}[\pi_h]$, then $(Y,d_Y)$ is assumed to be normed
and $T_h(x)$ is a barycenter for each $x\in E_i$ with $i\in I_+$.
Once again, Nearness Lemma \ref{nearness_lemma} implies
\begin{align*}
d_Y(T_h(x),T(x)) \leq
\sum\nolimits_{j}{\mathlarger{\tfrac{\pi_{ij}}{\mu_i}} \, d_Y(y_j,T(x))}
\end{align*}
for $i\in I_+$ and $x\in E_i$,
which improves \eqref{eq_proof_proposition_inequalities_dp_key_inequality_before_jensen} to
\begin{align*}
d_Y(T_h(x),T(x)) \leq
\delta + D_Y \mspace{-2mu} \sum\nolimits_{j\in B_{h,\delta}(x)}{\mathlarger{\tfrac{\pi_{ij}}{\mu_i}}}\,.
\end{align*}
Rewriting the proof above from this inequality instead of
\eqref{eq_proof_proposition_inequalities_dp_key_inequality_before_jensen}
shows that constants $A,B,C$ deduced from
\eqref{eq_last_inequality_d_p_semidiscrete_measure}
work also for $T_h \in \mathcal{T}_{B}[\pi_h]$.
\end{proof}

\begin{remark}
It follows from the previous proof that constants
\begin{equation}
\label{eq_inequality_dp_T_h_GM_Delta}
A = 3^{\frac{p-1}{p}} \Qfact
\;,\quad
B = 2 \cdot 3^{\frac{p-1}{p}}
\;,\quad
C = 2 \cdot 3^{\frac{p-1}{p}} \diam{Y}
\end{equation}
work for $T_h \in \mathcal{T}^{\Qfact}_{GM}[\pi_h]$.
Further, rewriting the proof shows
that constants
\begin{equation}
\label{eq_inequality_dp_T_h_GM}
A = 0
\;,\quad
B = 2 \cdot 2^{\frac{p-1}{p}}
\;,\quad
C = 2 \cdot 2^{\frac{p-1}{p}} \diam{Y}
\end{equation}
work for $T_h \in \mathcal{T}_{GM}[\pi_h]$, while for
$T_h \in \mathcal{T}_{B}[\pi_h]$
we get the constants
\begin{equation}
\label{eq_inequality_dp_T_h_B}
A = 0
\;,\quad
B = 2^{\frac{p-1}{p}}
\;,\quad
C = 2^{\frac{p-1}{p}} \diam{Y}.
\end{equation}
Note that constants
\eqref{eq_inequality_dp_T_h_GM_Delta}
work for $T_h \in \mathcal{T}^{\Qfact}[\pi_h]$
since they dominate
\eqref{eq_inequality_dp_T_h_GM}
and
\eqref{eq_inequality_dp_T_h_B}
for every $\Qfact \geq 0$.
Similarly, constants \eqref{eq_inequality_dp_T_h_GM} work for
$T_h \in \mathcal{T}[\pi_h]$.
\end{remark}

\begin{manualtheorem}{\ref*{introtheorem_conv_maps}}
\label{introtheorem_conv_maps_restatement}
Let $(X,Y,\mu,\nu,c)$ be a compact OT problem satisfying hypothesis
\ref{strong_uniqueness_hypothesis} with optimal transport $T_*$.
Then, if $(T_k)_k$ is an approximating sequence of projection maps,
we have $d_p(T_k,T_*)\longconverges[k]0$
for every $1\leq p < \infty$.
\end{manualtheorem}

\begin{proof}
According to Definitions
\ref{def_approximating_sequence_projection_maps}
and
\ref{def_approximating_sequence_discrete_plans},
there are a quality parameter $\Qfact \geq 0$ and
an approximating sequence of discrete plans $(\pi_k)_k$
such that $\pi_k$ is an $\varepsilon_k$-optimal $h_k$-plan with
$T_k \in \mathcal{T}^{\Qfact}[\pi_k]$ for all $k$, where
$(\varepsilon_k)_k$ and $(h_k)_k$ satisfy
$\varepsilon_k, h_k\longconverges[k]0$.
Fixing $1\leq p <\infty$, it follows from
\Cref{proposition_inequalities_dp_for_different_projections}
that
\begin{equation}
\label{eq_proof_introtheorem_conv_maps_restatement}
d_p(T_k,T_*) \leq A \cdot h_k + B \cdot \delta +
C \cdot \semidiscrete{\pi_k} [B_{T_*}(\delta)]^{\frac{1}{p}}
\end{equation}
for all $k$ and every $\delta>0$, where $B_{T_*}(\delta)$ is defined
as in
\eqref{eq_proposition_inequalities_dp_for_different_projections_definition_B_T}
by taking $T=T_*$.
Note that the same constants $A$, $B$, $C$ work for every $k$ and
every $\delta$, since they depend only on $p$, $\Qfact$ and $\diam{Y}$.
Now consider arbitrary $\varepsilon>0$ and fix $\delta>0$ with
$B\cdot\delta < \varepsilon$.
Notice that
$\semidiscrete{\pi_k} [B_{T_*}(\delta)] \longconverges[k] 0$
by \Cref{proposition_semidiscrete_pi_concentrate_around_T_opt},
so taking $\limsup_{k}$ on both sides of
\eqref{eq_proof_introtheorem_conv_maps_restatement}
yields $\limsup_{k}{\,d_p(T_k,T_*)} \leq \varepsilon$
and the proof is complete.
\end{proof}

\bigsubsection{Discussion}
\label{discussion_projection_maps}

We already showed that
\Cref{introtheorem_conv_maps_restatement}
fails for semicontinuous $c$
and also for $p=\infty$
in
\Cref{example_no_convergence_for_semicontinuous_cost}
and
\Cref{remark_conv_maps_fails_for_p_equals_infty}
respectively.
In order to keep on testing the sharpness of Theorem
\ref{introtheorem_conv_maps_restatement},
we may ask if it is possible to generalize its conclusion to the
case where an optimal transport map $T_*$ exists and satisfies
$\Koperator[*]=\Moperator[T_*]$ but fails to be unique.
As we did for Theorem
\ref{introtheorem_conv_values_plans_restatement},
we may define
\[
\minsetmonge = \big\{\mspace{2mu}
    T \in \mathcal{T}(\mu,\nu)
        \mspace{2mu}:\mspace{2mu}
    \Moperator[T] = \Moperator[T_*] = \Koperator[*]
\mspace{2mu}\big\}
\]
and then ask if $d_p(T_k,\minsetmonge) \longconverges[k] 0$.
It turns out that the answer is no, as shown by
\Cref{example_non_convergence_multiple_maps}
below.
Also, we may ask if the conclusion in Theorem
\ref{introtheorem_conv_maps_restatement}
still holds when both
\eqref{kantorovich_problem} and \eqref{monge_problem}
admit unique minimizers but $\Koperator[*] < \Moperator[T_*]$.
Once again, the answer is negative, as show by
\Cref{example_no_convergence_maps_KU_and_MU_but_not_SU}
below.
Note that constructing an example to illustrate this last case
necessarily requires taking $\mu$ with at least one atom.
This is because, as
explained in
\nameref{section_introduction},
for non-atomic $\mu$ we always have
$\min_{\mathcal{A}(\mu,\nu)}{\Koperator}
= \inf_{\mathcal{T}(\mu,\nu)}{\Moperator}$
and then $\Koperator[*] = \Moperator[T_*]$ follows.

\begin{example}
\label{example_non_convergence_multiple_maps}
Take $X,Y,\mu,\nu$ as in
\Cref{example_only_weak_convergence_for_pi_k}
but now consider the cost function $c(x,y)=(x-y)^2(1-x-y)^2$.
Then the zero set of $c$ is
\[
Z_c =  \big\{ (x,x) : x \in [0,1] \big\}
\;\cup\;
\big\{ (x,1-x) : x \in [0,1] \big\}
\subset [0,1]\times[0,1]
\]
and there is no unique optimal transport map $T_*$.
Indeed, the maps $T_\varepsilon$ defined by
\begin{align*}
T_\varepsilon(x) =
    \begin{cases}
        \; x       &, \; \text{if} \;\; \varepsilon < x < 1-\varepsilon \\
        \; 1 - x       &, \; \text{otherwise}
    \end{cases}
\end{align*}
clearly satisfy $\Moperator[T_\varepsilon]=0=\Koperator[*]$
and are easily seen to be in $\mathcal{T}(\mu,\nu)$ for every $\varepsilon$.

Fix $k\in\mathbb{N}$ and consider the same $h_k$-partitions as in
\Cref{example_only_weak_convergence_for_pi_k}.
Take $\pi_k$ to be the ``symmetrical'' $h_k$-plan concentrated
on $Z_c$ and defined by
\begin{equation*}
\pi_k = \frac{1}{2}
\left[ \mspace{3mu}
\sum_{i=1}^{2k+1}{\tfrac{1}{2k+1}\mspace{2mu}\delta_{(x_i,y_i)}}
+
\sum_{i=1}^{2k+1}{\tfrac{1}{2k+1}\mspace{2mu}\delta_{(x_i,y_{2k+2-i})}}
\right],
\end{equation*}
which is illustrated in
\Cref{figure_example_non_convergence_multiple_maps}.
By symmetry, we can take
$T_k \in \mathcal{T}[\pi_k]$ to be the constant $\nicefrac{1}{2}$.
As the graph of any $T\in\minsetmonge$ is essentially contained
in $Z_c$, we get
\[
d_p(T_k,T) = \big\|T_k-T\big\|_{L^p[0,1]}
\geq \big\|\nicefrac{1}{2}-T\big\|_{L^p[\nicefrac{3}{4},\, 1]}
\geq \tfrac{1}{16}
\]
for $p\in[1,\infty]$. It follows that $d_p(T_k,\minsetmonge)\geq \tfrac{1}{16}$
and repeating this construction for every $k$ yields $(T_k)_k$ with
$d_p(T_k,\minsetmonge) \centernot \longrightarrow 0$ for every $p\in[1,\infty]$.
\end{example}

\begin{example}
\label{example_no_convergence_maps_KU_and_MU_but_not_SU}
Take $X=Y=[-2,2]$ and the quadratic cost $c(x,y)=(x-y)^2$.
If $\mathcal{H}^1$ is the $1$-dimensional Hausdorff measure, let
\[
\mu = \tfrac{1}{2}\mspace{2mu}\delta_{-2}
    + \tfrac{1}{2}\mspace{2mu}\mathcal{H}^1\rvert_{[-1,0]}
\;\;\text{ and }\;\;
\nu = \tfrac{1}{2}\mspace{2mu}\mathcal{H}^1\rvert_{[0,1]}
    + \tfrac{1}{2}\mspace{2mu}\delta_{2}.
\]
Then both
\eqref{kantorovich_problem}
and
\eqref{monge_problem}
have unique minimizers $\pi_*$ and $T_*$ but $\pi_* \neq \pi_{T_*}$, so
\ref{strong_uniqueness_hypothesis} fails.
To see why note first that feasible maps for \eqref{monge_problem} send atoms to
atoms. Then, it follows from monotonicity considerations that $T_*$ is given by
\[
T_*(x) = 2 \cdot \mathbbm{1}_{\{-2\}}(x)
    + (x+1) \cdot \mathbbm{1}_{[-1,0]}(x)
\;\text{ and }\;
\Moperator[T_*] = \Koperator[\pi_{T_*}] = \tfrac{17}{2}.
\]
By monotonicity considerations again,
the optimal plan for \eqref{kantorovich_problem} is
\[
\pi_* = \tfrac{1}{2}\mspace{2mu}\mathcal{H}^1\vert_{ \{-2\} \times [0,1] }
    + \tfrac{1}{2}\mspace{2mu}\mathcal{H}^1\vert_{ [-1,0] \times \{2\} }
\;\text{ and }\;
\Koperator[\pi_*] = \tfrac{19}{3} < \Moperator[T_*].
\]
It is then clear that  $\pi_* \neq \pi_{T_*}$,
as shown in
\Cref{figure_example_no_convergence_maps_KU_and_MU_but_not_SU}.
Notice that projection maps $T_k$ are obtained by working with an approximation
\eqref{kantorovich_discrete_problem} of \eqref{kantorovich_problem},
so it is reasonable to expect that measures $\pi_{T_k}$ approximate
the barycentric projection of $\pi_*$.
In fact, it is easy to check that $(T_k)_k$ converges $\mu$-a.e. to the map
$T$ given by
\[
T(x) = \tfrac{1}{2} \cdot \mathbbm{1}_{\{-2\}}(x)
    + 2 \cdot \mathbbm{1}_{[-1,0]}(x),
\]
which represents the barycentric projection of $\pi_*$.
Then, since $\mu[\{T \neq T_*\}] = 1$ and $T_k \longconverges[k] T$ $\mu$-a.e.,
it follows that
$d_p(T_k,T_*) \centernot \longrightarrow 0$.
\end{example}

\begin{figure}
\centering
\subfigure[\Cref{example_non_convergence_multiple_maps}]{
    \label{figure_example_non_convergence_multiple_maps}
    \includegraphics[width=0.45\textwidth]{%
        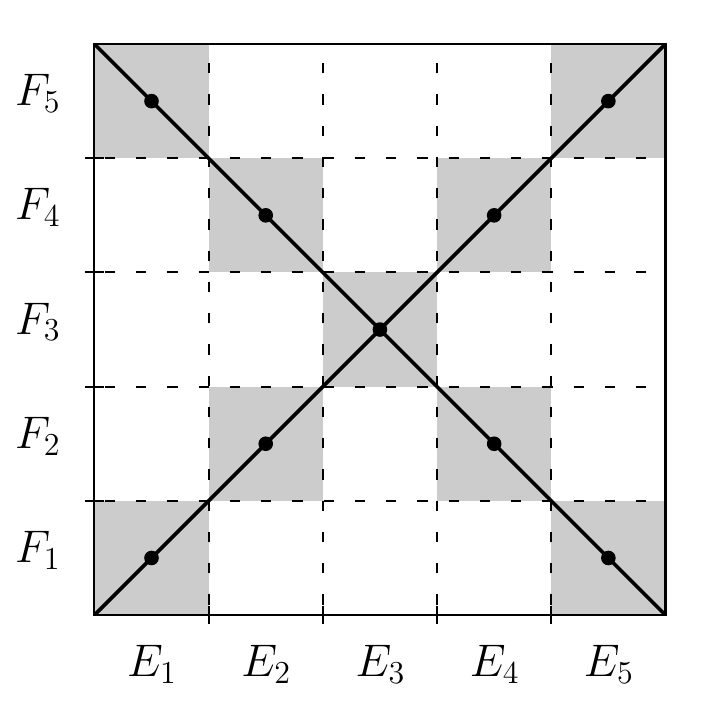%
        }
}\hspace{0mm}%
\subfigure[\Cref{example_no_convergence_maps_KU_and_MU_but_not_SU}]{
    \label{figure_example_no_convergence_maps_KU_and_MU_but_not_SU}
    \includegraphics[width=0.45\textwidth]{%
        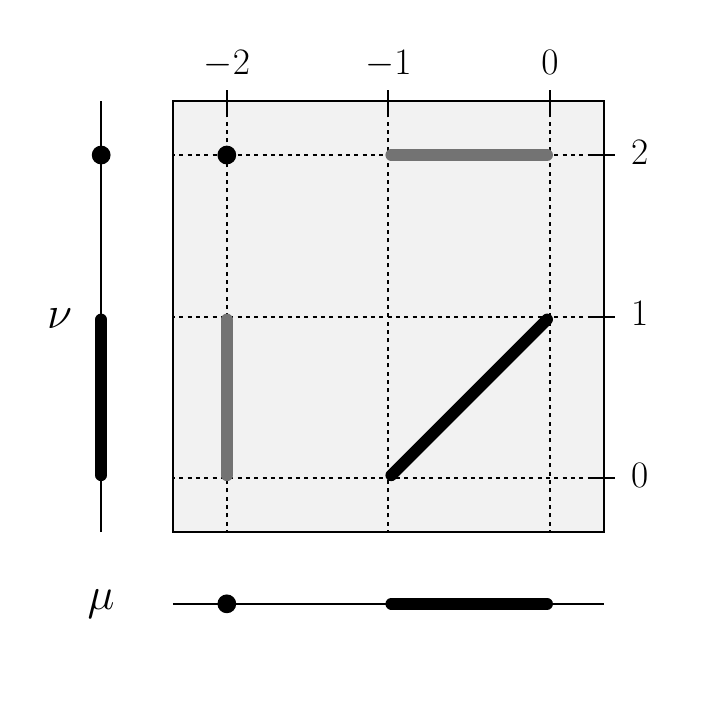%
    }
}
\captionsetup{width=0.9\linewidth}
\caption{
In \ref*{sub@figure_example_non_convergence_multiple_maps}
the thick lines show the zero set $Z_c$ of the cost $c$ and the dots
inside the gray squares show $\supp{\pi_k}$.
In \ref*{sub@figure_example_no_convergence_maps_KU_and_MU_but_not_SU}
the supports of $\mu$, $\nu$ and $\pi_{T*}$ are plotted in black and
the support of $\pi_*$ is plotted in dark gray.
}
\end{figure}

\section{Convergence of Projection Maps II}
\label{section_convergence_projection_maps_II}

Now we focus on
\Cref{introcorollary_fully_discrete_version_convergence_maps},
whose content is clarified below.

\begin{definition}
Consider an $h$-partition $\mathcal{C}_X = \{(E_i,x_i)\}_{i\in I}$ for $X$
and let $\mu_i = \mu[E_i]$ for all $i\in I$.
For each $1 \leq p < \infty$,
a pseudometric $\operatorname{disc}_p$ on $Y^X$ is defined by
\[
\operatorname{disc}_p(T_1,T_2)
= \left(\sum\nolimits_{i\in I}{
    \mu_i \mspace{2mu} d^{\,p}_Y(T_1(x_i),T_2(x_i))
}\right)^{1/p}
\;\text{ for $T_1, T_2 \in Y^X$.}
\]
\end{definition}

\begin{remark}
Since $\operatorname{disc}_p(T_1,T_2)$ is undefined unless some
reference $h$-partition $\mathcal{C}_X$ is given,
we stick to the following convention:
given a projection $T_h$,
$\operatorname{disc}_p(T_h,T)$ will always be evaluated with
respect to the $h$-partition used for constructing $T_h$.
Thus, if each projection map $T_k$ in
Corollary \ref{corollary_equivalence_discrete_continuous}
is constructed from
$\mathcal{C}^{\,k}_X = \big\{ \big( E_i^k, x_i^k \big) \big\}_{i\in I_k}$,
we write $\mu\big[E_i^k\big] = \mu_i^k$ for all $i \in I_k$ and then we have
\begin{equation}
\label{eq_disc_p_is_a_riemann_sum}
\operatorname{disc}_p(T_k,T)
= \left(\sum\nolimits_{i\in I_k}{
    \mu_i^k \mspace{2mu}
        d^{\,p}_Y \big( T_k\big(x_i^k\big) , T\big(x_i^k\big) \big)
}\right)^{1/p}
\end{equation}
for every $k$.
\end{remark}

\begin{manualcorollary}{\ref*{introtheorem_conv_maps}}[Reformulation]
\label{corollary_equivalence_discrete_continuous}
Assume $(X,Y,\mu,\nu,c)$ is a compact OT problem satisfying
\ref{strong_uniqueness_hypothesis} and let $T_*:X\to Y$
be an optimal transport map.
Then, the following statements are equivalent:
\begin{enumerate}[%
label=\roman*), ref=\emph{\roman*)}, topsep=0pt, itemsep=5pt, leftmargin=35pt ]
\item \label{corollary_equivalence_discrete_continuous_item_mu}
The discontinuity set of $T_*$ is a null set for $\mu$,
that is, $\mu[D(T_*)] = 0$.
\item \label{corollary_equivalence_discrete_continuous_item_disc}
It holds that $\operatorname{disc}_p(T_k,T_*)\longconverges[k] 0$
for every $1 \leq p < \infty$ whenever $(T_k)_k$ is an
approximating sequence of projection maps.
\end{enumerate}
\end{manualcorollary}

\begin{remark}
Under the assumptions of
Corollary \ref*{corollary_equivalence_discrete_continuous},
if $\mu[D(T_*)]>0$ and a particular sequence $(T_k)_k$ is given,
then $(T_k)_k$ may or may not converge to $T_*$ in $\operatorname{disc}_p$.
This fact is illustrated by \Cref{example_non_convergence_disc_p} below.
\end{remark}

\begin{example}
\label{example_non_convergence_disc_p}
For $X=Y=[0,1]$ consider $c(x,y)=|x-y|^2$ and take
\[
\mu = \tfrac{1}{2}\mspace{2mu}\delta_{0}
    + \tfrac{1}{2}\mspace{2mu}\mathcal{H}^1\rvert_{[0,1]}
\;\;\text{ and }\;\;
\nu = \tfrac{1}{2}\mspace{2mu}\delta_{0}
    + \tfrac{1}{2}\mspace{2mu}\delta_{1}.
\]
Reasoning as in
\Cref{example_no_convergence_maps_KU_and_MU_but_not_SU}
we see that $(X,Y,\mu,\nu,c)$ satisfies
\ref{strong_uniqueness_hypothesis}
and admits an optimal transport map $T_*$ given by
\[
T_*(x) = 0 \cdot \mathbbm{1}_{\{0\}}(x)
    + 1 \cdot \mathbbm{1}_{(0,1]}(x)
= \mathbbm{1}_{(0,1]}(x).
\]
Even though
\Cref{introtheorem_conv_maps}
implies $d_p(T_k,T_*)\longconverges[k] 0$ for every
approximating sequence of projections $(T_k)_k$,
Corollary \ref{corollary_equivalence_discrete_continuous}
suggests that we cannot expect
$\operatorname{disc}_p(T_k,T_*)\longconverges[k] 0$
since $D(T_*)=\{0\}$ and then $\mu[D(T_*)]=\nicefrac{1}{2}>0$.
We illustrate this below.

Fixing $k\in \mathbb{N}$, take $h_k= k^{-1}$ and consider $h_k$-partitions
$\mathcal{C}_X = \mathcal{C}_Y$ given by
\begin{equation}
\label{example_disc_p_partitions}
E_i = F_i = \big[\tfrac{i-1}{k}, \tfrac{i}{k}\big]
\;\text{ and }\;
x_i = y_i =
    \tfrac{1}{2}\big[\tfrac{i-1}{k} + \tfrac{i}{k}\big] =\tfrac{2i-1}{2k}
\;\;\text{ for $1\leq i \leq k$.}
\end{equation}
Note that the weights $\mu_i=\mu[E_i]$ and $\nu_j=\nu[F_j]$ satisfy
\[
\mu_1 = \tfrac{1}{2} + \tfrac{1}{2k}  \,,\;\;
\mu_2 = \dots = \mu_k = \tfrac{1}{2k} \,,\;\;
\nu_1 = \nu_k = \tfrac{1}{2}
\;\text{ and }\;
\nu_2 = \dots = \nu_{k-1} = 0.
\]
Now consider the optimal $h_k$-plan
$\pi_k = (\pi_{ij})_{ij} \in \mathcal{A}(\mu_{h_k},\nu_{h_k})$
given by
\[
\pi_k =
\tfrac{1}{2}\mspace{2mu}\delta_{(x_1,y_1)}
+
\tfrac{1}{2k}\mspace{2mu}\delta_{(x_1,y_k)}
+
\sum\nolimits_{2\leq i \leq k}{\tfrac{1}{2k}\mspace{2mu}\delta_{(x_i,y_k)}}
\]
and illustrated in
\Cref{figure_example_non_convergence_disc_p_plans}.
Extract a projection $T_k \in \mathcal{T}[\pi_{k}]$ and, as shown in
\Cref{figure_example_non_convergence_disc_p_maps},
note that $T_k$ is far from $T_*$ on $E_1\backslash\{0\}$ for large $k$.
This is a consequence of the fact that
$\nicefrac{1}{2k} = \pi_{1k} \altll \pi_{11} = \nicefrac{1}{2}$ for $k\altgg 1$.
Indeed, for $x\in E_1$,
Nearness \Cref{nearness_lemma}
and
\Cref{definition_projection_maps}
together imply
\begin{align*}
|T_k(x)| &= |T_k(x_1)| = d_Y(0,T_k(x_1)) \leq
2 \left[
\tfrac{\pi_{11}}{\mu_1} \mspace{2mu} d_Y(0,y_1) +
\tfrac{\pi_{1k}}{\mu_1} \mspace{2mu} d_Y(0,y_k)
\right] \\
&= 2 \left[
\tfrac{\nicefrac{1}{2}}
    {\nicefrac{1}{2} \mspace{4mu} + \mspace{4mu}\nicefrac{1}{2k}}
\mspace{4mu} \big| 0 - \tfrac{1}{2k} \big|
+
\tfrac{\nicefrac{1}{2k}}
    {\nicefrac{1}{2} \mspace{4mu} + \mspace{4mu}\nicefrac{1}{2k}}
\mspace{4mu} \big| 0 - \tfrac{2k-1}{2k} \big|
\right]
\leq \tfrac{3}{k}.
\end{align*}
Therefore, for $1\leq p < \infty$ we get
\begin{equation}
\label{eq_example_disc_p_inequality_for_disc_p}
\operatorname{disc}_p(T_k,T_*)
\geq \left( \mu_1 \big|  T_k(x_1) - T_*(x_1) \big|^p \right)^{1/p}
\mspace{-5mu} =
    \big(
        \tfrac{1}{2} \mspace{-3mu} + \mspace{-3mu} \tfrac{1}{2k}
    \big)^{1/p}
    \big| T_k(x_1) \mspace{-3mu} - \mspace{-3mu}  1 \big|
\geq \tfrac{1}{2} \big( 1 \mspace{-3mu} - \mspace{-3mu} \tfrac{3}{k} \big).
\end{equation}
Repeating the construction above for every $k\in\mathbb{N}$ to obtain $(T_k)_k$,
inequality
\eqref{eq_example_disc_p_inequality_for_disc_p}
implies
$\liminf_k{\operatorname{disc}_p(T_k,T_*)} > 0$ and then
$\operatorname{disc}_p(T_k,T_*) \centernot \longconverges 0$.
It is easy to see that we do get
$\operatorname{disc}_p(T_k,T_*) \longconverges[k] 0$
if we repeat the construction above but taking $x_1 = 0$
instead of doing it as in
\eqref{example_disc_p_partitions}.
As explained by
Corollary \ref{corollary_equivalence_discrete_continuous},
the convergence of $(T_k)_k$ to $T_*$ in
$\operatorname{disc}_p$
depends on the choice of $(T_k)_k$ because
$\mu[D(T_*)]>0$.
\end{example}

\begin{figure}
\centering
\subfigure[$\pi_k$ in \Cref{example_non_convergence_disc_p}]{
    \label{figure_example_non_convergence_disc_p_plans}
    \includegraphics[width=0.46\textwidth]{%
        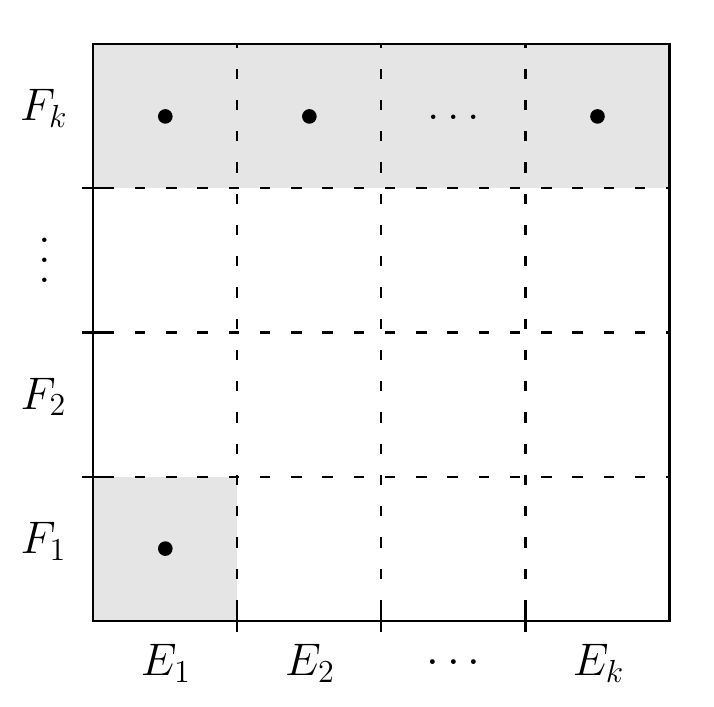%
        }
}\hspace{0mm}%
\subfigure[$T_k,\, T_*$ in \Cref{example_non_convergence_disc_p}]{
    \label{figure_example_non_convergence_disc_p_maps}
    \includegraphics[width=0.45\textwidth]{%
        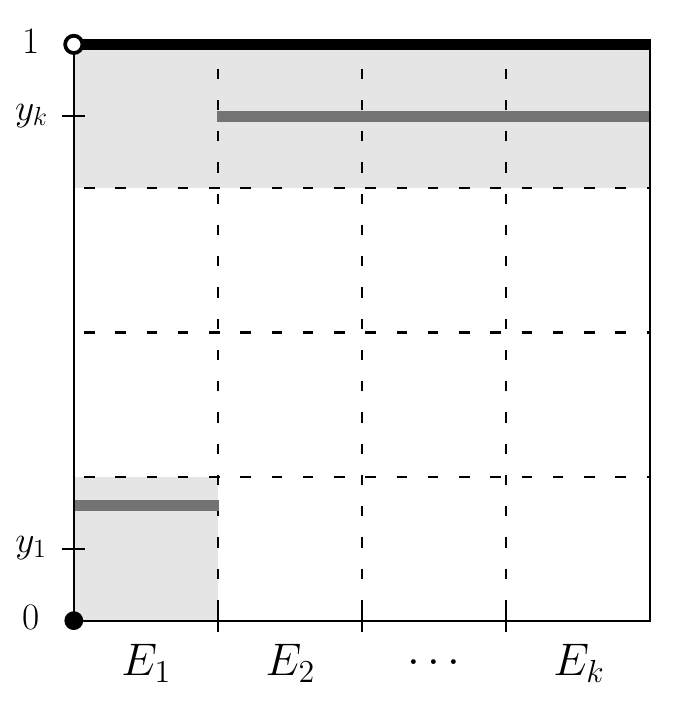%
    }
}
\captionsetup{width=0.9\linewidth}
\caption{
In \ref*{sub@figure_example_non_convergence_disc_p_plans} the dots show
$\supp{\pi_k}$.
In \ref*{sub@figure_example_non_convergence_disc_p_maps} the graph of
$T_*$ is plotted in black and the graph of $T_k$ is plotted in dark gray.
}
\end{figure}

\subsection*{The Proof of Corollary %
    \ref*{introcorollary_fully_discrete_version_convergence_maps}}

Looking at
\eqref{eq_disc_p_is_a_riemann_sum}
it is clear that the sum in
$\operatorname{disc}_p(T_k,T)$
may be regarded as a Riemann sum for the integral in $d_p(T_k,T)$.
In fact, since we have
$d_p(T_k,T_*) \approx 0$ for large $k$ by
\Cref{introtheorem_conv_maps},
what
Corollary \ref{corollary_equivalence_discrete_continuous}
says is that
\begin{equation}
\label{eq_interpretation_corollary_B_riemann_sum}
\sum\nolimits_{i\in I_k}{
    \mu_i^k \mspace{2mu}
        d^{\,p}_Y \big( T_k\big(x_i^k\big) , T_*\big(x_i^k\big) \big)
}
\approx
\int_{X}{d_Y^{\,p}(T_k(x),T_*(x))\,d\mu(x)}
\end{equation}
for large $k$ iff $\mu[D(T_*)] = 0$,
which is reminiscent of Lebesgue's criterion for Riemann integrability.
Moreover, Lebesgue's original proof
\cite[Chapitre II]{lebesgue_book_criterion},
available in
\cite[Section 7.26]{apostol_mathematical_analysis}
in a modern presentation,
shares with our proof of
Corollary \ref*{corollary_equivalence_discrete_continuous}
the fact that both
depend on the notion of \emph{oscillation} for maps, which we recall below.

In comparison with that of Lebesgue's criterion, the proof of
Corollary \ref{corollary_equivalence_discrete_continuous}
is much more complicated for two main reasons.
First, because in our case there are two limiting processes involved since
both the partition and the integrand $d_Y^{\,p}(T_k,T_*)$ in
\eqref{eq_interpretation_corollary_B_riemann_sum}
vary with $k$.
Second, because the very definition of projection maps makes it hard to control
$d_Y(T_k,T_*)$ pointwise, in particular for implication
\ref*{corollary_equivalence_discrete_continuous_item_disc}
$\mspace{-8mu}\implies\mspace{-8mu}$
\ref*{corollary_equivalence_discrete_continuous_item_mu}.
For the sake of clarity we split the proof
in two parts:
first we show the sufficiency of $\mu[D(T_*)] = 0$
for the convergence of $(T_k)_k$ to $T_*$ in $\operatorname{disc}_p$
and then we show its necessity.
The proofs of both implications consist in a sequence of technical lemmas
concerning the oscillation of $T_*$
plus an application of \Cref{introtheorem_conv_maps}.

\subsection*{Oscillation of a Map}

For general metric spaces $(X,d_X)$, $(Y,d_Y)$ let
\[
B_0(x)=\emptyset=B_0(y)
\,,\;
B_\infty(x)=X
\;\text{ and }\;
B_\infty(y) = Y
\;\text{ for all $x\in X$, $y\in Y$. }\;
\]
Given a map $T: X\to Y$, for $E\subset X$ we define the
\emph{oscillation of $T$ on $E$} by
\[
\operatorname{osc}_T(E)
= \diam{T(E)}
= \sup{ \big\{ d_Y(T(x_1), T(x_2))  \,:\, x_1,x_2 \in E \mspace{2mu} \big\}}
\]
if $E$ is nonempty and set $\operatorname{osc}_T(\emptyset) = 0$.
For $E_1 \subset E_2 \subset X$ we get
\[
0 = \operatorname{osc}_T(\emptyset)
\leq \operatorname{osc}_T(E_1)
\leq \operatorname{osc}_T(E_2)
\leq \diam{Y} \leq \infty.
\]
For a point $x\in X$ we define the
\emph{oscillation of $T$ at $x$} by
\[
\operatorname{osc}_T(x)
= \inf\nolimits_{r>0}{\operatorname{osc}_T(B_r(x))}
= \lim\nolimits_{r\to 0}{\operatorname{osc}_T(B_r(x))}.
\]
Note that, in general, we have
$\operatorname{osc}_T(x) \geq \operatorname{osc}_T(\{x\}) = 0$
and $\operatorname{osc}_T(x) \leq \operatorname{osc}_T(U)$ for every
neighborhood $U$ of $x$.
It is clear that $T$ is continuous at $x$ iff
$\operatorname{osc}_T(x)=0$, so its discontinuity set $D(T)$
is related to $\operatorname{osc}_T$ by
\begin{equation}
\label{eq_oscillation_implies_D(T)_is_F_sigma}
D(T) = \big\{
x\in X \,:\,\operatorname{osc}_T(x) > 0
\mspace{2mu}\big\}
=\bigcup\nolimits_{n\in \mathbb{N}}{D_{1/n}(T)},
\end{equation}
where
$D_{\alpha}(T) = \{x\in X \,:\,\operatorname{osc}_T(x) \geq \alpha \}$
for every $\alpha \geq 0$.
It is easy to see that each $D_{\alpha}(T)$ is a closed set,
so
\eqref{eq_oscillation_implies_D(T)_is_F_sigma}
proves that $D(T)$ is an $F_\sigma$ set.
This shows that $D(T)$ is Borel and then the condition $\mu[D(T_*)]=0$ in
Corollary \ref*{corollary_equivalence_discrete_continuous}
is well-defined.

\bigsubsection{Proof of Corollary B : Sufficiency}

\begin{lemma}
\label{lemma_phi_K_is_continuous}
Consider metric spaces $(X,d_X)$ and $(Y,d_Y)$, a map
$T:X\to Y$ and a nonempty compact set $K\subset X\backslash D(T)$.
Then $\phi_K:[0,\infty] \to [0,\infty]$ defined by
\[
\phi_K(r) =
    \sup_{x\in K}{ \, \operatorname{osc}_T(B_r(x))}
\]
is continuous at $0$. That is,
$\phi_K(r)\to 0$ for $r\to 0$.
\end{lemma}

\begin{proof}
Fix arbitrary $\varepsilon>0$ and for each $x\in K$ take $r_x>0$
with $\operatorname{osc}_T(B_{r_x}(x)) < \varepsilon$, which is possible
since $T$ is continuous at $x$.
The open cover $\{B_{r_x}(x)\}_{x\in K}$ for the compact set $K$ has a
Lebesgue number $\delta>0$ and we take $r_0 = \delta/2>0$.
Note that for any $r\in (0,r_0)$ and any $x\in K$ we have
$\diam{B_{r}(x)}<\delta$, so there is $\tilde{x}\in K$ with
$B_{r}(x) \subset B_{r_{\tilde{x}}}(\tilde{x})$ and then
\[
\operatorname{osc}_T(B_r(x))
= \diam{T(B_{r}(x))}
\leq \diam{T(B_{r_{\tilde{x}}}(\tilde{x}))}
= \operatorname{osc}_T(B_{r_x}(\tilde{x}))
< \varepsilon.
\]
This implies $\phi_K(r)\leq\varepsilon$ for $r\in (0,r_0)$ and completes the
proof.
\end{proof}

\begin{lemma}
\label{lemma_estimation_of_oscillation_for_sufficiently_fine_partitions}
Consider compact metric spaces $(X,d_X)$, $(Y,d_Y)$ and $\mu\in\mathcal{P}(X)$.
Let $T:X\to Y$ be such that $\mu[D(T)]=0$.
Then, given $1\leq p < \infty$ and $\varepsilon>0$, there is $h_0>0$
with the following property:
whenever $\mathcal{C}_X = \{(E_i, x_i)\}_{i\in I}$ is an $h$-partition for
$X$ with $h\leq h_0$, we get
\begin{equation}
\label{eq_lemma_estimation_of_oscillation_for_sufficiently_fine_partitions}
\left(
\sum\nolimits_{i\in I}{
    \mspace{2mu} \mu_i \mspace{1mu} \operatorname{osc}_T^{\,p}(E_i)
}\right)^{1/p} \leq \varepsilon.
\end{equation}
\end{lemma}

\begin{proof}
Fix $p,\varepsilon$ and write $D_Y = \diam{Y}<\infty$. Notice that
$\mu[X\backslash D(T)]=1$, so the inner regularity of $\mu$ implies the
existence of a compact set $K\subset X\backslash D(T)$ with
\begin{equation}
\label{eq_proof_oscillation_T_bound_X_minus_K}
\mu[X\backslash K] <
\left(\tfrac{\varepsilon}{2}\right)^{p} (D_Y^{\,p} + 1)^{-1}.
\end{equation}
Now apply \Cref{lemma_phi_K_is_continuous} to get
$r_0>0$ such that $\phi_K(r)<\varepsilon/2$ for $r\leq r_0$.
We finish the proof by showing that $h_0 = r_0/2$ works.

Fix an $h$-partition $\mathcal{C}_X$ for $X$ with $h\leq h_0$ and define
the index set
\[
I_K = \big \{ \mspace{1mu}
    i\in I \,:\, E_i \cap K \neq \emptyset \mspace{2mu}
\big \} \subset I.
\]
First, consider the case $i\in I_K$ and take $x_0 \in E_i\cap K$.
Note that every $x\in E_i$ satisfies
$d_X(x,x_0) \leq \diam{E_i} \leq h$,
so we get $E_i \subset B_{2h}(x_0) \subset B_{r_0}(x_0)$.
It follows that
\[
\operatorname{osc}_T(E_i)
\leq \operatorname{osc}_T(B_{r_0}(x_0))
\leq \phi_K(r_0) < \tfrac{\varepsilon}{2}
\]
and repeating the reasoning for all $i\in I_K$ leads us to the bound
\begin{equation}
\label{eq_proof_oscillation_T_bound_sum_I_in_K}
\sum\nolimits_{i\in I_K}{
    \mspace{2mu} \mu_i \mspace{1mu} \operatorname{osc}_T^{\,p}(E_i)
}
\leq
\left(\tfrac{\varepsilon}{2}\right)^{p}
    \left( \mathsmaller{\sum\nolimits_{i\in I_K}{\mspace{2mu} \mu_i}} \right)
\leq
\left(\tfrac{\varepsilon}{2}\right)^{p}.
\end{equation}
Now notice that $E_i \subset X\backslash K$ for $i \centernot \in I_K$,
which yields
\[
\sum\nolimits_{i\centernot \in I_K}{\mspace{2mu} \mu_i}
=
\sum\nolimits_{i\centernot \in I_K}{\mspace{2mu} \mu[E_i]}
=
\mu \left[ \mathsmaller{\bigsqcup\nolimits_{i\centernot \in I_K}{E_i}} \right]
\leq
\mu[X\backslash K]
\stackrel[\eqref{eq_proof_oscillation_T_bound_X_minus_K}]{}{<}
\left(\tfrac{\varepsilon}{2}\right)^{p} (D_Y^{\,p} + 1)^{-1}
\]
and allows us to conclude
\begin{equation}
\label{eq_proof_oscillation_T_bound_sum_I_not_in_K}
\sum\nolimits_{i\centernot \in I_K}{
    \mspace{2mu} \mu_i \mspace{1mu} \operatorname{osc}_T^{\,p}(E_i)
}
\leq
D_Y^{\,p} \left( \mathsmaller{
    \sum\nolimits_{i\centernot \in I_K}{\mspace{2mu} \mu_i}
} \right)
\leq
\left(\tfrac{\varepsilon}{2}\right)^{p}.
\end{equation}
Finally, the bounds
\eqref{eq_proof_oscillation_T_bound_sum_I_in_K}
and
\eqref{eq_proof_oscillation_T_bound_sum_I_not_in_K}
together yield
\[
\left(
\sum\nolimits_{i\in I}{
    \mspace{2mu} \mu_i \mspace{1mu} \operatorname{osc}_T^{\,p}(E_i)
}
\right)^{1/p}
\leq
\left(
    \left(\tfrac{\varepsilon}{2}\right)^{p}
        + \left(\tfrac{\varepsilon}{2}\right)^{p}
\right)^{1/p}
\leq
    \tfrac{\varepsilon}{2} + \tfrac{\varepsilon}{2}
= \varepsilon. \qedhere
\]
\end{proof}

\begin{lemma}
\label{lemma_disc_p_controlled_by_dp_and_sum_of_weighted_oscillations}
Consider a compact OT problem $(X,Y,\mu,\nu,c)$ and a projection map $T_h$
constructed from an $h$-partition $\mathcal{C}_X = \{(E_i, x_i)\}_{i\in I}$
for $X$. Then, for every map $T:X \to Y$ and every $1 \leq p < \infty$, we get
\begin{equation}
\label{eq_estimate_disc_p_T_h_T}
\operatorname{disc}_p(T_h, T) \leq   2 \mspace{2mu} d_p(T_h,T) +
2 \mspace{0mu} \left(
\sum\nolimits_{i\in I}{
    \mspace{2mu} \mu_i \mspace{1mu} \operatorname{osc}_T^{\,p}(E_i)
} \right)^{1/p} .
\end{equation}
\end{lemma}

\begin{proof}
For $i\in I$ and $x\in E_i$, Jensen's inequality implies
\begin{align*}
d_Y^{\,p}(T_h(x_i), T(x_i))
&\leq
2^{p-1} d_Y^{\,p}(T_h(x_i), T(x)) + 2^{p-1} d_Y^{\,p}(T(x), T(x_i)) \\
&\leq
2^{p-1} d_Y^{\,p}(T_h(x_i), T(x)) + 2^{p-1} \operatorname{osc}_T^{\,p}(E_i).
\end{align*}
Integrating the previous inequality over $E_i$ with respect to $x$ yields
\[
\mu_i \mspace{2mu} d_Y^{\,p}(T_h(x_i), T(x_i))
\leq
2^{p-1} \int_{E_i}{d_Y^{\,p}(T_h(x_i), T(x))\,d\mu(x)}
    + 2^{p-1} \mu_i \operatorname{osc}_T^{\,p}(E_i)
\]
and summing over $i \in I$ we get
\[
\operatorname{disc}_p^{\mspace{1mu}p}(T_h, T)
\leq
2^{p-1}  \mspace{2mu}  d_p^{\,p}(T_h, T) +
2^{p-1} \mspace{0mu} \left(
\sum\nolimits_{i\in I}{
    \mspace{2mu} \mu_i \mspace{1mu} \operatorname{osc}_T^{\,p}(E_i)
} \right) .
\]
Then,
\eqref{eq_estimate_disc_p_T_h_T}
follows by taking the $p$-th root on both sides of the
previous inequality and observing that
$(x+y)^{1/p}\leq x^{1/p} + y^{1/p}$ and $2^{(p-1)/p} \leq 2$ hold for $p\geq 1$.
\end{proof}

\begin{proof}[Proof of
\emph{\ref*{corollary_equivalence_discrete_continuous_item_mu}
$\mspace{-8mu}\implies\mspace{-8mu}$
\ref*{corollary_equivalence_discrete_continuous_item_disc}}
in \Cref{corollary_equivalence_discrete_continuous}]

Fix an optimal transport map $T_*:X\to Y$ for the compact problem
$(X,Y,\mu,\nu,c)$ and assume $\mu[D(T_*)]=0$. Consider an approximating sequence
of projections $(T_k)_{k\in\mathbb{N}}$ and fix $1\leq p <\infty$.
Now we take an arbitrary $\varepsilon>0$ and complete the proof by showing that
$\limsup_{k}{\;\operatorname{disc}_p(T_k, T_*)} \leq 2\varepsilon$.

Each $T_k$ is constructed from an $h_k$-partition $\mathcal{C}^k_X$ for $X$, say
\[
\mathcal{C}^k_X = \big\{ \big( E_i^k, x_i^k \big) \big\}_{i\in I_k}
\text{ with }\;
\mu\big[E_i^k\big] = \mu_i^k
\text{ for all $i\in I_k$.}
\]
Since $\mu[D(T_*)]=0$, apply
\Cref{lemma_estimation_of_oscillation_for_sufficiently_fine_partitions}
with $T=T_*$ to obtain $h_0$.
Take $k_0$ such that $h_k\leq h_0$ for $k\geq k_0$, which is possible
since $h_k \longconverges[k] 0$. Then,
\eqref{eq_lemma_estimation_of_oscillation_for_sufficiently_fine_partitions}
turns into
\[
\left(
\sum\nolimits_{i\in I_k}{
    \mspace{2mu} \mu_i^k \mspace{1mu}
        \operatorname{osc}_{T_*}^{\,p}\big(E^k_i\big)
} \right)^{1/p}
\leq \varepsilon
\;\;\;\text{ for all $k\geq k_0$.}
\]
For $k\geq k_0$, apply
\Cref{lemma_disc_p_controlled_by_dp_and_sum_of_weighted_oscillations}
with $T_h=T_k$ and $T=T_*$ to conclude
\[
\operatorname{disc}_p(T_k, T_*)
\leq
2 \mspace{2mu} d_p(T_k,T_*) + 2 \mspace{0mu} \left(
    \sum\nolimits_{i\in I_k}{
        \mspace{2mu} \mu_i^k \mspace{1mu}
            \operatorname{osc}_{T_*}^{\,p} \big(E^k_i\big)
    } \right)^{1/p}
\leq 2 \mspace{2mu} d_p(T_k,T_*) + 2 \varepsilon.
\]
Since $d_p(T_k,T_*)\longconverges[k] 0$ by
\Cref{introtheorem_conv_maps},
the previous inequality implies
\[
\limsup_{k}{\;\operatorname{disc}_p(T_k, T_*)}
\;\leq\;
\limsup_{k}{\; 2 \mspace{2mu} d_p(T_k,T_*) + 2 \varepsilon}
\;=\;
2 \varepsilon . \qedhere
\]
\end{proof}

\bigsubsection{Proof of Corollary B : Necessity}

\begin{lemma}
\label{lemma_balls_in_partition}
Consider a metric space $(X,d_X)$,
a nonempty compact set $K \subset X$ and $h>0$.
Then, there are points $x_1, \dots, x_n \in K$ and
disjoint Borel sets $E_1, \dots,  E_n \subset X$ with diameters
at most $h$ satisfying the following properties:
\begin{enumerate}[%
label=\roman*), ref=\emph{\roman*)}, topsep=0pt, itemsep=5pt, leftmargin=35pt]
\item \label{lemma_balls_in_partition_item_open_neighborhood}
The set $U = \bigsqcup_{i}{E_i}$ is an open neighborhood of $K$.
\item \label{lemma_balls_in_partition_item_radius_r}
There is $r>0$ such that $B_r(x_i) \subset E_i$ for all $i\in I$.
\end{enumerate}
\end{lemma}

\begin{proof}
Take $\rho = h/2$ and consider a finite, minimal cover of $K$ by
open balls of radius $\rho$, say $\{B_1, \dots , B_n\}$.
By minimality, there is $x_i\in B_i \cap K$ with
$x_i \centernot \in B_j$ for all $j\neq i$.
As in the proof of
\Cref{lemma_existence_of_h_partitions},
construct disjoint Borel sets
$\tilde{E_1}, \dots, \tilde{E_n}$ with
\[
\textstyle U \defeq
    \bigcup\nolimits_{i}{B_i}  = \bigsqcup\nolimits_{i}{\tilde{E}_i}
\;\text{ and }\;
x_i \in \tilde{E_i} \subset  B_i
\;\text{ for all $1\leq i \leq n$.}
\]
We are almost done, but the sets $\tilde{E_1}, \dots, \tilde{E_n}$
fail to satisfy
\ref{lemma_balls_in_partition_item_radius_r},
so we tweak them by adding a small ball $B_r(x_i)$
to each $\tilde{E_i}$ while preserving disjointness.
First note that there is $r_i>0$ with $B_{r_i}(x_i) \subset B_i$, so we can
take $r_0 = \min_{1\leq i\leq n}{r_i}>0$.
Now take $\delta = \min_{i\neq j}{d_X(x_i,x_j)}$ and
$r = \min{\{ \delta/2, r_0 \}} > 0$
to define
\[
V = B_r(x_1) \sqcup \dots \sqcup B_r(x_n)
\;\text{ and }\;
E_i = (\tilde{E_i} \backslash V) \cup B_r(x_i)
\;\text{ for all $1\leq i \leq n$. }\;
\]
Since $E_i \subset B_i$ then $\diam{E_i} \leq 2\rho = h$.
The sets $E_1, \dots, E_n$ remain disjoint
since $B_r(x_i)\cap B_r(x_j)=\emptyset$ for $i\neq j$.
Finally, notice that
$U = \bigsqcup\nolimits_{i}{E_i}$ since $V \subset U$.
\end{proof}

\begin{lemma}
\label{lemma_trick_build_counterexample_for_necessity_in_corollary_B}
Let $(X,Y,\mu,\nu,c)$ be a compact OT problem
and let $T:X\to Y$ be an arbitrary map with $\mu[D(T)]>0$.
Then, there is a constant $\delta>0$ depending only on $T$ with the following
property:
for every $h>0$, there are optimal $h$-plans $\pi^1_h$, $\pi^2_h$
and projection maps
$T_h^1 \in \mathcal{T}[\pi_h^1]$, $T_h^2 \in \mathcal{T}[\pi_h^2]$
such that
\begin{equation}
\label{eq_lemma_trick_counterexample_for_necessity_in_corollary_B}
\delta \leq \operatorname{disc}_p(T_h^1,T) + \operatorname{disc}_p(T_h^2,T)
+ d_p(T_h^1,T_h^2)
\end{equation}
for every $1\leq p < \infty$.
\end{lemma}

\begin{proof}
Since $\mu[D(T)]>0$ and $D(T) = \bigcup\nolimits_{n\in\mathbb{N}}{D_{1/n}(T)}$,
the continuity of $\mu$ implies the existence of
$\alpha>0$ with $\mu[D_{\alpha}(T)] = m>0$.
Let $\beta = \alpha/2$ and take
\begin{equation}
\label{eq_lemma_trick_definition_delta}
0 < \delta
\defeq \inf_{1 \leq p < \infty}{m^{1/p} \beta}
= \inf_{1 \leq p < \infty}{\mu[D_{\alpha}(T)]^{1/p} \beta}.
\end{equation}
Now we show that $\delta$ has the desired property. To this end, we fix $h>0$
and construct a pair of projection maps $T_h^1$, $T_h^2$ satisfying
\eqref{eq_lemma_trick_counterexample_for_necessity_in_corollary_B}
for all $1\leq p < \infty$.

Since $K \defeq D_{\alpha}(T)$ is compact,
\Cref{lemma_balls_in_partition}
gives points $x_1, \dots, x_n$ in $D_{\alpha}(T)$
and disjoint Borel sets $E_1, \dots, E_n$
with $\diam{E_i}\leq h$
covering $D_{\alpha}(T)$ and satisfying
\ref{lemma_balls_in_partition_item_open_neighborhood}
and
\ref{lemma_balls_in_partition_item_radius_r}.
Then there is $r>0$ such that $B_r(x_i) \subset E_i$
and we get
\[
\beta < \alpha
\leq \operatorname{osc}_{T}(B_r(x_i))
\leq \operatorname{osc}_{T}(E_i)
= \sup{ \big \{ d_Y(T(x^1), T(x^2)) \,:\, x^1,x^2 \in E_i \big\} }.
\]
It follows that, for $1\leq i\leq n$,
we can choose $x_i^1,x_i^2 \in E_i$ with
\begin{equation}
\label{eq_lemma_trick_lower_bound_for_T_evaluations}
d_Y(T(x_i^1), T(x_i^2)) \geq \beta.
\end{equation}

Notice that
$\{(E_i,x_i^1)\}_i$ and $\{(E_i,x_i^2)\}_i$
are $h$-partitions for the open set $U=\bigsqcup_i{E_i}$
and, since $X\backslash U$ is compact,
\Cref{lemma_existence_of_h_partitions}
ensures the existence of an $h$-partition
$\mathcal{C}_{X\backslash U}$ for $X\backslash U$.
We combine them to build two $h$-partitions for $X$:
\[
\mathcal{C}^{\, 1}_{X} =
\big\{\big(E_i,x_i^1\big)\big\}_{i=1}^{n} \cup \, \mathcal{C}_{X\backslash U}
\;\;\text{ and }\;\;
\mathcal{C}^{\, 2}_{X} =
\big\{\big(E_i,x_i^2\big)\big\}_{i=1}^{n} \cup \, \mathcal{C}_{X\backslash U}.
\]
Considering some $h$-partition $\mathcal{C}_{Y}$
for $Y$, we produce two optimal $h$-plans:
$\pi_h^1$ from $\mathcal{C}^{\, 1}_{X}$, $\mathcal{C}_{Y}$ and
$\pi_h^2$ from $\mathcal{C}^{\, 2}_{X}$, $\mathcal{C}_{Y}$.
Finally, we consider projections
$T_h^1 \in \mathcal{T}[\pi_h^1]$, $T_h^2 \in \mathcal{T}[\pi_h^2]$
and show that
\eqref{eq_lemma_trick_counterexample_for_necessity_in_corollary_B}
holds.

Fix $1 \leq p < \infty$ and notice that
\[ \textstyle
m = \mu[D_{\alpha}(T)] \leq \mu[U]
= \mu \big[ \bigsqcup_{i=1}^{n}{E_i} \big]
= \sum\nolimits_{i=1}^{n}{\mu[E_i]}
= \sum\nolimits_{i=1}^{n}{\mu_i},
\]
which together with
\eqref{eq_lemma_trick_lower_bound_for_T_evaluations}
implies
\begin{equation}
\label{eq_lemma_trick_bound_before_triangle_inequality}
\delta
\stackrel[\eqref{eq_lemma_trick_definition_delta}]{}{\leq}
m^{1/p} \beta
\leq \left( \sum\nolimits_{i=1}^{n}{\mu_i \mspace{2mu} \beta^p} \right)^{1/p}
\leq \left( \sum\nolimits_{i=1}^{n}{
    \mu_i \mspace{2mu} d^{\,p}_Y(T(x_i^1), T(x_i^2))
} \right)^{1/p}.
\end{equation}
Triangle inequality for $d_Y$ yields
$d_Y(T(x_i^1) , T(x_i^2)) \leq A_i^1 + A_{i}^{12} + A_i^2$,
where
\begin{equation*}
\textstyle
A_i^1 = d_Y(T(x_i^1), T_h^1(x_i^1))
\,,\;
A_i^{12} = d_Y(T_h^1(x_i^1),T_h^2(x_i^2))
\,,\;
A_i^2 = d_Y(T_h^2(x_i^2),T(x_i^2)).
\end{equation*}
Applying this bound to the RHS of
\eqref{eq_lemma_trick_bound_before_triangle_inequality}
results in
\[ \textstyle
\delta
\leq \big( \sum\nolimits_{i=1}^{n}{
        \mu_i \mspace{2mu} d^{\,p}_Y(T(x_i^1), T(x_i^2))
} \big)^{1/p}
\leq \big( \sum\nolimits_{i=1}^{n}{
    \mu_i \mspace{2mu} \big ( A_i^{1} + A_i^{12} + A_i^{2} \big)^p
} \big)^{1/p},
\]
and then the triangle inequality
for the seminorm $(\sum_{i}{\mu_i |t_i|^p})^{1/p}$ yields
\begin{equation}
\label{eq_lemma_trick_bound_with_three_sums}
\textstyle
\delta
\leq
\left( \sum\nolimits_{i=1}^{n}{
    \mu_i \mspace{2mu} \big( A_i^{1} \big)^p
}\right)^{1/p}
+
\left( \sum\nolimits_{i=1}^{n}{
    \mu_i \mspace{2mu} \big ( A_i^{12} \big)^p
}\right)^{1/p}
+
\left( \sum\nolimits_{i=1}^{n}{
    \mu_i \mspace{2mu} \big ( A_i^{2} \big)^p
}\right)^{1/p}.
\end{equation}

To finish the proof note first that for $k=1,2$ we have
\begin{equation}
\label{eq_bound_A_i_k}
\left( \sum\nolimits_{i=1}^{n}{
    \mu_i \mspace{2mu} \big( A_i^{k} \big)^p
}\right)^{1/p}
=
\left( \sum\nolimits_{i=1}^{n}{
    \mu_i \mspace{2mu} d^{\,p}_Y(T_h^k(x_i^k), T(x_i^k))
}\right)^{1/p}
\stackrel[(*)]{}{\leq}
\operatorname{disc}_p(T_h^k,T).
\end{equation}
Notice that the terms in $\operatorname{disc}_p(T_h^k,T)$ coming from the
pointed sets in $\mathcal{C}_{X\backslash U}$ are missing in the LHS of
\eqref{eq_bound_A_i_k},
and this is why equality fails in $(*)$.
A key observation here is that $\mu[E_i]=\mu_i$ for both
$\mathcal{C}^{\, 1}_{X}$ and $\mathcal{C}^{\, 2}_{X}$,
so the expressions of $\operatorname{disc}_p$ for both $h$-partitions differ
only in the points where the functions are evaluated, not in their weights.
On the other hand, note that for $1 \leq i \leq n$ we have
\[
\mu_i \mspace{2mu} d^{\,p}_Y(T_h^1(x_i^1),T_h^2(x_i^2))
=  \int_{E_i}{d^{\,p}_Y(T_h^1(x),T_h^2(x)) \, d\mu(x)}
\]
since $T_h^1$ and $T_h^2$ are constant on each $E_i$ by definition.
Then,
\begin{equation}
\label{eq_bound_A_i_12}
\left( \sum\nolimits_{i=1}^{n}{
    \mu_i \mspace{2mu} \big ( A_i^{12} \big)^p
}\right)^{1/p}
=
\left( \sum\nolimits_{i=1}^{n}{
    \mu_i \mspace{2mu} d^{\,p}_Y(T_h^1(x_i^1),T_h^2(x_i^2))
}\right)^{1/p}
\leq  d_p(T_h^1,T_h^2).
\end{equation}
Plugging the bounds
\eqref{eq_bound_A_i_k}
and
\eqref{eq_bound_A_i_12}
into
\eqref{eq_lemma_trick_bound_with_three_sums}
yields
\eqref{eq_lemma_trick_counterexample_for_necessity_in_corollary_B}.
\end{proof}

\begin{proof}[Proof of
\emph{\ref*{corollary_equivalence_discrete_continuous_item_disc}
$\mspace{-8mu}\implies\mspace{-8mu}$
\ref*{corollary_equivalence_discrete_continuous_item_mu}}
in \Cref{corollary_equivalence_discrete_continuous}]
We prove the contrapositive statement, so we start by assuming the negation
of \ref{corollary_equivalence_discrete_continuous_item_mu} to get
$\mu[D(T_*)]>0$. Now we fix a constant $\delta>0$
by applying
\Cref{lemma_trick_build_counterexample_for_necessity_in_corollary_B}
with $T=T_*$ and also choose some $1 \leq p < \infty$.
For each $k\in \mathbb{N}$ we let $h_k=k^{-1}$ and then we take
optimal $h_k$-plans $\pi_k^1,\pi_k^2$ and projection maps
$T_k^1 \in \mathcal{T}[\pi_k^1]$, $T_k^2 \in \mathcal{T}[\pi_k^2]$
such that
\begin{equation}
\label{eq_proof_necessary_condition_for_discrete_convergence_ineq_from_lemma}
\delta \leq \operatorname{disc}_p(T_k^1,T_*) + \operatorname{disc}_p(T_k^2,T_*)
+ d_p(T_k^1,T_k^2),
\end{equation}
which is possible due to the property enjoyed by $\delta$
in
\Cref{lemma_trick_build_counterexample_for_necessity_in_corollary_B}.

Note that both $(T^1_k)^{\vphantom{1}}_k$ and $(T^2_k)^{\vphantom{2}}_k$ are approximating sequences
of projection maps for $(X,Y,\mu,\nu,c)$,
so triangle inequality for $d_p$ and
\Cref{introtheorem_conv_maps}
together imply
\begin{equation}
\label{eq_proof_necessary_condition_for_discrete_convergence_ineq_T_1_T_2}
\limsup_{k}{\; d_p(T_k^1,T_k^2)}
\leq
\limsup_{k}{\; \big[ d_p(T_k^1,T_*) + d_p(T_*,T_k^2) \big]} = 0.
\end{equation}
Then, taking $\limsup_{k}$ in
\eqref{eq_proof_necessary_condition_for_discrete_convergence_ineq_from_lemma}
and using
\eqref{eq_proof_necessary_condition_for_discrete_convergence_ineq_T_1_T_2},
we obtain
\[
0 < \delta \leq
\limsup_{k}{\; \big[
    \operatorname{disc}_p(T_k^1,T_*) + \operatorname{disc}_p(T_k^2,T_*)
\big] }.
\]
Therefore, either $\operatorname{disc}_p(T_k^1,T_*) \centernot\longconverges 0$
or $\operatorname{disc}_p(T_k^2,T_*) \centernot\longconverges 0$.
This shows that
$\ref*{corollary_equivalence_discrete_continuous_item_disc}$
in
\Cref{corollary_equivalence_discrete_continuous}
fails and the proof is complete.
\end{proof}

\renewcommand\thesection{\Alph{section}}
\setcounter{section}{24}
\setcounter{theorem}{0}

\section*{Appendix: Stability of Optimal Transport}
\label{section_stability_analysis_OT}

The validity of our three main results, i.e.\
Theorems
\ref{introtheorem_conv_values_plans} and
\ref{introtheorem_conv_maps}
and
\Cref{introcorollary_fully_discrete_version_convergence_maps},
is a consequence of the good stability properties enjoyed by
compact OT problems. %
We summarize these properties in
\Cref{proposition_stability_properties_OT},
whose proof is a variation on that of a well known stability result
for transport maps
\cite[Corollary 5.23]{villani_old_new}.
Both proofs rely on Prokhorov's theorem
\cite[Theorems 5.1 and 5.2]{billingsley_convergence_prob_measures}
and on an abstract version of Lusin's theorem
\cite[Theorem 2.3.5]{federer_GMT},
which we state below for completeness.

\begin{theorem}[Prokhorov]
\label{prokhorov_theorem}
Let $(Z,d_Z)$ be a Polish metric space
and let $\mathcal{F}$ be a family of measures in $\mathcal{P}(Z)$.
Then, $\mathcal{F}$ is sequentially precompact with respect to the topology of
weak convergence of measures in $\mathcal{P}(Z)$
iff $\mathcal{F}$ is \emph{tight}, which means
that for every $\varepsilon>0$ there is a compact set $K\subset Z$ with
$\sigma[Z\backslash K]<\varepsilon$ for all $\sigma \in \mathcal{F}$.
\end{theorem}

\begin{theorem}[Lusin]
\label{lusin_theorem}
Let $X$ be a locally compact metric space and let
$Y$ be a separable metric space.
Consider $\mu \in \mathcal{P}(X)$ and a Borel map
$T:X \to Y$.
Then, for every $E \in \mathcal{B}(X)$ with $\mu[E]<\infty$ and every
$\varepsilon>0$, there is a closed set $K\subset E$ such that
$\mu[E\backslash K]<\varepsilon$ and $T|_K$ is continuous.
\end{theorem}

\begin{proposition}
\label{proposition_stability_properties_OT}
Let $(X,Y,\mu,\nu,c)$ be a compact OT problem with optimal value
$\Koperator[*]$ for its Kantorovich problem.
Let $(\pi_k)_k \subset \mathcal{P}(X\times Y)$ be such that
$\Koperator[\pi_k] \longconverges[k] \Koperator[*]$
and assume
$\pi_k \in \mathcal{A}(\mu_k,\nu_k)$ for all $k$
with
$\mu_k, \nu_k \longweaklyconverges[k] \mu,\nu$.
It follows that:
\begin{enumerate}[%
label=\roman*), ref=\emph{\roman*)}, topsep=0pt, itemsep=4pt, leftmargin=25pt]
\item \label{proposition_stability_properties_OT_item_kantorovich}
    If $\minsetkantorovich$ is the set of minimizers for
    \eqref{kantorovich_problem}, that is
    \[
    \minsetkantorovich = \big\{\mspace{2mu}
        \pi \in \mathcal{A}(\mu,\nu)
            \mspace{2mu}:\mspace{2mu} \Koperator[\pi] = \Koperator[*]
    \mspace{2mu}\big\},
    \]
    then $W_p(\pi_k, \minsetkantorovich) \longconverges[k] 0$
    for all $1\leq p < \infty$. In particular,
    if \ref{kantorovich_uniqueness_hypothesis} holds
    with optimal transport plan $\pi_*$,
    then $W_p(\pi_k, \pi_*) \longconverges[k] 0$
    and $\pi_k \longweaklyconverges[k] \pi_*$.
\item \label{proposition_stability_properties_OT_item_monge}
    If $\mu_k=\mu$ for all $k$ and \ref{strong_uniqueness_hypothesis} holds
    with optimal transport map $T_*$, then
        \begin{equation}\label{eq_proposition_stability_properties_OT}
        \pi_k \big[ \{ (x,y)\in X\times Y: \;d_Y(y,T_*(x)) \geq \delta \} \big]
            \longconverges[k] 0
        \end{equation}
    for every $\delta>0$.
\end{enumerate}
\end{proposition}

\begin{remark}
\label{remark_extract_subsequences}
It follows from $\mu_k,\nu_k \longweaklyconverges[k] \mu,\nu$ and
Prokhorov's \Cref{prokhorov_theorem} that
$\{\mu_k\}_k$ and $\{\nu_k\}_k$ are tight, which together with the
condition $\pi_k \in \mathcal{A}(\mu_k,\nu_k)$ for all $k$ easily implies
the tightness of the family $\{\pi_k\}_k$
(see e.g.\ \cite[Lemma 4.4]{villani_old_new} for a proof).
Applying Prokhorov's theorem again we get a
weakly convergent subsequence $(\pi_{k_j})_j$
and its limit necessarily lies in $\mathcal{A}(\mu,\nu)$.
A quick proof of this last claim follows from the fact that,
for arbitrary $\pi \in \mathcal{P}(X \times Y)$,
the condition $\pi \in \mathcal{A}(\mu,\nu)$ is equivalent to
the validity of equality
\begin{equation}
\label{eq_equivalent_condition_marginals}
\int_{X}{\phi \,d\mu} + \int_{Y}{\psi \,d\nu}
= \int_{X\times Y}{ \big[ \phi(x) + \psi(y) \big] \,d\pi(x,y)}
\end{equation}
for all $\phi\in C_b(X)$ and $\psi\in C_b(Y)$.
Thus, rewriting
\eqref{eq_equivalent_condition_marginals}
for $\mu_{k_j}$, $\nu_{k_j}$, $\pi_{k_j}$ and passing to the limit
immediately yields $\lim_{j}{\pi_{k_j}} \in \mathcal{A}(\mu,\nu)$.
\end{remark}

\begin{proof}[Proof of \Cref*{proposition_stability_properties_OT}]
Assume, on the contrary, that
\ref{proposition_stability_properties_OT_item_kantorovich}
fails to hold. Then, fixing $p$ and passing to a subsequence if necessary,
we may assume
$W_p(\pi_k, \minsetkantorovich) \geq \varepsilon > 0$
for every $k$.
Working as in
\Cref{remark_extract_subsequences}
extract a subsequence
$(\pi_{k_j})_j$ converging weakly to $\pi \in \mathcal{A}(\mu,\nu)$.
It turns out that $\pi\in\minsetkantorovich$
since $c\in C_b(X\times Y)$ implies
\begin{equation*}
\Koperator[\pi]
= \int_{X\times Y}{c \,d\pi}
\stackrel[]{\eqref{eq_intro_weak_convergence_definition}}{=}
    \lim_{j}{\int_{X\times Y}{c \,d\pi_{k_j}}}
= \lim_{j}{\,\Koperator[\pi_{k_j}]}
= \lim_{k}{\,\Koperator[\pi_{k}]} = \Koperator[*].
\end{equation*}
Finally, given that $W_p$ metrizes the weak convergence
of measures on compact spaces \cite[Theorem 6.9]{villani_old_new},
we get $W_p(\pi_{k_j}, \pi) \longconverges[j] 0$ and the contradiction follows
\[
0 < \varepsilon
\leq
W_p(\pi_{k_j}, \minsetkantorovich)
=
\inf_{\widetilde{\pi} \in \minsetkantorovich}{ W_p(\pi_{k_j}, \widetilde{\pi})}
\leq
W_p(\pi_{k_j}, \pi) \longconverges[j] 0.
\]
Note that the last part of
\ref{proposition_stability_properties_OT_item_kantorovich}
is immediate since
\ref{kantorovich_uniqueness_hypothesis} implies
$\minsetkantorovich=\{\pi_*\}$.

For \ref{proposition_stability_properties_OT_item_monge},
assume \ref{strong_uniqueness_hypothesis} and note that
$\pi_k \longweaklyconverges[k] \pi_* = \pi_{T_*}$.
Fix $\delta>0$ and for arbitrary $\varepsilon>0$ apply Lusin's Theorem
\ref{lusin_theorem} to
obtain a closed set $K\subset X$ with
$\mu[X \backslash K] \leq \varepsilon$ and $T_*|_K$ continuous.
Observe that the set
\begin{equation*}
B = \big\{ (x,y) \in K\times Y : d_Y(y,T_*(x)) \geq \delta \big\}
\end{equation*}
is closed in $X \times Y$. Also, we get
$\pi_* [B] = \pi_{T_*} [B] = 0$
as $\pi_*$ is concentrated on the graph of $T_*$.
Since $B$ is closed,
the weak convergence of $(\pi_k)_k$ to $\pi_*$ implies
\begin{align*}
0 &= \,\pi_* [B]
\geq \,\limsup_{k}{ \; \pi_k[B]} \\
&= \,\limsup_{k}{ \;
    \pi_k \big[ \{  (x,y) \in K\times Y : d_Y(y,T_*(x)) \geq \delta  \} \big]
} \\
&\geq \,\limsup_{k}{
    \; \pi_k \big[ \{  (x,y) \in X\times Y : d_Y(y,T_*(x)) \geq \delta  \} \big]
} - \varepsilon,
\end{align*}
where the last inequality follows from the fact that
$\pi_k \in \mathcal{A}(\mu,\nu_k)$ and $\mu[X \backslash K] \leq \varepsilon$.
Finally, taking $\varepsilon \to 0$ yields
\eqref{eq_proposition_stability_properties_OT}.
\end{proof}

\vspace{6mm}

\centerline{ \scshape Acknowledgments }
\pdfbookmark[1]{Acknowledgments}{Acknowledgments}

This work was done by the author as part of his doctoral thesis
while partially supported by a CONICET doctoral fellowship.
The author was also supported by ANPCyT under grant PICT-2018-03017,
by Universidad de Buenos Aires under grant 20020160100002BA
and by CONICET under grant PIP 11220200100175CO.

\renewcommand{\bibname}{References}

\end{document}